\documentclass{article}
\usepackage{bm}
\usepackage{tikz,mathpazo}
\usepackage{pgf}
\usepackage[top=2.5cm, bottom=2.5cm, left=3cm, right=3cm]{geometry}   
\usepackage{indentfirst}
\usepackage{graphicx}
\usepackage{graphics}
\usepackage[toc,page,title,titletoc,header]{appendix}
\usepackage{bm}
\usepackage{listings}  
\usepackage{amsmath}
\usepackage{setspace} 
\usepackage{indentfirst}
\usepackage{caption}
\usepackage{multirow} 
\usepackage{lipsum,multicol}
\usepackage{pdfpages}
\usepackage{float}
\usepackage{amsthm}
\usepackage{pdfpages}
\usepackage{url}
\usepackage{colortbl}
\usepackage{subfigure}
\usepackage{epsfig}
\usepackage{epstopdf}
\usetikzlibrary{shapes.geometric, arrows}
\usepackage{fancyhdr}
\usepackage{abstract}
\usepackage{tikz,mathpazo}
\usetikzlibrary{shapes.geometric, arrows}
\usepackage{amssymb}
\usepackage{latexsym}
\usepackage{verbatim}
\usepackage[numbers]{natbib} 
\usepackage{booktabs}

\allowdisplaybreaks[4]

\newcommand{\norm}[1]{\left\Vert #1\right\Vert}
\newcommand{\bb}[1]{\mathbb{#1}}

\newcommand{\ca}[1]{\mathcal{#1}}
\newcommand{\tr}[0]{\mathrm{tr}}

\newcommand{\Diag}[0]{\mathrm{Diag}}

\newcommand{\M}[0]{\mathcal{M}}

\newcommand{\tp}{^\top}

\newcommand{\A}{\ca{A}}

\newcommand{\Tx}{{\ca{T}_x}}
\newcommand{\Nx}{\ca{N}_x}

\newcommand{\xk}{{x_{k} }}
\newcommand{\yk}{{y_{k} }}
\newcommand{\xkp}{{x_{k+1} }}

\newcommand{\dk}{{d_{k} }}

\newcommand{\ykp}{{y_{k+1} }}

\newcommand{\Jc}{{{J}_c}}
\newcommand{\Ja}{{{J}_{\A}}}

\newcommand{\DJa}{ {\ca{D}_{{J}_\A}} }
\newcommand{\DJc}{{\ca{D}_{{J}_c}}}

\newcommand{\grad}{{\mathit{grad}\,}}
\newcommand{\hess}{{\mathit{hess}\,}}

\newcommand{\Lf}{ {M_{x,f}} }

\newcommand{\Lg}{ {L_{x,g}} }

\newcommand{\Lsc}{ \sigma_{x,c} }
\newcommand{\tLsc}{ {\tilde{\sigma}_{x_0,c}} }

\newcommand{\Mc}{ {M_{x,c}} }
\newcommand{\tMc}{ {\tilde{M}_{x_0,c}} }

\newcommand{\Ma}{ M_{x,A} } 
\newcommand{\tMa}{ {\tilde{M}_{x_0,A}} } 

\newcommand{\Lc}{ {L_{x,c}} }
\newcommand{\tLc}{ {\tilde{L}_{x_0,c}} }

\newcommand{\La}{ {L_{x,A}} }
\newcommand{\tLa}{ {\tilde{L}_{x_0,A}} }

\newcommand{\Lac}{ {L_{x, b}} }

\newcommand{\Omegax}[1]{ {\Omega_{#1}} }

\newcommand{\HOmegax}[1]{ \hat{\Omega}_{#1} }

\newcommand{\BOmegax}[1]{ {\bar{\Omega}_{#1}} }

\newcommand{\y}{{y}}

\newtheorem{theo}{Theorem}[section]
\newtheorem{lem}[theo]{Lemma}
\newtheorem{prop}[theo]{Proposition}

\newtheorem{coro}[theo]{Corollary}

\newtheorem{defin}[theo]{Definition}
\newtheorem{rmk}[theo]{Remark}
\newtheorem{assumpt}[theo]{Assumption}

\usepackage[figuresright]{rotating}
\usepackage{pdflscape}

\DeclareMathOperator*{\argmin}{arg\,min}

\usepackage{caption}
\usepackage{algorithm}
\usepackage{algpseudocode}
\usepackage{longtable}
\usepackage{appendix}

\numberwithin{equation}{section}

\title{Dissolving Constraints for Riemannian Optimization}
\author{Nachuan Xiao
	\thanks{The Institute of Operations Research and Analytics, National University of Singapore, Singapore. (xnc@lsec.cc.ac.cn). The research of this author is supported by the Ministry of Education, Singapore, under its Academic Research Fund Tier 3 grant call (MOE-2019-T3-1-010).} ~
	Xin Liu\thanks{State Key Laboratory of Scientific and Engineering Computing, Academy of Mathematics and Systems Science, Chinese Academy of Sciences, and University of Chinese Academy of Sciences, China (liuxin@lsec.cc.ac.cn). Research is supported in part by the National Natural Science Foundation of China (No. 12125108, 11971466, 11991021), Key Research Program of Frontier Sciences, Chinese Academy of Sciences (No. ZDBS-LY-7022).} ~
	and Kim-Chuan Toh \thanks{Department of Mathematics, and Institute of Operations Research and Analytics, National University of Singapore, Singapore 119076 (mattohkc@nus.edu.sg). The research of this author is supported by the Ministry of Education, Singapore, under its Academic Research Fund Tier 3 grant call (MOE-2019-T3-1-010).}}

\begin{document}
	\maketitle
	
	\begin{abstract}
		In this paper, we consider optimization problems over closed embedded submanifolds of $\bb{R}^n$, which are defined by the constraints $c(x) = 0$. We propose a class of constraint dissolving approaches for these Riemannian optimization problems. In these proposed approaches, solving a Riemannian optimization problem is transferred into  the unconstrained minimization of a constraint dissolving function named \ref{Prob_Pen}. 
	Different from existing exact penalty functions, the exact gradient and Hessian of \ref{Prob_Pen} are easy to compute. 
	We study the theoretical properties of \ref{Prob_Pen} and prove that the original problem and \ref{Prob_Pen} have the same first-order and second-order stationary points, local minimizers, and {\L}ojasiewicz exponents in a neighborhood of the feasible region. 
Remarkably, the convergence properties of our proposed constraint dissolving approaches can be directly inherited from 
the existing rich results in unconstrained optimization. Therefore, the proposed constraint dissolving approaches build up short cuts from unconstrained optimization to Riemannian optimization. Several illustrative examples further demonstrate the potential of our proposed constraint dissolving approaches.
	\end{abstract}

	\section{Introduction}
	\subsection{Problem description}
	In this paper, we consider the following constrained optimization problem
	\begin{equation}
		\label{Prob_Ori}
		\tag{OCP}
		\begin{aligned}
			\min_{x \in \bb{R}^n} \quad & f(x)\\
			\text{s.t.} \quad & c(x) = 0.
		\end{aligned}
	\end{equation}
	We denote the feasible region of \ref{Prob_Ori} by $\M := \{ x \in \bb{R}^n: c(x) = 0 \}$. In addition,
	the objective function  $f: \bb{R}^n \to \bb{R}$   and constraint mapping 
	$c: \bb{R}^n \to \bb{R}^p$ of \ref{Prob_Ori} satisfy the following assumptions.
	\begin{assumpt}{\bf Blank assumptions}
		\label{Assumption_1}
		
		\begin{enumerate}
			\item $\nabla f(x)$ 
			is locally Lipschitz continuous in $\bb{R}^n$;
			\item The transposed Jacobian of $c$, denoted as $\Jc(x) \in \bb{R}^{n \times p}$, 
			is locally Lipschitz continuous in $\bb{R}^{n}$;
			\item The linear independence constraint qualification (LICQ) holds for any $x \in \M$, i.e. 
			$\Jc(x) \in \bb{R}^{n\times p}$ has full column rank for any $x \in \M$. 
		\end{enumerate}
	\end{assumpt}
	
	{When $c$ is smooth in $\bb{R}^n$, the set $\M$ is a closed Riemannian submanifold embedded in $\bb{R}^n$. 
		Thus \ref{Prob_Ori} can be regarded as a smooth optimization problem over a class of embedded submanifolds of the vector space $\bb{R}^n$, where the Riemannian metric is fixed to be the Euclidean metric. 
		In fact, \ref{Prob_Ori} satisfying Assumption \ref{Assumption_1} covers a wide variety of practically interesting smooth optimization problems over closed Riemannian manifolds. Interested readers can refer to the books \cite{Absil2009optimization,boumal2020introduction} and a recent survey paper \cite{hu2020brief} for instances. 
	
	\subsection{Existing approaches}
	Due to the diffeomorphisms between the Euclidean space and the Riemannian manifold, various {\it unconstrained optimization approaches} (i.e., approaches for solving unconstrained nonconvex optimization) can be transferred to their corresponding {\it Riemannian optimization approaches} (i.e., the approaches for Riemannian optimization). In practice, \cite{Absil2009optimization} provides several well-recognized frameworks based on two basic materials in differential geometry: geodesics and parallel transports. The geodesics generalize the concept of straight lines from Euclidean spaces to Riemannian manifolds, but may be expensive to compute in most cases. To this end, \cite{Absil2009optimization} provides the concept of retractions as relaxations to geodesics, which makes it more affordable to update the iterates on a certain Riemannian manifold at the cost of 
	introducing approximation errors. 
	Besides,  parallel transports are  mappings that move a tangent vector from one tangent space to another under certain rules. Computing parallel transports is essential in computing the difference of two vectors from different tangent spaces. Specifically, computing parallel transports is necessary for algorithms that utilize  information in the past iterates to construct searching direction (e.g., quasi-Newton methods, nonlinear conjugate gradient methods, momentum accelerated methods).  For various Riemannian manifolds, computing the parallel transports amounts to solving differential equations,  which is generally unaffordable in practice \cite{Absil2009optimization}. To alleviate the computational cost, \cite{Absil2009optimization} proposes the concept of vector transports as approximations to parallel transports. As mentioned in \cite{qi2010riemannian}, computing vector transports is usually cheaper than parallel transports. With retractions and vector transports, many unconstrained optimization approaches have been extended to their Riemannian versions, including Riemannian gradient descent with line-search \cite{abrudan2008steepest,Absil2009optimization,wen2013feasible,wang2020multipliers}, Riemannian conjugate gradient methods \cite{abrudan2009conjugate,sato2016dai}, Riemannian accelerated gradient methods \cite{zhang2018r,zhang2018towards,siegel2019accelerated,criscitiello2020accelerated}, and Riemannian adaptive gradient methods \cite{becigneul2018riemannian}: see \cite{Absil2009optimization,boumal2020introduction} for instances.

	In recent years, there emerge an increasing number of approaches for solving unconstrained optimization problems,
	which have superior convergence properties, marvelous numerical behaviors, or both. 
	However, transferring an unconstrained optimization approach to its Riemannian versions requires some basic geometrical materials of the Riemannian manifold, including computing Riemannian gradients, retractions, and vector transports \cite{Absil2009optimization}. Determining those geometrical materials can be challenging for various Riemannian manifolds, see \cite{edelman1998geometry,nickel2018learning,gao2021riemannian} for instances. 
	Based on those geometrical materials,  transferring an unconstrained optimization approach to its Riemannian versions requires profound modifications, including replacing the computation of differentials by Riemannian differentials, introducing retractions to keep the iterates feasible, and employing vector transports to move vectors on the Riemannian manifold. As a result,  it is challenging to keep Riemannian optimization approaches updated with the advances in unconstrained nonconvex optimization.

		Furthermore, the convergence properties of many practically useful Riemannian optimization approaches cannot directly follow from existing results for unconstrained optimization. As shown in \cite{Absil2009optimization,boumal2020introduction}, the convergence properties of these Riemannian optimization approaches need to be carefully revisited when retractions and vector transports are employed. Extending the existing unconstrained optimization approaches to their Riemannian manifold versions and retaining the convergence guarantee
		are nontrivial and sometimes intractable. Therefore, it is quite natural to ask the following question.
	\begin{quote}
		\it
		Could unconstrained optimization approaches, together with their convergence properties, have a straightforward implementation for the constrained
		optimization problem \ref{Prob_Ori}?
	\end{quote}
	This question drives us to propose {\it constraint dissolving} approaches for \ref{Prob_Ori}, i.e., transferring \ref{Prob_Ori} into an unconstrained optimization problem  
	while keeping stationary points unchanged. Therefore, constraint dissolving approaches enable direct implementation of unconstrained optimization approaches to solve \ref{Prob_Ori}, while the convergence properties of those unconstrained optimization approaches
	are retained simultaneously.

	We should mention that there are two classes of existing approaches attempting to achieve a similar goal but using completely different philosophies.
	One of them is the well-known $\ell_1$  penalty function methods. The $\ell_1$  penalty function is known as 
	an exact penalty function. However, its nonsmooth penalty term, with the nonconvex manifold constraints inside, 
	usually leads to difficulties in developing efficient unconstrained optimization approaches \cite{nocedal2006numerical}. 
	The other class of methods are based on the augmented Lagrange penalty function \cite{hestenes1952methods,powell1969method}.
	The Lagrangian penalty function is an exact penalty function for \ref{Prob_Ori} when the Lagrange multipliers $\lambda$
	take their optimal values, which are certainly unknown in advance. Therefore,  classical augmented Lagrange methods
	need to solve the unconstrained penalty function subproblem with a fixed $\lambda$ and then update the multipliers in each iteration.
	These two hierarchical approaches are not as efficient as the existing Riemannian optimization approaches
	for solving \ref{Prob_Ori}.
	In particular, \cite{fletcher1970class} proposes a class of exact penalty functions named Fletcher's penalty function below
	\begin{equation}
		\label{Penalty_function_Fletcher} 
		\phi(x) := f(x) - u(x)\tp c(x) + \frac{\beta}{2}\norm{c(x)}^2.
	\end{equation}
	Here $u(x)$ is defined by the following linear least squares problem
	\begin{equation}
		\label{Fletcher_linear_sysytem}
		u(x) := \mathop{\arg\min}_{y \in \bb{R}^{p}}~ \frac{1}{2}\norm{ \sum_{i=1}^p y_i\nabla c_i(x) - \nabla f(x)  }^2. 
	\end{equation}

	Fletcher's penalty function and its variants \cite{di1986exact,zavala2014scalable,estrin2020implementing}  involve the first-order derivative of the original objective function $f$.
	Therefore, their differentiability depend on 
	the second-order differentiability of $f$. Moreover, 
	calculating the derivatives of these penalty functions requires
	the second-order derivative of $f$, which are not always available in
	practice. As a result, 
	existing approaches based on Fletcher's penalty function, such as 
	approximated steepest descent methods \cite{estrin2020implementing}, approximated Newton methods \cite{toint1981towards,steihaug1983conjugate,zavala2014scalable}, and approximated quasi-Newton methods \cite{estrin2020implementing}, are usually combined with certain approximation strategies to estimate higher-order derivatives. 
	Hence, various existing unconstrained optimization approaches are not compatible with the Fletcher's penalty function framework.
	
	For optimization problems on the Stiefel manifold, i.e. $\M= \ca{S}_{m,s}:= \{X \in \bb{R}^{m\times s}: X\tp X = I_s \}$, \cite{xiao2020class} presents an exact penalty model named PenC based on the explicit expression of the multipliers \cite{gao2019parallelizable}, which further yields efficient infeasible algorithms \cite{xiao2020class,xiao2020l21,hu2020anefficiency,xiao2021penalty}. However, PenC involves $\nabla f$ in its objective function as well. Therefore, those aforementioned limitations of Fletcher's penalty function approaches still remain unsolved.
	
	\subsection{Constraint dissolving function}
	Very recently, for optimization problems on the Stiefel manifold,  \cite{xiao2021solving} proves that under mild conditions, all the stationary 
	points of the following smooth penalty function
	are either its strict saddle points or are the first-order stationary points of the original problems:
	\begin{equation}
		\label{Eq_ExPen}
		\tag{ExPen}
		\psi(X) := f\left( X\left( \frac{3}{2} I_s - \frac{1}{2}X\tp X \right)\right) + \frac{\beta}{4} \norm{X\tp X - I_s}_{F}^2.
	\end{equation}
	
	As a result, various algorithms designed for unconstrained optimization can be directly applied to optimization problems on the Stiefel manifold, while the convergence properties are straightforwardly retained. 
	
	The constraint dissolving approaches for Riemannian optimization proposed in this paper is motivated by \eqref{Eq_ExPen}. To this end, we first introduce the following constraint dissolving operator $\A: \bb{R}^n \to \bb{R}^n$, which is a smooth mapping independent of $f$ and satisfies the following assumptions.
	\begin{assumpt}{\bf Blanket assumptions on $\A$}
		\label{Assumption_2}
		
		\begin{itemize}
			\item $\ca{A}$ is locally Lipschitz smooth in $\bb{R}^n$;
			\item $\A(x) = x$ holds for any $x \in \M$;
			\item The Jacobian of $c(\A(x))$ equals to $0$ for any $x \in \M$. That is, $\Ja(x) \Jc(x) = 0$  holds for any $x \in \M$ (notice that $\A(x) = x$ holds for any $x \in \M$), where $\Ja(x) \in \bb{R}^{n\times n}$ is the transposed Jacobian of $\A$ at $x$.
		\end{itemize}
	\end{assumpt}
	
	With the constraint dissolving operator, we propose the constraint dissolving function (\ref{Prob_Pen}) for \ref{Prob_Ori}:
	\begin{equation}
		\tag{CDF}
		\label{Prob_Pen}
		h(x) := f(\A(x)) + \frac{\beta}{2} \norm{c(x)}^2. 
	\end{equation} 
	Clearly, the first requirement in Assumption \ref{Assumption_2}  guarantees the smoothness of $h(x)$. Meanwhile the second requirement ensures that $f(x) = h(x)$ for any $x \in \M$. Finally, the last requirement implies that the first-order derivative of $c(\A(x))$
	vanishes at any feasible $x$. As a result, we can further conclude that $c(\A(x)) = \ca{O}(\norm{c(x)}^2)$ when $||c(x)||$ is sufficiently small, whose rigorous proof is
	presented later. The practical choices of $\A$ are introduced 
	in Section \ref{Section_Implementation}.

	\subsection{Contribution}
	In this paper, we propose a class of constraint dissolving approaches, which transfer \ref{Prob_Ori} into minimizing the corresponding constraint dissolving function (\ref{Prob_Pen}) in $\bb{R}^n$.
	We prove that \ref{Prob_Ori} and \ref{Prob_Pen} have the same first-order stationary points, second-order stationary points, and local minimizers in a neighborhood of $\M$. In addition, we show that  \ref{Prob_Pen} has the same {\L}ojasiewicz exponent as \ref{Prob_Ori} over $\M$. 
	Furthermore, we show that the exact gradient and Hessian of \ref{Prob_Pen} can be easily obtained.

	As \ref{Prob_Pen} requires a constraint dissolving operator $\A$ satisfying Assumption \ref{Assumption_2}, we present representative formulations of $\A$ for many well-known Riemannian manifolds. Moreover, we discuss how to choose $\A$ for general cases and demonstrate that the general formulation does not involve any information on the objective function, and hence \ref{Prob_Pen} is different from the Fletcher's penalty function. 
	More importantly, constructing \ref{Prob_Pen} is completely independent of any geometrical material of $\M$.  Since $\nabla h(x)$ is not necessarily restricted to the tangent space of $\M$ when $x$ is feasible,  \ref{Prob_Pen} waives all the calculations of geometrical materials of $\M$, including computing Riemannian gradients, retractions, and vector transports on $\M$.
	Therefore, we can develop various constraint dissolving approaches to solve optimization problems over a broad class of Riemannian manifolds, without prior knowledge of their geometrical properties.

	The convergence properties, including the global convergence and iteration complexity of applying any unconstrained optimization approach to \ref{Prob_Pen} can be guaranteed by a unified
	framework.
	We also present a representative example to demonstrate how to adopt \ref{Prob_Pen}
	and invoke the theoretical framework. These examples
	further highlight the significant advantages and great potentials of \ref{Prob_Pen}.

	\subsection{Organization}
	The rest of this paper is arranged as follows. In Section 2, we present some notations, definitions, and constants that are necessary for concise narrative in later parts of the paper. We establish the theoretical properties of \ref{Prob_Pen} and illustrate how  our proposed constraint dissolving approaches inherit the convergence properties from the implemented unconstrained approaches in Section 3. The proofs for the  theoretical properties of \ref{Prob_Pen} are presented in the appendix. In Section 4, we discuss how to choose the constraint dissolving operator $\A$ for \ref{Prob_Pen}.  We conclude the paper in the last section.

	\section{Notations, definitions and constants}
	\subsection{Notations}
	Let $\mathrm{range}(A)$ be the subspace spanned by the column vectors of matrix $A$, and $\norm{\cdot}$ represents the $\ell_2$-norm of a vector or an operator.
	The notations $\mathrm{diag}(A)$ and $\Diag(x)$
	stand for the vector formed by the diagonal entries of a matrix $A$,
	and the diagonal matrix with the entries of $x\in\bb{R}^n$ as its diagonal, respectively. 
	We denote the smallest and largest eigenvalues of $A$ by $\lambda_{\mathrm{min}}(A)$ and $\lambda_{\max}(A)$, respectively. Besides, $\sigma_{\min}(A)$ refers to the smallest singular value of matrix $A$. Furthermore, for any matrix $A \in \bb{R}^{n\times p}$, 
	the pseudo-inverse of $A$ is denoted by $A^\dagger \in \bb{R}^{p\times n}$, which satisfies $AA^\dagger A = A$, $A^\dagger AA^\dagger = A^\dagger$, and both $A^{\dagger} A$ and $A A^{\dagger}$ are symmetric \cite{GolubMatrix}. 
	
	In this paper, the Riemannian metric for $\M$ is chosen as the Euclidean metric in $\bb{R}^n$. For any $x\in\M$, we denote  
	$\Tx:= \{ d \in \bb{R}^n: d\tp \Jc(x) = 0  \} = \mathrm{Null}( \Jc(x)\tp )$
	and 
	$\Nx := \{d \in \bb{R}^{n}:  d\tp u = 0, ~  \forall u \in \Tx \} = \mathrm{range}(\Jc(x))$ as the tangent and normal spaces of $\M$ at $x$,
	respectively. Additionally, for any $x \in \M$, we denote the Riemannian gradient and Riemannian Hessian of $f$ at $x$ as $\grad f(x)$ and $\hess f(x)$, respectively.
	
	For any $x \in \bb{R}^n$, we define the projection from $x \in \bb{R}^n$ to $\M$ as 
	\begin{equation*}
		\mathrm{proj}(x, \M) := \mathop{\arg\min}_{y \in \M} ~ \norm{x-y}. 
	\end{equation*} 
	It is worth mentioning that the optimality condition of the above problem leads to the fact that $x - w \in \mathrm{range}(\Jc(w))$ for any $w \in \mathrm{proj}(x, \M)$.
	Furthermore, $\mathrm{dist}(x, \M)$ refers to the distance between $x$ and $\M$, i.e. 
		\begin{equation*}
			\mathrm{dist}(x, \M) := \mathop{\min}_{y \in \M} ~ \norm{x-y}.
		\end{equation*}

	The transposed Jacobian of the mapping $\A$ is denoted as $\Ja(x) \in \bb{R}^{n\times n}$. Recall the definition of $\Jc(x)$, and let $c_i$ and $\A_{i}$ be the $i$-th coordinate of the mapping $c$ and $\A$ respectively, then $\Jc$ and $\Ja$ can be expressed by
	\begin{equation*}
		\scriptsize
		\Jc(x) := \left[  \begin{smallmatrix}
			\frac{\partial c_1(x_1)}{\partial x_1} & \cdots & \frac{\partial c_p(x_1)}{\partial x_1} \\
			\vdots & \ddots & \vdots \\
			\frac{\partial c_1(x_n)}{\partial x_n} & \cdots & \frac{\partial c_p(x_n)}{\partial x_n} \\
		\end{smallmatrix}\right] \in \bb{R}^{n\times p},
		\quad \text{and} ~   
		\Ja(x) := \left[  \begin{smallmatrix}
			\frac{\partial \A_1(x_1)}{\partial x_1} & \cdots & \frac{\partial \A_n(x_1)}{\partial x_1} \\
			\vdots & \ddots & \vdots \\
			\frac{\partial \A_1(x_n)}{\partial x_n} & \cdots & \frac{\partial \A_n(x_n)}{\partial x_n} \\
		\end{smallmatrix}\right] \in \bb{R}^{n\times n}. 
	\end{equation*}
	Besides, $\DJa(x): d \mapsto \DJa(x)[d]$ denotes the second-order derivative of the mapping $\A$, which can be regarded as a linear mapping from $\bb{R}^n$ to $\bb{R}^{n\times n}$ and satisfies $\DJa(x)[d] = \lim_{t \to 0} \frac{1}{t}(\Ja(x + td) - \Ja(x) ) \in \bb{R}^{n\times n}$. 
	Similarly,  $\DJc(x)$ refers to the second-order derivative of the mapping $c$, which satisfies $\DJc(x)[d] = \lim_{t \to 0} \frac{1}{t}(\Jc(x + td) - \Jc(x) ) \in \bb{R}^{n\times p}$. Additionally,  we set $$\A^{k}(x) := \underbrace{\A(\A(\cdots\A(x)\cdots))}_{k \text{ times}},$$
	for $k \geq 1$, and define $\A^0(x) := x$, $\A^{\infty}(x):= \lim\limits_{k \to +\infty} \A^k(x)$. Furthermore, we denote $g(x) := f(\A(x))$ and use $\nabla f(\A(x))$ to denote $\nabla f(z)\large|_{z = \A(x)}$ in the rest of this paper.

	\subsection{Definitions}
	We first state the first-order optimality condition of \ref{Prob_Ori} as follows.
	\begin{defin}[\cite{nocedal2006numerical}]\label{Defin_FOSP}
		Given $x \in \bb{R}^n$, we say $x$ is a first-order stationary point of \ref{Prob_Ori} if there exists $\tilde{\lambda} \in \bb{R}^p$ that satisfies  
		\begin{equation}
			\label{Eq_Defin_FOSP}
			\left\{\begin{aligned}
				\nabla f(x) -  \sum_{i=1}^p\tilde{\lambda}_i \nabla c_i(x) &= 0,\\
				c(x) &= 0.
			\end{aligned}\right.
		\end{equation}
	\end{defin}

	For any given $x \in \M$, 
	we define $\lambda(x)$ as $$\lambda(x) := \Jc(x)^\dagger \nabla f(x)\in\argmin\limits_{\lambda\in\bb{R}^p} \norm{\nabla f(x) - \Jc(x) \lambda},$$
	where $\Jc(x)^\dagger =  \left(\Jc(x)\tp \Jc(x)\right)^{-1}\Jc(x)\tp $ since $\Jc(x)$ has full column rank when $x \in \M$.   
	Then it can be easily verified that  
	\begin{equation}\label{add:2}
		\nabla f(x) = \Jc(x) \lambda(x), 
	\end{equation}
	whenever $x$ is a first-order stationary point of \ref{Prob_Ori}. 
	
	\begin{defin}\label{Defin_SOSP}
		Given $x \in \bb{R}^n$, we say $x$ is a second-order stationary point of \ref{Prob_Ori} if $x$ is a first-order stationary point of \ref{Prob_Ori} and for any $d \in \Tx$, it holds that 
		\begin{equation}
			d\tp \left(\nabla^2 f(x) - \sum_{i=1}^p \lambda_i(x) \nabla^2 c_i(x)\right)d \geq 0.
		\end{equation}
	\end{defin}

		As the Riemannian metric on $\M$ is fixed as the Euclidean metric in $\bb{R}^n$, the following proposition presents the closed-form expressions for the Riemannian gradient $\grad f(x)$ and Riemannian Hessian $\hess f(x)$ for any $x \in \M$. 
		\begin{prop}
			\label{Prop_FOSP_Rie}
			Given $x \in \M$,  the Riemannian gradient of $f$ at $x$ can be expressed as  
			\begin{equation}\label{add:4}
				\grad f(x) = \nabla f(x) - \Jc(x)\lambda(x).
			\end{equation}
			Moreover, $\hess f(x)$ can be expressed as the following self-adjoint linear map $\hess f(x) : \Tx \to \Tx$ such that
			\begin{equation*}
				d\tp \hess f(x) d =  d\tp \left(\nabla^2 f(x) - \sum_{i=1}^p \lambda_i(x) \nabla^2 c_i(x)\right)d, \qquad \text{for any } d \in \Tx. 
			\end{equation*}
		\end{prop}
		The proof of the above proposition directly follows \cite{Absil2009optimization,boumal2020introduction}, and hence is omitted for simplicity.

	Given $x \in \M$,  the smallest eigenvalue of $\hess f(x)$ is defined as $$\lambda_{\min}(\hess f(x)) := \min_{d \in \Tx, \norm{d} = 1} ~d\tp \hess f(x) d.$$
	Let $U_x$ be a matrix whose columns forms an orthonormal basis of $\Tx$, we define the projected Hessian of \ref{Prob_Ori} at $x$ as 
	\begin{equation}
		\label{Eq_defin_HX}
		\ca{H}(x):=  U_x\tp \left(\nabla^2 f(x) - \sum_{i=1}^p \lambda_i(x) \nabla^2 c_i(x)\right)U_x,
	\end{equation}
	The following proposition characterizes the relationship between $\ca{H}(x)$ and $\hess f(x)$. 
	
	\begin{prop}
		\label{Prop_SOSP_Rie}
		Given any $x \in \M$, suppose $x$ is a first-order stationary point of \ref{Prob_Ori}, then $\ca{H}(x)$ and $\hess f(x)$ have the same eigenvalues. Moreover, $\lambda_{\min}(\hess f(x)) = \lambda_{\min}(\ca{H}(x))$ and $x$ is a second-order stationary point of \ref{Prob_Ori} if and only if $\ca{H}(x) \succeq 0$. 
	\end{prop}
	The proof of the above proposition directly follows from the Proposition \ref{Prop_FOSP_Rie} and \cite{Absil2009optimization}, and hence we omit it for simplicity. 
	
		\begin{defin}[\cite{lee2019first,criscitiello2019efficiently}]
			\label{Defin_Saddles}
			Given $x \in \M$, we say that $x$ is a strict saddle point of \ref{Prob_Ori} if $x$ is a first-order stationary point of \ref{Prob_Ori} and  $\lambda_{\min}(\hess f(x)) < 0$. 
			Moreover, given $x \in \bb{R}^n$, we say that  $x$ is a strict saddle point of \ref{Prob_Pen} if $x$ is a first-order stationary point of \ref{Prob_Pen} and $\lambda_{\min}(\nabla^2 h(x)) < 0$. 
		\end{defin}
	
	\medskip
	\begin{defin}
		Given $x \in \bb{R}^n$, we say $x$ is a first-order stationary point of \ref{Prob_Pen} if
		\begin{equation}
			\nabla h(x) = 0.
		\end{equation}
		Besides, when $h$ is twice-order differentiable, we say a point $x \in \bb{R}^n$ is a second-order stationary point of $h$ if $x$ is a first-order stationary point of $h$ and satisfies
		\begin{equation}
			\nabla^2 h(x) \succeq 0. 
		\end{equation}
	\end{defin}

	Next, we present the definition of the {\L}ojasiewicz gradient inequality \cite{lojasiewicz1961probleme,lojasiewicz1963propriete,
		bolte2014proximal}, which is a powerful tool in analyzing the convergence of various unconstrained optimization approaches.
	\begin{defin}
		Given $x  \in \bb{R}^n$, the function $f$ is said to satisfy the (Euclidean) {\L}ojasiewicz gradient inequality at $x$ if and only if there exist a neighborhood $\ca{U}$ of $x$, and constants $\theta \in (0,1]$, $C>0$, such that the following inequality holds for any $y \in \ca{U}$,
		\begin{equation*}
			\norm{\nabla  f(y)}  \geq C |f(y) - f(x)|^{1-\theta}.
		\end{equation*}
	\end{defin}
	The {\L}ojasiewicz gradient inequality on a Riemannian manifold \cite{hosseini2015convergence} can be similarly defined in the following definition.
	\begin{defin}
		Given $x \in \M$, the function $f$ is said to satisfy the Riemannian {\L}ojasiewicz gradient inequality at $x$ if and only if there exist a neighborhood $\ca{U} \subset \M$ of $x$, and constants $\theta \in (0,1]$, $C>0$, such that the following inequality holds for any $y \in \ca{U}$, 
		\begin{equation*}
			\norm{\grad  f(y)}  \geq C |f(y) - f(x)|^{1-\theta}.
		\end{equation*}
	\end{defin}
	The constant $\theta$ is usually referred as {\it (Riemannian) {\L}ojasiewicz exponent in the gradient inequality}, which is simply abbreviated as {\it (Riemannian) {\L}ojasiewicz exponent}.

	\subsection{Constants}
	For any given $x \in \M$, we define the positive scalar $\rho_x \leq 1$ as
	\begin{equation*}
		\rho_x := \mathop{\arg\max}_{0 < \rho \leq 1}~ \rho \quad \text{s.t.} ~ \inf \left\{\sigma_{\min}(\Jc(y)): y \in \bb{R}^n, \norm{y-x} \leq \rho \right\} \geq \frac{1}{2} \sigma_{\min}(\Jc(x)).
	\end{equation*}
	Based on the definition of $\rho_x$, we can define the set $\Theta_x:= \{ y \in \bb{R}^n: \norm{y-x} \leq \rho_x \}$ and define several constants as follows:
	\begin{itemize}
		\item $\Lsc  := \sigma_{\min}(\Jc(x))$; 
		\item $\Lf  := \sup_{y \in \Theta_x}~ \norm{\nabla f(\A(y))}$;
		\item $\Mc  := \sup_{y \in \Theta_x}   \norm{\Jc(y)} $;
		\item $\Ma  := \sup_{y \in \Theta_x}  \norm{\Ja(y)}$;
		\item $\Lg  := \sup_{y, z \in \Theta_x, y\neq z} ~   \frac{\norm{\nabla g(y) - \nabla g(z)}}{\norm{y-z}}$;
		\item $\Lc := \sup_{y, z \in \Theta_x, y\neq z} \frac{\norm{\Jc(y) - \Jc(z) }}{\norm{y-z}} $;
		\item $\La := \sup_{y, z \in \Theta_x, y\neq z} \frac{\norm{\Ja(y) - \Ja(z) }}{\norm{y-z}} $;
		\item $\Lac := \sup_{y, z \in \Theta_x, y\neq z} \frac{\norm{ \Ja(y)\Jc(\A(y)) - \Ja(z)\Jc(\A(z)) }}{\norm{y-z}} $.
	\end{itemize}
	Based on these constants, we further set 
	\begin{equation*}
		\varepsilon_x := \min \left\{ \frac{\rho_{x}}{2}, \frac{\Lsc}{32\Lc(\Ma+1)}, \frac{\Lsc ^2}{8\Lac\Mc } \right\},
	\end{equation*}
	and define the following sets:
	\begin{itemize}
		
		\item $\Omegax{x} := \left\{ y \in \bb{R}^n: \norm{y - x}  \leq \varepsilon_x \right\}$;
		\item $\BOmegax{x} := \left\{ y \in \bb{R}^n: \norm{y - x}  \leq \frac{\Lsc \varepsilon_x}{4\Mc(\Ma +1) + \Lsc}  \right\}$; 
		\item $\Omega:= \bigcup_{x \in \M} \Omegax{x}$;
		\item $\bar{\Omega}:= \bigcup_{x \in \M} \BOmegax{x}$.
	\end{itemize}
	It is worth mentioning that Assumption \ref{Assumption_1} guarantees that $\Lsc > 0$ for any given $x \in \M$, which implies that $\varepsilon_x > 0$. On the other hand, we can conclude that $\BOmegax{x} \subset \Omegax{x} \subset \Theta_x$ holds for any given $x \in \M$, and $\M$ lies in the interior of $\bar{\Omega}$.

	\begin{defin}
		\label{Cond_beta}
		For any given $x \in \M$, we set
		\begin{equation*}
			\beta_{x}:= \max\left\{ \frac{128\Lg (\Ma +1)^2}{\Lsc^2 }, \frac{64\Lf (\Ma +1)\Lac }{\Lsc ^3}, \frac{16\Lf \La (2\Ma  + 1)}{\Lsc ^2}\right\}.
		\end{equation*} 
	\end{defin}
	
	\begin{rmk}
		When the manifold $\M$ is compact, there exists a finite set $\ca{I}\subset \M$ such that $\M \subseteq \bigcup_{x \in \ca{I}} \Theta_x$. Therefore, we can choose uniform positive lower bounds for $\Lsc$ and $\varepsilon_x$, while find uniform upper bounds for all the other aforementioned constants. Specifically, we can choose a uniform upper bound for $\beta_x$,
		which can be marked as a threshold.
		Any $\beta$ greater than this threshold
		can determine an exact penalty function for \eqref{Prob_Pen}.
	\end{rmk}
	
	Finally, the following assumption is needed when we discuss 
	the second-order stationarity of \ref{Prob_Pen}.
	\begin{assumpt}{\bf Assumption on twice-order differentiability}
		\label{Assumption_4}
		
		\begin{itemize}
			\item $f$, $c$ and $\A$ are twice differentiable in $\bb{R}^n$.
		\end{itemize}
	\end{assumpt}

	\section{Theoretical Results}
	In this section, we present some theoretical properties of \ref{Prob_Pen}. We begin with the characteristics of the mapping $\A$ in Section \ref{Subsection_Theo_properties_A}. Then we investigate the stationarity of \ref{Prob_Pen} as well as the {\L}ojasiewicz exponents in Section \ref{Subsection_Theo_properties_CDF}. Finally, in Section \ref{Subsection_framework}, we propose a framework 
	showing that the convergence properties of  \ref{Prob_Pen}
	can directly be inherited from those of the applied unconstrained optimization approaches.

	\subsection{Theoretical properties of $\A$}
	\label{Subsection_Theo_properties_A}
	We start with evaluating the relationships among $\norm{c(x)}$, $\norm{c(\A(x))}$ and $\mathrm{dist}(x, \M)$ in the following lemmas.
	\begin{lem}
		\label{Le_bound_cx}
		For any given $x \in \M$, the following inequalities hold for any $\y \in \Omegax{x}$,
		\begin{equation}
			\frac{1}{\Mc } \norm{c(\y)} \leq \mathrm{dist}(\y, \M) \leq \frac{2}{\Lsc } \norm{c(\y)}.
		\end{equation}
	\end{lem}
	\begin{proof}
	Let $z \in \mathrm{proj}(y, \M)$, then from the definition of $z$ we can conclude that $\norm{z-y}\leq \norm{y-x}$. Thus $\norm{z-x} \leq \norm{z-y} + \norm{y-x} \leq 2\varepsilon_x \leq \rho_x$, and hence $z \in \Theta_x$. 
	By the mean-value theorem, for any fixed $\nu \in \bb{R}^p$ there exists a point $\xi_\nu \in  \bb{R}^n$ that is a convex combination of $\y$ and $z$ such that $ \nu\tp c(y) = (y-x)\tp\Jc(\xi_\nu)\nu $. By the convexity of $\Theta_x$, $\xi_\nu \in \Theta_x$ holds for any $\nu \in \bb{R}^p$. Therefore, we get
	\begin{equation*}
		\norm{c(\y)} = \sup_{\nu \in \bb{R}^p, \norm{\nu} = 1} \nu\tp c(y) = \sup_{\nu \in \bb{R}^p, \norm{\nu} = 1} (y-x)\tp\Jc(\xi_\nu)\nu  \leq 
		\sup_{\nu \in \Theta_x} 
		\norm{\Jc(\xi_\nu)} \norm{\y - z} \leq \Mc  \mathrm{dist}(\y, \M).
	\end{equation*}
	Moreover, it follows from the definition of $z$ that $y-z \in \mathrm{range}(\Jc(z))$.  As a result, let $\tilde{\nu} = \frac{\Jc(z)\tp(y-z)}{\norm{\Jc(z)\tp(y-z)}}$, we have
	\begin{equation*}
		\begin{aligned}
			&\norm{c(\y)}= \sup_{\nu \in \bb{R}^p, \norm{\nu} = 1} (y-z)\tp\Jc(\xi_\nu)\nu  \geq (y-z)\tp\Jc(\xi_{\tilde{\nu}})\tilde{\nu}  \\
			={}&  (y-z)\tp\Jc(z)\tilde{\nu} + (y-z)\tp\left(\Jc(z) - \Jc(\xi_{\tilde{\nu}})\right)\tilde{\nu}  \\
			={}& \norm{ \Jc(z)\tp (y-z)} + (y-z)\tp\left(\Jc(z) - \Jc(\xi_{\tilde{\nu}})\right)\tilde{\nu} \\
			\geq{}& \norm{ \Jc(z)\tp (y-z)} - \Lc \norm{y-z}^2 \\
			\geq{}&  (\Lsc - \varepsilon_x \Lc )\mathrm{dist}(\y, \M)  \geq  \frac{\Lsc }{2} \mathrm{dist}(\y, \M).
		\end{aligned}
	\end{equation*}
	\end{proof}

	\begin{lem}
		\label{Le_Ax_c}
		For any given $x \in \M$, it holds that
		\begin{equation}
			\norm{\A(\y) - \y }  \leq \frac{2(\Ma +1)}{\Lsc }\norm{c(\y)}, \qquad \text{ for any $\y \in \Omegax{x}$}.
		\end{equation}
	\end{lem}
	\begin{proof}
	For any given $\y \in \Omegax{x}$, we choose $z \in \mathrm{proj}(y, \M)$. Then we can conclude that $z \in \Theta_x$. Furthermore, from the Lipschitz continuity of $\A$ and the fact that $\A(z) - z = 0$,  it holds that
	\begin{equation}
		\norm{\A(\y) - \y} = \norm{(\A(\y) - \y)  - (\A(z) - z)}  \leq  (\Ma +1)\mathrm{dist}(y, \M) \leq \frac{2(\Ma +1)}{\Lsc }\norm{c(\y)},
	\end{equation}
	where the last inequality follows from Lemma \ref{Le_bound_cx}.
	\end{proof}

	\medskip
	\begin{lem}
		\label{Le_A_secondorder_descrease}
		For any given $x \in \M$,  it holds that 
		\begin{equation}
			\norm{c(\A(\y))} \leq \frac{4\Lac }{\Lsc ^2}\norm{c(\y)}^2, \qquad \text{for any $\y \in \Omegax{x}$}.
		\end{equation}
	\end{lem}
	
	\begin{proof}
	For any given $\y \in \Omegax{x}$, we choose $z \in \mathrm{proj}(y, \M)$. It holds that $z \in \Theta_x$. By the mean-value theorem, for any $\nu \in \bb{R}^p$, there 
	exists $t_\nu \in [0,1]$ and $\xi_\nu = t_\nu \y + (1-t_\nu)z$ such that
	\begin{equation}
		\nu\tp  c(\A(\y)) = \nu\tp \left(\Ja(\xi_\nu)\Jc(\A(\xi_\nu))\right)\tp (\y - z) .
	\end{equation}
	The convexity of $\Theta_x$ ensures that $\xi_\nu \in \Theta_x$ holds for any $\nu \in \bb{R}^p$. 
	Therefore, from the definition of $\Lac $ and $\Omegax{x}$, we get
	\begin{equation}
		\begin{aligned}
			\norm{c(\A(\y))} ={}& \sup_{\nu \in \bb{R}^p, \norm{\nu} = 1}  \nu\tp  c(\A(\y)) = \sup_{\nu \in \bb{R}^p, \norm{\nu} = 1}  
			\nu\tp \left(\Ja(\xi_\nu)\Jc(\A(\xi_\nu))\right)\tp (\y - z) 
			\\
			\leq{}& \sup_{\nu \in \bb{R}^p, \norm{\nu} = 1} \norm{\left(\Ja(\xi_\nu)\Jc(\A(\xi_\nu))\right)\tp (\y - z)} 
			\leq{}   \sup_{\nu \in \bb{R}^p, \norm{\nu} = 1} \norm{\Ja(\xi_\nu)\Jc(\A(\xi_\nu))}\mathrm{dist}(\y, \M)\\
			\leq{}& \Lac  \sup_{\nu \in \bb{R}^p, \norm{\nu} = 1} \norm{\xi_\nu - z} \mathrm{dist}(\y, \M) 
			\leq{} \Lac  \mathrm{dist}(\y, \M) ^2 \leq \frac{4\Lac }{\Lsc ^2}\norm{c(\y)}^2,
		\end{aligned}
	\end{equation}
	where the last inequality follows from Lemma \ref{Le_bound_cx}. 
	\end{proof}

	\medskip
	For any given $x \in \M$ and $y \in \Omegax{x}$, Lemma \ref{Le_A_secondorder_descrease} illustrates that the operator $\A$ can reduce the feasibility violation of $y$ quadratically when $y$ is sufficiently close to $\M$. 
	
	Next, we present the theoretical property of $\A^{\infty}(\y)$.
	\begin{lem}
		\label{Le_dist_Ainfty}
		For any given $x \in \M$ and any $\y \in \BOmegax{x}$, $\A^{\infty}(\y)$ exists and  $\A^{\infty}(\y) \in \Omegax{x} \cap \M$. Moreover,  it holds that
		\begin{equation}
			\norm{\A^{\infty}(\y) - \y} \leq \frac{4(\Ma +1)}{\Lsc } \norm{c(\y)}. 
		\end{equation}
	\end{lem}
	The proof for Lemma \ref{Le_dist_Ainfty} is presented in Section \ref{Section_proofs_311}.

	In the rest of this subsection, we study the properties of $\Ja(x)$. The following lemmas characterize the range space and null space of $\Ja(x)$ for any given $x \in \M$.  
	\begin{lem}
		\label{Le_Ja_identical}
		For any given $x \in \M$, the  inclusion 
		$\Ja(x)\tp d  \in \Tx$
		holds for any $d  \in \bb{R}^n$. 
		Moreover, when $d \in \Tx$, it holds that
		$\Ja(x)\tp d = d$. 
	\end{lem}
	\begin{proof}
	Firstly, for any $d \in \bb{R}^n$, and any $d_1 \in \Nx$, from the fact that $\Ja(x)\Jc(x) = 0$, we can conclude that $\Ja(x) d_1 = 0$. Then we have that 
	$d_1\tp \Ja(x)\tp d = 0$ holds for any $d_1\in \Nx$, which implies that
	$\Ja(x)\tp d  \in \Nx^\perp = \Tx.$
	Thus we obtain that $\Ja(x)\tp d  \in \Tx$ holds for any $d \in \bb{R}^n$. 
	
	On the other hand, by the Taylor expansion of $c(x + td)$ up to the 
	second-order term, we have  $\norm{c(x+td)} \leq \Mc \norm{d}^2$ holds for any $d \in \Tx$. In addition, when $d \in \Tx$,
	we have that $x+td \in \Omegax{x}$ for $t \in (0,1)$ sufficiently small. Then by Lemma \ref{Le_Ax_c}, we get
	\begin{equation*}
		\norm{x+ td - \A(x+ td)} \leq \frac{ 2(\Ma  +1)}{\Lsc } \norm{c(x+ td)} \leq \frac{2(\Ma  + 1)\Mc }{\Lsc } \norm{d}^2 t^2. 
	\end{equation*}
	By Assumption \ref{Assumption_2}, $\A(x) = x$. Thus we have
	\begin{equation*}
		\begin{aligned}
			&\Ja(x)\tp d = \lim\limits_{t \to 0} \frac{\A(x+td) - \A(x)}{t} \\
			={}& \lim\limits_{t \to 0} \frac{\A(x+td)-(x+td ) }{t} + \lim\limits_{t \to 0} \frac{x+td - x}{t}= \lim\limits_{t \to 0} \frac{x+td - x}{t} = d,
		\end{aligned}
	\end{equation*}
	which completes the proof.
	\end{proof}

	\begin{lem}
		\label{Le_Ja_nullspace}
		For any given $x \in \M$, the equality 
		$\Ja(x)d = 0$ holds if and only if
		$d \in \Nx.$
	\end{lem}
	\begin{proof}
	
	For any  $d_2$ satisfying $\Ja(x)d_2 = 0$ and any $d_1 \in \Tx$, by Lemma \ref{Le_Ja_identical} we have
	$0 = d_1\tp \Ja(x) d_2 = d_1\tp d_2$,
	which implies that $d_2 \in \Nx$. 
	
	On the other hand, let $d_2\in\Nx$. Then $d_2 = \Jc(x) \eta$ for some $\eta \in \bb{R}^p$. By Assumption \ref{Assumption_2}, $\Ja(x) d_2 = \Ja(x) \Jc(x) \eta = 0$.  This completes the proof.
	\end{proof}
	
	\medskip
	Next we show the idempotence of $\Ja(x)$, i.e. $\Ja(x)^2 = \Ja(x)$ holds for any $x \in \M$.
	\begin{lem}
		\label{Le_Ja_idempotent}
		For any given $x \in \M$, it holds that
		$\Ja(x)^2 = \Ja(x)$.
	\end{lem}
	\begin{proof}
	For any $d_1 \in \Tx$, let $d_2 = \Ja(x)d_1 - d_1$. 
	Then for any $d_3 \in \Tx$, it follows from Lemma \ref{Le_Ja_identical} that $d_3\tp d_2 = d_3\tp \Ja(x)d_1 - d_3\tp d_1 = 0$, which further implies $d_2 \in \Nx$. 
	On the other hand, for any $d_4 \in \Nx$,
	it follows from Lemma \ref{Le_Ja_nullspace} that $\Ja(x)d_4 = 0$.
	
	Therefore, for any $d_1 \in \Tx$ and any $d_4 \in \Nx$, it holds that 
	\begin{equation*}
		\Ja(x)^2 (d_1 + d_4) = \Ja(x)^2 d_1 =  \Ja(x) d_1 + \Ja(x) d_2 = \Ja(x) d_1 = \Ja(x)(d_1 + d_4). 
	\end{equation*}
	Therefore, we complete the proof by recalling the arbitrariness of $d_1 \in \Tx$ and $d_4 \in \Nx$. 
	\end{proof}

	\subsection{Theoretical properties of \ref{Prob_Pen}}
	\label{Subsection_Theo_properties_CDF}
	This subsection investigates the relationships between \ref{Prob_Ori} and \ref{Prob_Pen} on their stationary points, {\L}ojasiewicz exponents and local minimizers. We start by presenting the explicit expression for the gradient and Hessian of \ref{Prob_Pen}. 
	\begin{prop}
		\label{Prop_gradient_h}
		The gradient of $h$ in \ref{Prob_Pen} can be expressed as 
		\begin{equation}
			\nabla h(x) = \Ja(x)\nabla f(\A(x)) + \beta \Jc(x) c(x). 
		\end{equation}
		Furthermore, under Assumption \ref{Assumption_4}, the Hessian of $g(x) := f(\A(x))$ can be expressed as
		\begin{equation}
			\label{Eq_Prop_gradient_h_0}
			\nabla^2 g(x) = \Ja(x)\nabla^2 f(\A(x)) \Ja(x)\tp + \DJa(x)[\nabla f(\A(x))].
		\end{equation}
		and the Hessian of $h(x)$ is
		\begin{equation}
			\nabla^2 h(x) = \nabla^2 g(x) + \beta \left(\Jc(x)\Jc(x)\tp + \DJc(x)[c(x)]\right). 
		\end{equation}
	\end{prop}
	The statements of Proposition \ref{Prop_gradient_h} can be verified by straightforward calculations, and hence its proof is omitted. 
	
	\subsubsection{Relationships on stationary points}
	\label{Subsubsection_Relationship_stationary}
	This subsection shows that \ref{Prob_Ori} and \ref{Prob_Pen} share the same first-order and second-order stationary points in a neighborhood
	of $\M$.   
	Moreover, for any given $x \in \M$ and $\beta \geq \beta_{x}$, we can prove that all first-order stationary points of \ref{Prob_Pen} in $\BOmegax{x}$ are feasible.
	For a more concise presentation, we put all the proofs for  Proposition \ref{Prop_Firstorder_Equivalence_feasible}, Theorem \ref{The_firstorder_equivalence} and Theorem \ref{The_secondorder_equivalence} in Appendix \ref{Section_proofs_321}.
	
	\begin{prop}	
		\label{Prop_Firstorder_Equivalence_feasible}
		Any first-order stationary point of \ref{Prob_Ori} is a first-order stationary point of \ref{Prob_Pen}. On the other hand, any first-order stationary point of \ref{Prob_Pen} over $\M$ is a first-order stationary point of \ref{Prob_Ori}. 
	\end{prop}
	
		\begin{theo}
			
			\label{The_firstorder_equivalence}
			For any given $x \in \M$, suppose $\beta \geq \beta_{x}$, then the following inequality holds for any $\y \in \Omegax{x}$,
			\begin{equation}
				\norm{\nabla h(\y)} \geq \frac{\beta\Lsc}{8(\Ma +1)} \norm{c(\y)}. 
			\end{equation}
			Moreover, any first-order stationary point of \ref{Prob_Pen} in $\Omegax{x}$ is a first-order stationary point of \ref{Prob_Ori}. 
		\end{theo}

		\medskip
		\begin{theo}
			
			\label{The_secondorder_equivalence}
			Suppose Assumption \ref{Assumption_4} holds. Any second-order stationary point of \ref{Prob_Ori} is a second-order stationary point of \ref{Prob_Pen}. 
			On the other hand, for any given $x \in \M$, if $\beta \geq \beta_{x}$ and $y\in\Omega_x$ 
			is a second-order stationary point of \ref{Prob_Pen}, then $y$ is a second-order stationary point of \ref{Prob_Ori}. 
		\end{theo}
		
		Based on Theorem \ref{The_secondorder_equivalence}, the following corollary illustrates that \ref{Prob_Ori} and \ref{Prob_Pen} have the same strict saddle points locally. The proof for Corollary \ref{Coro_saddle_equivalence} directly follows from Theorem \ref{The_firstorder_equivalence} and Theorem \ref{The_secondorder_equivalence}, hence we omit it for simplicity. 
		\begin{coro}
			\label{Coro_saddle_equivalence}
			Suppose Assumption \ref{Assumption_4} holds. Any strict saddle point of \ref{Prob_Ori} is a strict saddle point of \ref{Prob_Pen}. 
			On the other hand, for any given $x \in \M$, if $\beta \geq \beta_{x}$ and $y\in\Omega_x$ is a strict saddle point of \ref{Prob_Pen}, then $y$ is a strict saddle point of \ref{Prob_Ori}. 
		\end{coro}

	\subsubsection{Stationarity at infeasible points}
	\label{Subsubsection_Estimation_stationarity}
	From Lemma \ref{Le_A_secondorder_descrease},
	we know that the constraint dissolving operator $\A$ can quadratically reduce
	the feasibility violation of any infeasible point $x \in \bar{\Omega}$. Together with  Lemma \ref{Le_dist_Ainfty}, 
	$\A^{\infty}(x)$ is feasible and it is in a neighborhood of $x$.
	The relationships  between $x$ and $\A^{\infty}(x)$ in terms of the function values and derivatives are of great importance in characterizing the properties of 
	\ref{Prob_Pen} at those infeasible points.

	The detailed proofs of Proposition \ref{Prop_postprocessing} -- Proposition \ref{Prop_esti_lambda_min} are presented in Appendix \ref{Section_proofs_322}.
	\begin{prop}	
		\label{Prop_postprocessing}
		For any given $x \in \M$, suppose $\beta \geq \beta_{x}$, then the following inequalities hold for any $\y \in \BOmegax{x}$
		\begin{align}
			&h(\A(\y)) - h(\y) \leq - \frac{\beta}{8}\norm{c(\y)}^2,\\
			&h(\A^{\infty}(\y)) - h(\y)\leq  - \frac{\beta}{4}\norm{c(\y)}^2.
		\end{align}
	\end{prop}

	When we invoke a specific unconstrained optimization algorithm to solve \ref{Prob_Pen}, we usually terminate the algorithm once the stopping criterion reaches a certain tolerance, meanwhile the feasibility violation at the returned solution $x$ may not be sufficiently small. To pursue a
	solution with high feasibility accuracy,
	certain post-processing step should be imposed. As
	we have mentioned,  $\A^{\infty}(x)$ is feasible if $x$ is
	sufficiently close to $\M$. Thus we can recursively compute
	$x \leftarrow \A(x)$ until the desired accuracy on feasibility is satisfied.

	\begin{prop}
		
		\label{Prop_upperbound_projection_gradient}
		For any given $x \in \M$, suppose $\beta \geq \beta_{x}$, then it holds that 
		\begin{equation}
			\norm{\nabla h(\y)} \geq \frac{1}{2}\norm{\grad f(\A^{\infty}(\y))}, \qquad \text{for all } y \in \BOmegax{x}.
		\end{equation}
		
	\end{prop}
	Proposition \ref{Prop_upperbound_projection_gradient}
	shows that $\norm{\grad f(A^{\infty}(y))}$  can
	be controlled by $\norm{\nabla h(y)}$. Next, we study the
	relationship between the Riemannian Hessian of $f$ at $A^{\infty}(y)$
	and the Hessian of $h$ at $y$. 
	
	\begin{prop}
		\label{Prop_sepctral_interlace}
		Suppose Assumption \ref{Assumption_4} holds, then the following inequalities hold for any $x \in \M$
		\begin{align}
			& \lambda_{\min}(\hess f(x)) \geq \lambda_{\min}(\nabla^2 h(x)) - \La \norm{\grad f(x)},  \\
			& \lambda_{\max}(\hess f(x)) \leq \lambda_{\max}(\nabla^2 h(x)) + \La \norm{\grad f(x)}.
		\end{align}
		
	\end{prop}

	\begin{prop}
		\label{Prop_esti_lambda_min}
		Suppose Assumption \ref{Assumption_4} holds. For any given $x \in \M$ and $\beta \geq \beta_{x}$, the following inequality holds for any $y \in \BOmegax{x}$,
		\begin{equation}
			\begin{aligned}
				\lambda_{\min}(\hess f(\A^{\infty}(y))) \geq{}& \lambda_{\min}( \nabla^2 h(y)) - \norm{\nabla^2 g(\A^{\infty}(y)) - \nabla^2 g(y)} \\
				&- \left(\frac{5}{2}\La + \frac{96 \Lc\Mc(\Ma + 1)^2}{ \Lsc^2} \right) \norm{\nabla h(y)}.
			\end{aligned}
		\end{equation} 
	\end{prop}
	
	\subsubsection{{\L}ojasiewicz exponents and local minimizers}
	\label{Subsubsection_Relationships_oters}
	The following proposition guarantees the fact that \ref{Prob_Ori} and \ref{Prob_Pen} share the same {\L}ojasiewicz exponents at any $x \in \M$. 
	\begin{prop}
		
		\label{Prop_KL}
		For any given $x \in \M$, suppose $\beta \geq \beta_{x}$ and \ref{Prob_Ori} satisfies the Riemannian {\L}ojasiewicz gradient inequality at  $x$ with Riemannian {\L}ojasiewicz exponent $\theta \in (0, \frac{1}{2}]$,  then \ref{Prob_Pen} satisfies the (Euclidean) {\L}ojasiewicz gradient inequality at $x$ with {\L}ojasiewicz exponent $\theta$.
	\end{prop}
	
	At the end of this subsection, we establish the relationship
	on the local minimizers between \ref{Prob_Ori} and \ref{Prob_Pen}.
	\begin{theo}
		
		\label{The_local_minimizer}
		For any given $x \in \M$, suppose $\beta \geq \beta_{x}$, then any local minimizer of \ref{Prob_Ori} in $\BOmegax{x}$ is a local minimizer of \ref{Prob_Pen}. Moreover, any local minimizer of \ref{Prob_Pen} in $\BOmegax{x}$ is a local minimizer of \ref{Prob_Ori}. 
	\end{theo}
	
	The detailed proofs of Proposition \ref{Prop_KL} and Proposition \ref{The_local_minimizer} are presented in Appendix \ref{Section_proofs_323}.
	
	\subsection{Constraint dissolving approach and theoretical analysis}
	
	\label{Subsection_framework}

	In this subsection, we show how to establish the convergence properties
	of the proposed constraint dissolving approaches directly from existing results.  In our proposed constraint dissolving approaches, we first construct the corresponding \ref{Prob_Pen} for \ref{Prob_Ori}, then we select a specific unconstrained optimization approach to minimize \ref{Prob_Pen}. Moreover, we perform a post-processing procedure to achieve high accuracy in feasibility, if necessary.  The details of the constraint dissolving approach are summarized in Algorithm \ref{Alg:Constraint_Dissloving}.

	\begin{algorithm}
		\caption{Constraint Dissolving Approach for \ref{Prob_Ori}.}
		\label{Alg:Constraint_Dissloving}
		\begin{algorithmic}[1]   
			\Require Input data:  manifold $\M$, objective function $f$, penalty parameter $\beta$, initial guess $x_0$, stationarity tolerance $\epsilon_s$, feasibility tolerance $\epsilon_f$, and a selected unconstrained optimization approach (UCO).
			\State Construct the constraint dissolving function \ref{Prob_Pen}.
			\State Initiated from $x_0$, invoke the UCO to solve CDF and generate a sequence $\{x_k\}$. Terminate when tolerance $\epsilon_s$ is reached and obtain $\tilde{x}$.
			\If{require post-processing}
			\While{$\norm{c(\tilde{x})} > \varepsilon_f$} 
			\State $\tilde{x} = \A(\tilde{x})$.
			\EndWhile
			\EndIf
			\State Return $\tilde{x}$.
		\end{algorithmic}  
	\end{algorithm}
	
		\begin{rmk}
			As shown in Lemma \ref{Le_A_secondorder_descrease}, for any $x \in \M$, the constraint dissolving operator $\A$ can quadratically reduce the feasibility violation of any infeasible point in $\Omegax{x}$. Therefore, given any $\varepsilon_f > 0$ in Algorithm \ref{Alg:Constraint_Dissloving}, step 4-6 in Algorithm \ref{Alg:Constraint_Dissloving} are only performed for at most $\ca{O}\left( \log(\log(\varepsilon_f^{-1})) \right)$ times. More precisely, when the UCO in Algorithm \ref{Alg:Constraint_Dissloving} yields an $\tilde{x}$ that satisfies $\norm{\nabla h(\tilde{x})} \leq \varepsilon_s$, then from Lemma \ref{Le_A_secondorder_descrease} and Theorem \ref{The_firstorder_equivalence},  step 4-6 in Algorithm \ref{Alg:Constraint_Dissloving} are only performed for at most $\ca{O}\left(   \log(\log( \varepsilon_s  \varepsilon_f^{-1} \beta^{-1})) \right)$ iterations. 
		\end{rmk}
	
	For convenience, we call the selected unconstrained optimization approach in Algorithm \ref{Alg:Constraint_Dissloving} as UCO, and let $\{x_k\}$
	be the iterates generated by
	Algorithm \ref{Alg:Constraint_Dissloving}.
	We assume that there exists a compact set $\hat{\Gamma} \subset \bb{R}^n$ such that
	$\{x_k\} \subset \bar{\Omega} \cap \hat{\Gamma}$. Hence, $\sup_{x \in \M\cap \hat{\Gamma}} \beta _x < +\infty$ due to the compactness of $\M\cap \hat{\Gamma}$. 
	We choose  sufficiently large $\beta$ such that  $\beta \geq \sup_{x \in \M\cap \hat{\Gamma}} \beta _x$. 
	Then we can adopt the following framework to establish the corresponding theoretical results.
	\begin{itemize}
		\item {\bf Global convergence:} Theorem \ref{The_firstorder_equivalence} illustrates that \ref{Prob_Ori} and \ref{Prob_Pen} have the same first-order stationary points in $\bar{\Omega} \cap \hat{\Gamma}$. Therefore, if a cluster point of  $\{x_k\}$ is a first-order stationary point of \ref{Prob_Pen}, then we can claim that it is a  first-order stationary point of \ref{Prob_Ori}. Moreover, Proposition \ref{Prop_KL} illustrates that \ref{Prob_Ori} and \ref{Prob_Pen} have the same {\L}ojasiewicz gradient exponents. Therefore, when \ref{Prob_Ori} satisfies the KL property, then the sequence convergence of $\{x_k\}$ can be guaranteed by Proposition \ref{Prop_KL} and existing results, for instance \cite{bolte2014proximal}, established for the selected unconstrained optimization approach in Algorithm \ref{Alg:Constraint_Dissloving}.
		\item {\bf Local convergence rate: } 
		Theorem  \ref{The_local_minimizer} shows that \ref{Prob_Ori} and \ref{Prob_Pen} have the same local minimizers in $\bar{\Omega} \cap \hat{\Gamma}$. Together with Proposition \ref{Prop_KL}, the local convergence rate of the sequence $\{x_k\}$ can be established from prior works \cite{ochs2014ipiano,li2015accelerated,lei2019stochastic}. 
		\item {\bf Worst case complexity: } Proposition \ref{Prop_upperbound_projection_gradient} establishes the relationship between $\norm{\nabla h(x_k)}$ and $\norm{\grad f(\A^{\infty}(x_k))}$. Consequently, if Algorithm \ref{Alg:Constraint_Dissloving} produces
		an iterate $x_k$ satisfying 
		$\norm{\nabla f(x_k)} \leq \epsilon_s$, then it holds that 
		$\norm{\grad f(\A^{\infty}(x_k))} \leq 2\epsilon_s$. As a result, the worst case complexity of Algorithm \ref{Alg:Constraint_Dissloving} can be obtained from prior works immediately  \cite{cartis2012adaptive,cartis2012complexity,cartis2018worst}. 
		\item {\bf The ability of escaping from saddle points: } 
		If a cluster point of  $\{x_k\}$ is a second-order stationary point of \ref{Prob_Pen}, then  
		Theorem \ref{The_secondorder_equivalence} guarantees that 
		this cluster point is a second-order stationary point of \ref{Prob_Ori} as well. Moreover, for any sequence $\{x_k\}$ generated by Algorithm \ref{Alg:Constraint_Dissloving}, Proposition \ref{Prop_esti_lambda_min} provides the relationship among $\hess f(\A^{\infty}(x_k))$, $\norm{\nabla h(x_k)}$ and $\lambda_{\min}(\nabla^2 h(\xk))$. Consequently,  Algorithm \ref{Alg:Constraint_Dissloving} inherits the escaping-from-saddle-point properties from its UCO, while the theoretical analysis directly follows existing results in unconstrained optimization \cite{ge2015escaping,jin2017escape,jin2018accelerated}. 
	\end{itemize}

	Clearly, the existence of a bounded set $\bar{\Omega} \cap \hat{\Gamma}$, which 
	contains all iterates 
	generated by Algorithm \ref{Alg:Constraint_Dissloving}, 
	is crucial for the establishment of the above-mentioned theoretical properties. 
	In the rest of this subsection, we provide easy-to-verify conditions for the existence of such a compact set under a mild assumption, which covers a board class of scenarios.

	\begin{assumpt}{\bf Assumption on the coercivity of $f(x)$ over $\M$}
		\label{Assumption_3}
		\begin{itemize}
			\item  The level set $ \Gamma_x^\star := \{ y \in \M: f(y) \leq f(x)  \}$ is compact for any $x \in \M$.
		\end{itemize}  
	\end{assumpt}
	
	Assumption \ref{Assumption_3} straightforwardly holds when 
	$\M$ is compact, and it is commonly assumed in the literature \cite{bai2014minimization,gao2021riemannian,son2021symplectic}.
	Those extreme situations that Assumption \ref{Assumption_3} does not hold are out of the scope of this paper.

	For any given $x \in \M$ and any constant $\zeta > 0$, we set $\Gamma_{x, \zeta}:= \{ y\in \bb{R}^n: \mathrm{dist}(y, \Gamma_x^\star) \leq \zeta + 1 \}$ and 
	$\mu_{x, \zeta} := \inf_{y \in \Gamma_{x, \zeta}\setminus \bar{\Omega}} \norm{c(y)}^2$. From the definition of $\bar{\Omega}$ and the compactness of $\Gamma_{x, \zeta}$, we can conclude that $\mu_{x, \zeta} >0$ holds for any $x \in \M$. In addition, we define $M_{\Gamma_{x, \zeta}}:= \sup_{y, z \in \Gamma_{x, \zeta}} g(y) - g(z)$, and $\Xi_{x} := \{ z \in \bar{\Omega} \cap \Gamma_{x, \zeta}: \mathrm{dist}(z, \Gamma_x^\star) \leq  1/2 \}$.  Then we  introduce the following threshold value of $\beta$. 
	\begin{defin}
		\label{Cond_beta_global}
		For any given $x \in \M$ and any $\zeta > 0$, we define 
		\begin{equation}
			\bar{\beta}_x := \max\left\{  \frac{4 M_{\Gamma_{x, \zeta}} }{\mu_{x, \zeta}}, \sup_{w \in \M \cap \Gamma_{x, \zeta}} \beta_w \right\}.
		\end{equation}
	\end{defin}

	\begin{prop}
		\label{Prop_barrier}
		Suppose Assumption \ref{Assumption_3} holds. For any given $x \in \M$ and $\zeta > 0$, let $\beta \geq \bar{\beta}_x$, then it holds that 
		\begin{equation}
			\{ y \in \bb{R}^n: h(y ) \leq h(x) \} \cap \Gamma_{x, \zeta} \subset  \Xi_{x}. 
		\end{equation}
	\end{prop}
	\begin{proof}
	For any $y \in \Gamma_{x, \zeta} \setminus \bar{\Omega}$, it holds from the definition of $M_{\Gamma_{x, \zeta}}$ and $\mu_{x, \zeta}$ that 
	\begin{equation} 
		h(y) - h(x) = g(y) + \frac{\beta}{2} \norm{c(y)}^2 - g(x) \geq -M_{\Gamma_{x, \zeta}} +  \frac{\mu_{x, \zeta} \beta }{2} > 0. 
	\end{equation}	
	Moreover, for any $y \in \left(\bar{\Omega} \cap \Gamma_{x, \zeta}\right) \setminus \Xi_x$, we show that $\A^{\infty}(y) \not \in \Gamma_x^\star$ by contradiction.
	Suppose on the contrary that 
	$\A^{\infty}(y) \in \Gamma_x^\star$. Then Lemma \ref{Le_dist_Ainfty} demonstrates that 
	$$\mathrm{dist}(y, \Gamma_x^\star) \leq \norm{\A^{\infty}(y) - y} \leq \sup_{z \in \Gamma_{x, \zeta} \cap \M} \varepsilon_z \leq \sup_{z \in \Gamma_{x, \zeta} \cap \M} \frac{\rho_z}{2} \leq \frac{1}{2}.$$ As a result, $y \in  \Xi_x$ and this contradicts the fact that $y \in \bar{\Omega} \setminus \Xi_x$. Therefore, $\A^{\infty}(y) \in \M \setminus \Gamma_x^\star$, and from Proposition \ref{Prop_postprocessing}, we obtain that 
	\begin{equation}
		h(y) \geq h(\A^{\infty}(y)) = f( \A^{\infty}(y) ) >  f(x) = h(x).
	\end{equation}
	This completes the proof. 
	\end{proof}
	
	\medskip
	The following corollary illustrates that with the help of Assumption \ref{Assumption_3}, we can actually further relax the requirement $\{x_k\} \subset \bar{\Omega} \cap \hat{\Gamma}$ to $x_0\in \bar{\Omega} \cap \hat{\Gamma}$ under mild conditions.
	\begin{coro}
		\label{Coro_barrier}
		Given any $x_0 \in \M$ and $\zeta > 0$, suppose Assumption \ref{Assumption_3} holds, $\beta \geq \bar{\beta}_{x_0}$ and Algorithm \ref{Alg:Constraint_Dissloving} generates a sequence $\{x_k\}$ that satisfies $h(x_k) \leq h(x_0)$ and $ \norm{x_{k+1} - x_k} \leq \zeta$ for any $k \geq 0$. Then it holds that $\{x_k\} \subset \bar{\Omega} \cap \Gamma_{x_0, \zeta}$. 
	\end{coro}
	\begin{proof}
	We prove the inclusion $\{ x_k \} \subset \bar{\Omega} \cap \Gamma_{x_0, \zeta}$ by induction. Suppose $x_j \subset \bar{\Omega} \cap \Gamma_{x_0, \zeta}$ for $0\leq j \leq k$. Then Proposition \ref{Prop_barrier} and the fact that $h(x_k) \leq h(x_0)$ 
	implies that $x_k \in \Xi_{x_0}$. 
	Moreover, it follows from the definition of $\Gamma_{x_0, \zeta}$ and $\Xi_{x_0}$ that $\mathrm{dist}( \bb{R}^n \setminus \Gamma_{x_0, \zeta},  ~ \Xi_{x_0}) \geq \zeta$. 
	
	Therefore, the fact that $\norm{x_{k+1} - x_k } \leq \zeta$ implies that $x_{k+1} \in \Gamma_{x_0, \zeta}$. Furthermore, together with Proposition \ref{Prop_barrier} and the fact that $h(x_{k+1}) \leq h(x_0)$, we arrive at $x_{k+1} \in \Xi_{x_0} \subset \bar{\Omega} \cap \Gamma_{x_0, \zeta}$. Namely, the inclusion $x_j \in \bar{\Omega} \cap \Gamma_{x_0, \zeta}$ holds for $0\leq j\leq k+1$. Then by induction, we obtain that $\{ x_k \} \subset \bar{\Omega} \cap \Gamma_{x_0, \zeta}$ holds for any $k\geq 0$.
	\end{proof}

	\medskip
	\begin{rmk}
		The conditions in Corollary \ref{Coro_barrier} are not restrictive at all. By choosing any monotone algorithm or algorithm that employ nonmonotone line search techniques \cite{grippo1986nonmonotone} as UCO, it is easy to guarantee that the relationship $h(x_k) \leq h(x_0)$ holds for any $k \geq 0$. 
		Moreover, the condition that $\norm{x_{k+1} - x_k} \leq \zeta$ holds for any $k \geq 0$ is also priorly verifiable in most cases. 
		For example, we can choose the UCO in Algorithm \ref{Alg:Constraint_Dissloving} as  
		a line-search method with maximal stepsize, a trust-region method with maximal 
		radius \cite{yuan2015recent}, or a cubic regularization method with an appropriate regularization parameter \cite{nesterov2006cubic}.
		For the other situations, 
		we can prefix a large $\zeta$ as a loose upper-bound for the distance between two consecutive iterates. Additionally, when $\{x_k\}$ has sequential convergence, the restriction $\norm{x_{k+1} - x_k} \leq \zeta$ naturally holds for any sufficiently large $k$.  
	\end{rmk}

	\section{Implementation}
	\label{Section_Implementation}
	In this section, we first show that we can construct the constraint dissolving operator $\A$ directly from $c(x)$ without any prior knowledge of the geometrical properties of $\M$. In addition, we provide easy-to-compute formulations of $\A$ for several well-known Riemannian manifolds, such as the Stiefel manifold, the Grassmann manifold, the symplectic Stiefel manifold, the hyperbolic manifold, etc. 
	Moreover, we provide an illustrative example of selecting the momentum-accelerated cubic regularization method as the unconstrained optimization approach in Algorithm \ref{Alg:Constraint_Dissloving}, and establish its convergence properties directly from existing works.

	\subsection{Construction of constraint dissolving operators}

	When we have no prior knowledge on the constraints $c(x) = 0$ in \ref{Prob_Ori}, we can consider the following mapping 
	\begin{equation}
		\label{Operator_A}
		\A_{c}(x):= x -  \Jc(x) \left( \Jc(x)\tp \Jc(x) + \alpha \norm{c(x)}^2 I_p \right)^{-1}  c(x).
	\end{equation}  
	It is worth mentioning that $\Jc(x)$ may be rank-deficient for some $x \in \bb{R}^n \setminus \M$ \cite{fletcher2002nonlinear,estrin2020implementing}, resulting in the discontinuity of $\Jc(x)^{\dagger}$. To this end, we choose to add a regularization term to $\Jc(x)\tp \Jc(x)$ with a prefixed constant $ \alpha > 0$ in \eqref{Operator_A}.
	Then $\A_{c}$ is locally Lipschitz smooth in $\bb{R}^n$ and we can consider the following penalty function, 
	\begin{equation}
		\label{Penalty_Expen_fletcher}
		h_c(x) := f(\A_{c}(x)) + \frac{\beta}{2} \norm{c(x)}^2. 
	\end{equation}
	The following lemma illustrates that $\A_{c}$ satisfies Assumption \ref{Assumption_2}. 
	\begin{lem}
		Suppose $c$ is twice locally Lipschitz continuously differentiable, then the constraint dissolving operator $\A_{c}$ satisfies Assumption \ref{Assumption_2}. 
	\end{lem}
	\begin{proof}
	The Lipschitz smoothness of $\A_{c}$ is guaranteed by the twice locally Lipschitz continuous differentiability of $c$. In addition, for any $x \in \M$, it follows from the equality $c(x) = 0$ that the equality $\A_{c}(x) =  x$ holds. Moreover, according to the fact that $J_{\A_{c}}(x) = I_n - \Jc(x) \Jc(x)^\dagger$ holds for any $x \in \M$, we obtain
	\begin{equation}
		J_{\A_{c}}(x) \Jc(x)  = \Jc(x) - \Jc(x) \Jc(x)^\dagger \Jc(x) = 0.
	\end{equation}
	Therefore, we can conclude that $\A_{c}$ satisfies Assumption \ref{Assumption_2}. 
	\end{proof}
	
	\medskip
	The mapping $\A_c$ in \eqref{Penalty_Expen_fletcher} only depends on $\Jc(x)$. As a result, for a wide range of Riemannian manifolds, we can develop the corresponding constraint dissolving function without any prior knowledge on the geometrical properties of $\M$.
	
	On the other hand, for several Riemannian manifolds 
	with explicit expressions, which are widely used in real life,
	we can choose specific constraint dissolving operators that are easy to calculate. We present the details in Table \ref{Table_implementation}. It can be easily verified that all the constraint dissolving operators presented in Table \ref{Table_implementation} satisfy Assumption \ref{Assumption_2}, and we omit the proofs for simplicity.
	Moreover, calculating these operators $\A$ and the corresponding $\Ja(x)$ only involve matrix-matrix multiplications. This implies that it is efficient to compute $\nabla h$ 
	once $\nabla f$ is obtained. In particular, compared with the Fletcher's penalties, \ref{Prob_Pen} avoids the needs to solve a system of linear equations in each function evaluation by appropriately selecting the the constraint dissolving operators
	for a variety of Riemannian manifolds in Table \ref{Table_implementation}.

	\begin{table}[htbp]
		\caption{Implementation of $\A$ for several Riemannian manifolds. Here ${\bf 0}_{m\times m}$ denotes the $m$-th order zero matrix, and $X^H$ denotes the conjugate transpose of a complex matrix $X$. }
		
		\label{Table_implementation}
		\footnotesize
		\centering
		\begin{tabular}{|lll|}
			\hline
			Name of the manifold & Expression of $\M$ & Possible choice of $\A$  \\ \hline
			Sphere & $\left\{ x \in \bb{R}^{n}:  x\tp x = 1 \right\}$ &  $x \mapsto 2x/(1 + \norm{x}_2^2)$ \\ \hline
			Oblique manifold & $\left\{ X \in \bb{R}^{m\times s}: \mathrm{Diag} (X \tp X) = I_s \right\}$ &  $X \mapsto 2X\left( I_s + \mathrm{Diag}(X\tp X) \right)^{-1} $ \\ \hline
			Stiefel manifold &$\left\{ X \in \bb{R}^{m\times s}: X \tp X = I_s \right\}$& $X \mapsto X\left( \frac{3}{2}I_s - \frac{1}{2} X\tp X \right)$ \cite{xiao2021solving}\\ \hline
			Complex Stiefel manifold & $\left\{ X \in \bb{C}^{m\times s}: X^H X = I_s \right\}$  & $X \mapsto X\left( \frac{3}{2}I_s - \frac{1}{2} X^H X \right)$\\ \hline
			\multirow{2}{*}{Generalized Stiefel manifold } &  $\left\{ X \in \bb{R}^{m\times s}: X \tp B X = I_s \right\}$ for some  & \multirow{2}{*}{$X \mapsto X\left( \frac{3}{2}I_s - \frac{1}{2} X\tp B X \right)$  } \\
			&  positive definite $B$ & \\ \hline
			Grassmann manifold &$\left\{ \mathrm{range}(X): X \in \bb{R}^{m\times s}, X \tp X = I_s \right\}$& $X \mapsto X\left( \frac{3}{2}I_s - \frac{1}{2} X\tp X \right)$ \cite{xiao2021solving}\\ \hline
			Complex Grassmann manifold & $\left\{ \mathrm{range}(X): X \in \bb{C}^{m\times s}, X^H X = I_s \right\}$  & $X \mapsto X\left( \frac{3}{2}I_s - \frac{1}{2} X^H X \right)$\\ \hline
			\multirow{2}{*}{Generalized Grassmann manifold } &  $\left\{ \mathrm{range}(X): X \in \bb{R}^{m\times s}, X \tp B X = I_s \right\}$   & \multirow{2}{*}{$X \mapsto X\left( \frac{3}{2}I_s - \frac{1}{2} X\tp B X \right)$  } \\
			& for some positive definite $B$ & \\ \hline
			\multirow{2}{*}{Hyperbolic manifold \cite{bai2014minimization} } &  $\left\{ X \in \bb{R}^{m\times s}: X \tp B X = I_s \right\}$ for some $B$ & \multirow{2}{*}{$X \mapsto X\left( \frac{3}{2}I_s - \frac{1}{2} X\tp B X \right)$  } \\
			& that satisfies $\lambda_{\min}(B) < 0 < \lambda_{\max}(B)$ & \\ \hline
			\multirow{2}{*}{Symplectic Stiefel manifold \cite{son2021symplectic}} &  $\left\{ X \in \bb{R}^{2m\times 2s}: X \tp Q_m X = Q_s \right\}$ & \multirow{2}{*}{$X \mapsto X \left(\frac{3}{2}I_{2s} + \frac{1}{2}Q_sX\tp Q_mX \right)$}\\  
			&$Q_m := \left[ \begin{smallmatrix}
				{\bf 0}_{m\times m} & I_m\\
				-I_m & {\bf 0}_{m\times m}
			\end{smallmatrix}\right]$  & \\ \hline
			\multirow{2}{*}{Quadratic matrix Lie groups \cite{zhang2020riemannian}} &  $\left\{ X \in \bb{R}^{m\times m}:  X\tp R_m X = R_m \right\}$ & \multirow{2}{*}{$X \mapsto X  - \frac{1}{2}X \left(  X\tp R_m\tp  X R_m - I_m \right)$}\\  
			&$R_m^2 = \nu I_m,  R_m\tp = \nu R_m$ for $\nu = \pm 1$ & \\ \hline
		\end{tabular}
		
	\end{table}

		\subsection{Choosing penalty parameter $\beta$ for \ref{Prob_Pen}}
		In this subsection, we discuss how to choose the penalty parameter for \ref{Prob_Pen}. The following propositions illustrate that for any local minimizer $x$ of \ref{Prob_Ori}, we can choose a sufficiently large penalty parameter $\beta$ to guarantee that $x$ is also a local minimizer for \ref{Prob_Pen} through a practical formula. 
		
		\begin{prop}
			\label{Prop_finite_estimate_beta}
			For any $x \in \M$, it holds that
			\begin{equation*}
				\sup_{y \in \Omegax{x}\setminus \M}\left\{  \max\left\{\frac{f(\A^2(y)) - f(\A(y))}{\norm{c(y)}^2 - \norm{c(\A(y))}^2  }, 0  \right\} \right\} \leq \frac{32\Lf (\Ma +1)\Lac }{3\Lsc ^3}.
			\end{equation*}
		\end{prop}
		\begin{proof}
		
		For any $x \in \M$ and  any $y \in \Omegax{x}$, it holds directly  from Lemma \ref{Le_A_secondorder_descrease}  that 
		\begin{equation*}
			\norm{c(\A(\y))} \leq \frac{4\Lac }{\Lsc ^2}\norm{c(\y)}^2 \leq \frac{1}{2} \norm{c(y)},
		\end{equation*}
		which leads to 
		\begin{equation}
			\label{Eq_Prop_finite_estimate_beta_0}
			\norm{c(y)}^2 - \norm{c(\A(y))}^2 \geq \frac{3}{4} \norm{c(y)}^2. 
		\end{equation}
		
		On the other hand, it follows from Lemma \ref{Le_Ax_c} and  Lemma \ref{Le_A_secondorder_descrease}  that
		\begin{equation}
			\label{Eq_Prop_finite_estimate_beta_1}
			\begin{aligned}
				&|f(\A^{2}(\y)) - f(\A(\y))| 
				\leq \Lf  \norm{\A^{2}(\y) - \A(\y)} \leq \frac{2\Lf (\Ma +1)}{\Lsc }\norm{c(\A(\y))}\\
				\leq{}&  \frac{8\Lf (\Ma +1)\Lac }{\Lsc ^3} \norm{c(\y)}^2.
			\end{aligned}
		\end{equation}
		
		Therefore, for any $y \in \Omegax{x} \setminus \M$,  from \eqref{Eq_Prop_finite_estimate_beta_0} and \eqref{Eq_Prop_finite_estimate_beta_1} we get
		\begin{equation*}
			\left|\frac{f(\A^2(y)) - f(\A(y))}{\norm{c(y)}^2 - \norm{c(\A(y))}^2  }\right| \leq \frac{\frac{8\Lf (\Ma +1)\Lac }{\Lsc ^3} \norm{c(\y)}^2}{\frac{3}{4} \norm{c(y)}^2} \leq \frac{32\Lf (\Ma +1)\Lac }{3\Lsc ^3},
		\end{equation*}
		which illustrates that $\sup\limits_{y \in \Omegax{x}\setminus \M}\left\{  \max\left\{\frac{f(\A^2(y)) - f(\A(y))}{\norm{c(y)}^2 - \norm{c(\A(y))}^2  }, 0  \right\} \right\} \leq \frac{32\Lf (\Ma +1)\Lac }{3\Lsc ^3}$ and  completes the proof. 
		\end{proof}
		
		\medskip
		\begin{prop}
			\label{Prop_esti_beta}
			For any local minimizer $x$ of \ref{Prob_Ori}, suppose the penalty parameter $\beta$ in \ref{Prob_Pen} satisfies 
			\begin{equation}
				\label{Eq_Prop_esti_beta_2}
				\beta \geq 2\cdot \sup_{y \in \Omegax{x}\setminus \M}\left\{  \max\left\{\frac{f(\A^2(y)) - f(\A(y))}{\norm{c(y)}^2 - \norm{c(\A(y))}^2  }, 0  \right\} \right\},
			\end{equation}
			then it holds that $x$ is a local minimizer for \ref{Prob_Pen}. 
		\end{prop}
		\begin{proof} We prove this proposition by contradiction. Suppose $x$ is not a local minimizer for \ref{Prob_Pen}. Then there exists a sequence $\{y_i\} \subset \BOmegax{x}$ such that $\{y_i\} \to x$ and $h(y_i) < h(x)$ holds for any $i \geq 1$.

		From \eqref{Eq_Prop_esti_beta_2},  it is easy to verify that the following inequality holds for any $y \in \Omegax{x}$, 
		\begin{equation}
			\label{Eq_Prop_esti_beta_0}
			h(\A(y)) = f(\A^2(y)) + \frac{\beta }{2} \norm{c(\A(y))}^2 \leq f(\A(y)) + \frac{\beta}{2}\norm{c(y)}^2  =h(y). 
		\end{equation}
		Notice that Lemma \ref{Le_dist_Ainfty} implies that $\A^k(y) \in \Omegax{x}$ holds for any $y \in \BOmegax{x}$ and any $k \geq 1$. Therefore, together with \eqref{Eq_Prop_esti_beta_0}, we can conclude that for any $y \in \BOmegax{x}$, it holds that 
		\begin{equation}
			\label{Eq_Prop_esti_beta_1}
			f(\A^{\infty}(y)) = h(\A^{\infty}(y)) \leq h(\A(y)) \leq h(y). 
		\end{equation}

		From \eqref{Eq_Prop_esti_beta_1} and Lemma \ref{Le_dist_Ainfty}, it holds that $\{\A^{\infty}(y_i)\} \to x$ and  for any $i \geq 1$, $f(\A^{\infty}(y_i)) \leq h(y_i) < h(x)$, 
		which contradicts the fact that $x$ is a local minimizer of \ref{Prob_Ori}. Therefore, from the contradiction, we can conclude that $x$ is a local minimizer for \ref{Prob_Pen} and complete the proof. 
		\end{proof}

		\medskip
		\begin{rmk}
			\label{Rmk_esti_beta}
			As illustrated in Proposition \ref{Prop_finite_estimate_beta},  the right-hand-side of \eqref{Eq_Prop_esti_beta_2} is upper-bounded by a constant. Moreover, based on Proposition \ref{Prop_esti_beta}, we suggest the following procedure to choose a penalty parameter for \ref{Prob_Pen}. We first choose a reference point $\tilde{x} \in \M$ and randomly sample $N_{\beta}$ points  $\{x^{ref}_1,...,x^{ref}_{N_{\beta}}\} \in \ca{B}(\tilde{x}, \delta_{\beta})$  where $\ca{B}(z, \delta_{\beta}):= \{ x \in \bb{R}^n: \norm{x - z} \leq \delta_{\beta} \}$. Then we compute an estimated value for $\beta$ as follows:
			\begin{equation}
				\label{Eq_esti_beta}
				\beta_{esti} = 2\theta_{\beta} \cdot \max_{1\leq i\leq N_{\beta}}\left\{  \max\left\{\frac{f(\A^2(x^{ref}_i)) - f(\A(x^{ref}_i))}{\left|\norm{c(x^{ref}_i)}^2 -\norm{c(\A(x^{ref}_i))}^2 \right|  + \varepsilon_{\beta} }, 0  \right\} \right\}. 
			\end{equation} 
			Here $\delta_{\beta}>0$, $\varepsilon_{\beta} \geq  0$ and $\theta_{\beta} \geq 1$ are some prefixed hyper-parameters. 
		\end{rmk}

	\subsection{Comparison with existing penalty approaches}
	We summarize the differences between \ref{Prob_Pen} and Fletcher's penalty function \eqref{Penalty_function_Fletcher} in Table \ref{Table_compare}.  As mentioned in the introduction, Fletcher's penalty function involves $\nabla f$ in its function value.
	Therefore, the cost of computing its exact gradient is similar to
	computing the Hessian of  \ref{Prob_Pen}, meanwhile, calculating the
	Hessian of Fletcher's penalty function is usually intractable.

	\begin{table}[htbp]
		\caption{Differences between \ref{Prob_Pen} and Fletcher's penalty function.}
		\label{Table_compare}
		\centering
		\scriptsize
		\begin{tabular}{|lll|}
			\hline
			Objective function $f$ in \ref{Prob_Ori} & \ref{Prob_Pen} & Fletcher's penalty function \eqref{Penalty_function_Fletcher} \\ \hline
			Bounded below           &  Bounded below           &   Not bounded below            \\ 
			Lipschitz smooth 		&  Lipschitz smooth        &   Lipschitz continuous        \\
			Twice differentiable    &  Twice differentiable    &   Differentiable       \\
			$\nabla f$ is available &  $h$ and $\nabla h$ are achievable & Only $\phi$ is achievable\\
			$\nabla f$ and $\nabla^2 f$ are available &  $h$, $\nabla h$ and $\nabla^2 h$ are achievable & Only $\phi$ and $\nabla \phi$ are achievable\\\hline
		\end{tabular}
		
	\end{table}

	On the other hand, the constraint dissolving operator $\A$ that satisfies Assumption \ref{Assumption_2}, its transposed Jacobian $\Ja(x)$ is not necessary symmetric. As a result, from the expression of $\nabla f(x)$ presented in Proposition \ref{Prop_gradient_h}, $\nabla h(x)$ is not necessarily in $\Tx$ even when $x \in \M$. However, from the expression of Fletcher's penalty function in \eqref{Penalty_function_Fletcher}, $\nabla \phi(x) = \grad f(x) \in \Tx$ holds for any given $x \in \M$.

	To further illustrate the differences between \ref{Prob_Pen} and Fletcher's penalty function, we provide an example by considering a problem in $\bb{R}^2$ that minimizes $f(w):= \norm{w-[1, 1]\tp}^2$ over the constraint $w\tp C w = 1$ with $C:= \Diag(1,-1)$. As illustrated  in Section 4.1,  the mapping $\A: w \mapsto w - \frac{1}{2}w(w\tp C w -1)$ satisfies Assumption \ref{Assumption_2} and thus its constraint dissolving function $h(w):=f(w - \frac{1}{2}w(w\tp C w -1) ) + \frac{\beta}{2} (w\tp C w -1)^2$  shares the same first-order stationary points with itself. We plot the contours of $h(w)$, together with the contours of the corresponding Fletcher's penalty function in  Figure \ref{Fig_contour}. 
	These figures illustrate that even when $w$ is feasible, $\nabla h (w)$ is not necessarily contained in the tangent space of the feasible set. Therefore, $\nabla h(x)$ is independent of the Riemannian gradient of $f$ even when 
	$x$ is feasible, which further illustrates that minimizing \ref{Prob_Pen} can completely waive the computation of geometrical materials of the Riemannian manifold $\M$. However, the expression of the corresponding Fletcher's penalty function forces $\nabla \phi(x) = \grad f(x) \in \Tx$ for any given $x \in \M$. That is, for any $x \in \M$, computing $\nabla \phi(x)$ is equivalent to computing the Riemannian gradient of $f$ at $x$. 
	
	\begin{figure}[tbp]
		\caption{The contours of \ref{Prob_Pen} and \eqref{Penalty_function_Fletcher} with $\beta = 1$ for $ w \in  [0,2] \times [-0.5,1.5]$, where a lighter contour corresponds to a higher function value. The red lines denote the feasible set. 
			(a) The contours of $ h (w)$; (b) The contours of Fletcher's penalty function. }
		\label{Fig_contour}
		\centering
		\subfigure[]{
			\begin{minipage}[t]{0.45\linewidth}
				\centering
				\includegraphics[width=\linewidth]{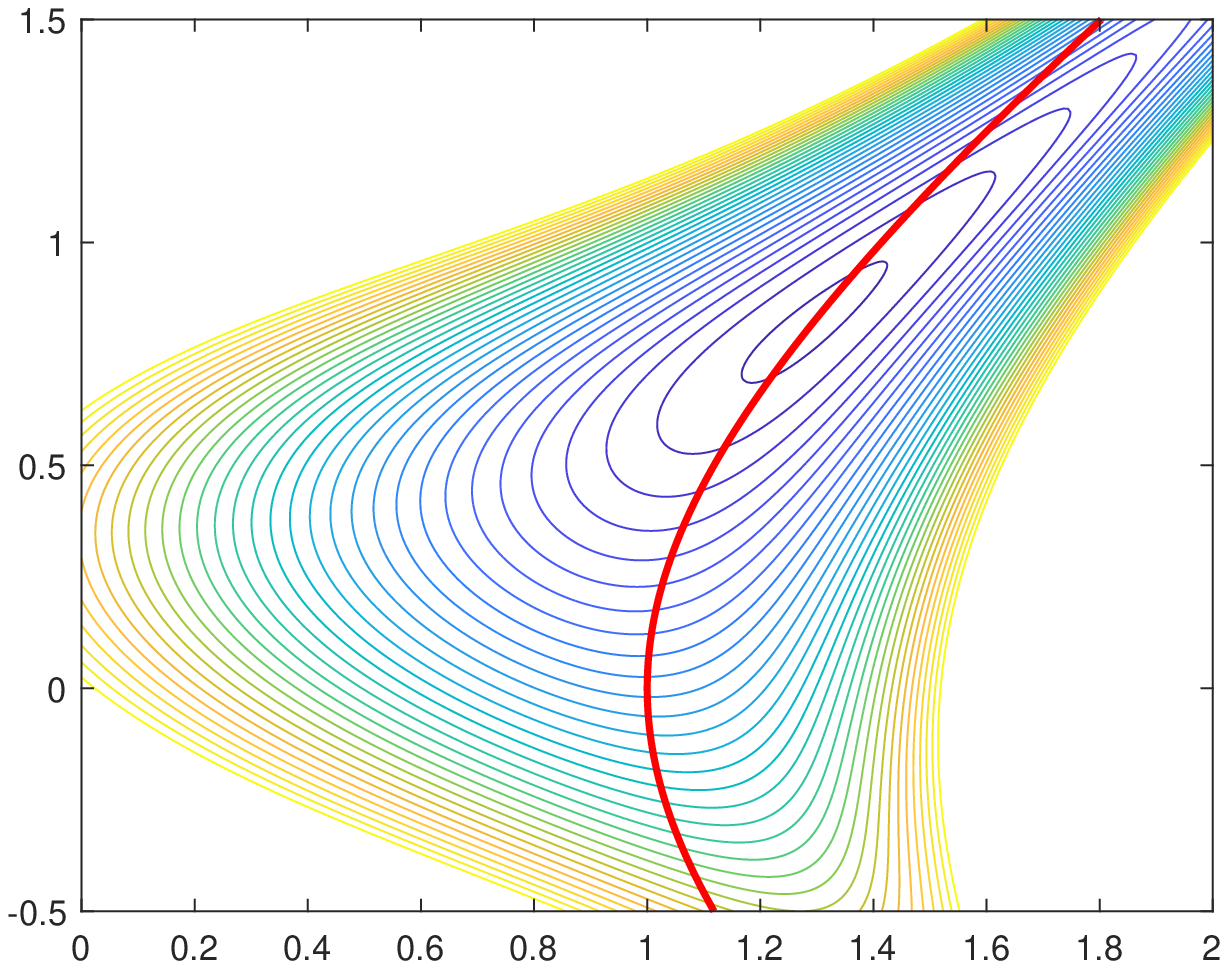}
				\label{Fig_contour_large}
			\end{minipage}%
		}%
		\subfigure[]{
			\begin{minipage}[t]{0.45\linewidth}
				\centering
				\includegraphics[width=\linewidth]{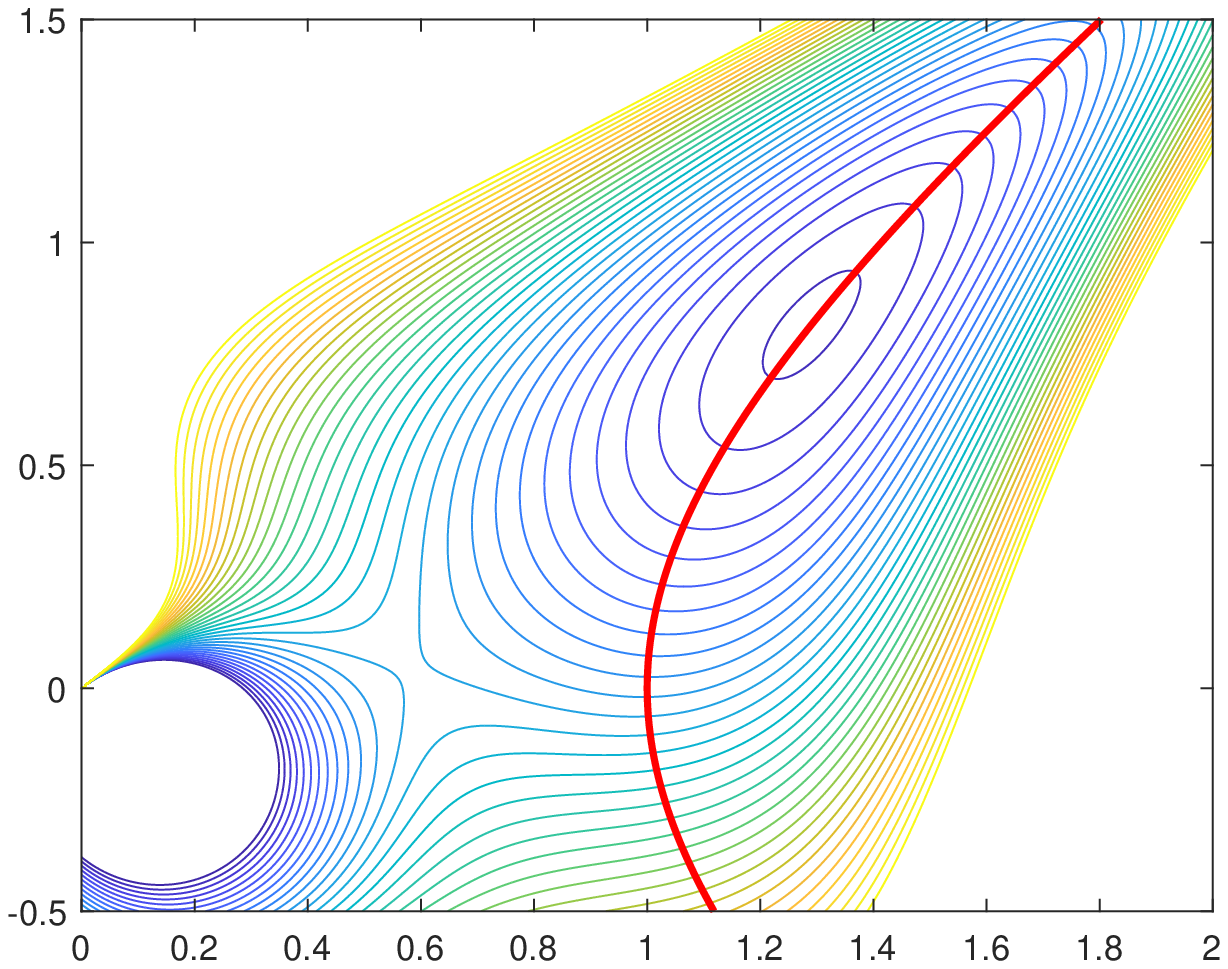}
				\label{Fig_contour_tiny}
			\end{minipage}%
		}%
		
	\end{figure}

	Furthermore, when $f$ is bounded below, it is easy to verify that \ref{Prob_Pen} is bounded below in $\bb{R}^n$. Meanwhile, Fletcher's penalty function does not have this property. As illustrated in Figure \ref{Fig_contour}, Fletcher's penalty function can be unbounded below in $\bb{R}^n$ even if the objective function is bounded below. Moreover, Fletcher's penalty function is not well defined in $\bb{R}^n$ since the Jacobian of the constraint is singular at $[0,0]\tp$.

	\subsection{Example}
	In this subsection, we present a representative example to illustrate how to apply \ref{Prob_Pen} to solve \ref{Prob_Ori} by the routine in Algorithm \ref{Alg:Constraint_Dissloving}, and that it  inherits all the convergence properties from  the selected unconstrained optimization approach. Moreover, we present several supplementary examples in Appendix \ref{Section_Supplementary_Examples}. These supplementary examples further illustrate that the proposed constraint dissolving approaches enable us to directly employ various existing efficient unconstrained solvers to \ref{Prob_Pen}.
	
	Before we start the proof, we first define several constants based on Assumption \ref{Assumption_3} for any given $x_0 \in \M$ and any $\zeta > 0$:

	\begin{itemize}
		\item $\tLsc:= \inf_{x \in \Gamma_{x_0, \zeta}\cap \bar{\Omega}} \Lsc$;
		\item $\tMa:= \sup_{x \in \Gamma_{x_0, \zeta}\cap \bar{\Omega}} \Ma$;
		\item $\tMc:= \sup_{x \in \Gamma_{x_0, \zeta}\cap \bar{\Omega}} \Mc$;
		\item $\tLc:= \sup_{x \in \Gamma_{x_0, \zeta}\cap \bar{\Omega}} \Lc$;
		\item $\tLa:= \sup_{x \in \Gamma_{x_0, \zeta}\cap \bar{\Omega}} \La$.
	\end{itemize}

	Recently, there is a growing interest in designing algorithms that can escape from saddle points in unconstrained nonconvex optimization. Among these approaches, the cubic regularization Newton’s method is a popular optimization algorithm. Recently, \cite{wang2020cubic} proposed a cubic regularization method with momentum (CRm), which achieves the best possible convergence rate to a second-order stationary point for nonconvex optimization. However, transferring CRm into its Riemannian version by the framework from \cite{Absil2009optimization} requires deep modifications to the original framework, since solving the cubic step in tangent space requires specially designed solvers, computing the momentum step involves vector transports on the Riemannian manifold, and retractions should also be introduced to enforce the feasibility of the iterates. Noting that the iterates are not updated along geodesics, and the momentum steps involve vector transports, we need great efforts in establishing the convergence properties for the Riemannian version of CRm. 
	
	Alternatively, we can directly apply CRm algorithm to solve \ref{Prob_Ori} through \ref{Prob_Pen}. The detailed algorithm is presented in Algorithm \ref{Alg:CRM}.

	\begin{algorithm}[htbp]
		\begin{algorithmic}[1]   
			\Require Input data:  functions $f$, $\rho < 1$, $\nu$.
			\State Choose initial guess $x_0 \in \M$, and $\beta \geq \bar{\beta}_{x_0}$ according to Definition \ref{Cond_beta_global}, set $y_0 = x_0$, $k:=0$.
			\While{not terminated}
			\State Compute cubic step:
			\begin{equation}
				\dk = \mathop{\arg\min}_{d \in \bb{R}^{n}}~ d\tp \nabla h(\xk) + \frac{1}{2} d\tp \nabla^2 h(\xk)d + \frac{\nu}{6} \norm{d}^3.
			\end{equation}
			\State $\ykp = \xk + \dk$, $\tau_k = \min\{ \rho, \norm{\nabla h(\ykp)}, \norm{\ykp-\xk}\}$.
			\State Compute momentum step: $	v_{k+1} = \ykp + \tau_k (\ykp - \yk)$. 
			\State Set $\xkp = \ykp$ if $h(\ykp) \leq h(v_{k+1})$; otherwise, set $\xkp = v_{k+1}$.
			\State $k = k+1$.
			\EndWhile
			\State Return $\xk$.
		\end{algorithmic}  
		\caption{Cubic regularization method with momentum for \ref{Prob_Pen}.}  
		\label{Alg:CRM}
	\end{algorithm}

	Next, we establish the convergence results of Algorithm \ref{Alg:CRM} by combining Theorem 1 in \cite{wang2020cubic}, Theorem \ref{The_firstorder_equivalence} and Theorem \ref{The_secondorder_equivalence}.
	\begin{theo}
		
		Suppose Assumption \ref{Assumption_4} and Assumption \ref{Assumption_3} hold,  $\nabla^2 g(x)$ is locally Lipschitz continuous in $\bb{R}^{n}$, Algorithm \ref{Alg:CRM} sets $\zeta > 0$, chooses its parameters as
		\begin{align*}
			&\tilde{M} :=  \sup_{y,z \in \Gamma_{x_0, \zeta}, y\neq z } \frac{\norm{\nabla^2 g(y) - \nabla^2 g(z)}}{\norm{y - z}},\quad \beta \geq \max\left\{\bar{\beta}_{x_0}, ~\frac{96\tilde{M}(\tMa +1)^2}{\tLsc^2 } \right\},\\
			&\nu =  \max\left\{ \sup_{y,z \in \Gamma_{x_0, \zeta}, y\neq z } \frac{\norm{\nabla^2 h(y) - \nabla^2 h(z)}}{\norm{y - z}},  \sup_{y \in \Gamma_{x_0, \zeta} } \frac{2\norm{\nabla^2 h(y)}}{\zeta}, \sup_{y \in \Gamma_{x_0, \zeta} }  \frac{\norm{\nabla h(y)}}{\zeta^2} \right\},
		\end{align*}
		and  produces iterates $\{\xk\}$.   Then for any $\varepsilon\in (0,1)$, there exists an constant $C$ that is dependent on $\beta$, $\Lf $, $\Lg$ and $\nu$, such that for any $K \geq \frac{C}{\varepsilon^{3/2}}$,  
		there exists an $\tilde{x} \in \{ x_i \}_{0\leq i\leq K}$ such that
		\begin{equation*}
			\norm{\grad f(\A^{\infty}(\tilde{x})) } \leq \varepsilon, \quad \text{ and } \quad 
			\lambda_{\min}(\hess f(\A^{\infty}(\tilde{x})))\geq -\sqrt{\varepsilon}.
		\end{equation*}
		
	\end{theo}
	\begin{proof}
	We first conclude from the definition for $\nu$ that $\norm{x_{k+1} - x_k} \leq \zeta$ holds for any $k \geq 0$. 
	Due to step 7 in Algorithm \ref{Alg:CRM}, the sequence $\{\xk\}$ generated by Algorithm \ref{Alg:CRM} satisfies $h(x_{j+1}) \leq h(x_j), j = 0,1,...$, which leads to the fact that $\{\xk\} \subset \Gamma_{x_0, \zeta}$. Then  Corollary \ref{Coro_barrier} ensures that the sequence is restricted in $\bar{\Omega} \cap \Gamma_{x_0, \zeta}$, which implies the validity of Assumption 1 in \cite{wang2020cubic}. Together with Theorem 1 in \cite{wang2020cubic}, we can conclude that for any $\varepsilon\in (0,1)$ and any $K \geq \frac{C}{\varepsilon^{3/2}}$, there exists an $\tilde{x} \in \{ x_i \}_{0\leq i\leq K}$ satisfying
	\begin{equation*}
		\norm{ \nabla h(\tilde{x}) } \leq    \frac{\varepsilon}{2+  8\tLa + \frac{288 \tLc\tMc(\tMa + 1)^2}{ \tLsc^2}}, \quad \text{and} \quad \lambda_{\min}( \nabla^2 h(\tilde{x}) ) \geq -\frac{\sqrt{\varepsilon}}{3}.
	\end{equation*}
	
	By Proposition \ref{Prop_upperbound_projection_gradient},  the relationship $ \norm{\grad f(\A^{\infty}(\tilde{x})) } \leq 2\norm{\nabla h(\tilde{x})}  \leq \varepsilon$ holds.
	In addition, it follows from Theorem \ref{The_firstorder_equivalence} and Lemma \ref{Le_dist_Ainfty} that
	\begin{equation*}
		\norm{\A^{\infty}(\tilde{x}) - \tilde{x}} \leq \frac{4(\tMa+1)}{\tLsc}\norm{c(\tilde{x})} \leq \frac{32(\tMa+1)^2}{\beta \tLsc^2} \norm{\nabla h(\tilde{x})} \leq  \frac{1}{3\tilde{M}}\varepsilon. 
	\end{equation*}
	
	Finally, we recall Proposition \ref{Prop_esti_lambda_min} and arrive at
	\begin{equation*}
		\begin{aligned}
			&\lambda_{\min}(\hess f(\A^{\infty}(\tilde{x})))  \\
			\geq{}& \lambda_{\min}( \nabla^2 h(\tilde{x})) - \norm{\nabla^2 g(\A^{\infty}(\tilde{x})) - \nabla^2 g(\tilde{x})}- \left(\frac{5}{2}\tLa + \frac{96 \tLc\tMc(\tMa + 1)^2}{ \tLsc^2} \right) \norm{\nabla h(\tilde{x})}\\
			\geq{}& \lambda_{\min}( \nabla^2 h(\tilde{x}))  - \tilde{M}\norm{\A^{\infty}(\tilde{x}) - \tilde{x}} - \left(\frac{5}{2}\tLa + \frac{96 \tLc\tMc(\tMa + 1)^2}{ \tLsc^2} \right) \norm{\nabla h(\tilde{x})}\\
			\geq{}& -\frac{\sqrt{\varepsilon}}{3} - \frac{\varepsilon}{3}- \frac{\varepsilon}{3} = -\sqrt{\varepsilon}, 
		\end{aligned}
	\end{equation*}
	which completes the proof.
	\end{proof}

	\section{Conclusion}
	Riemannian optimization has close connections with unconstrained optimization. To extend existing unconstrained optimization approaches to solve Riemannian optimization problems and establish the corresponding theoretical properties, most existing approaches are developed based on the frameworks summarized in \cite{Absil2009optimization} by utilizing various geometrical materials from differential geometry. However, determining and preparing the geometrical materials are challenging for various Riemannian manifolds. In addition,
	incorporating
	these geometrical materials requires significant modifications to the original unconstrained optimization approaches. Last but not least, the approximation errors introduced by retractions and vector transports generally will lead to difficulties in establishing the theoretical 
	convergence properties. Therefore,
	it is challenging to apply advanced unconstrained optimization approaches to solve Riemannian optimization problems. 
	
	The main contribution of this paper is to propose a class of constraint dissolving approaches, based on the so-called constraint dissolving functions \ref{Prob_Pen}.
	We prove that under mild assumptions, \ref{Prob_Ori} and \ref{Prob_Pen} have the same first-order stationary points, second-order stationary points, local minimizers, and {\L}ojasiewicz exponents in a neighborhood of the feasible region. In addition, the exact gradient and Hessian of \ref{Prob_Pen} can directly be calculated based on the same order of differentials of $f$.
	We summarize our proposed constraint dissolving approaches in Algorithm \ref{Alg:Constraint_Dissloving}.
	We provide a framework to establish the global convergence, worst-case complexity, and
	escaping-from-saddle-point properties of Algorithm \ref{Alg:Constraint_Dissloving} under mild assumptions
	directly based on existing results for the selected unconstrained optimization approach.
	
	Moreover, we discuss how to choose the constraint dissolving operator $\A$ for \ref{Prob_Pen} and present an easy-to-compute form of $\A$ for several well-known manifolds, including the generalized Stiefel manifold, the symplectic Stiefel manifold, and the hyperbolic manifold. The construction of \ref{Prob_Pen} is independent of the geometrical properties of $\M$, which avoids the difficulties in analyzing the geometrical properties of the underlying manifold and hence it enables us to design constraint dissolving approaches for a number of well-known Riemannian manifolds. Finally, we use the cubic regularization method with momentum as an example to illustrate how to directly apply unconstrained optimization approaches to \ref{Prob_Ori} and inherit existing theoretical results.

	\appendix

		\section{Proofs for Main Results}
		\label{Section_proofs}

		\subsection{Proofs for Section \ref{Subsection_Theo_properties_A}}
		\label{Section_proofs_311}
		\paragraph{Proof for Lemma \ref{Le_dist_Ainfty}}
		\begin{proof}
		We first show that the inclusion $\A^j(\y)\in\Omegax{x}$ holds
		for any $j\in\mathbb{N}\cup \{0\}$ by mathematical induction. 
		It is clear that this statement holds at $j=0$. Suppose that
		the statement holds for any  $0\leq j \leq K-1$ with certain
		$K\in\mathbb{N}$. Then  for any $1\leq k\leq K-1$, Lemma \ref{Le_bound_cx}, Lemma \ref{Le_A_secondorder_descrease} and the definition of $\varepsilon_x$ imply that
		\begin{equation}
			\label{Eq_Le_dist_Ainfty_0}
			\norm{c(\A^k(\y))} \leq \frac{4\Lac }{\Lsc ^2} \norm{c(\A^{k-1}(\y))}^2 \leq \frac{1}{2} \norm{c(\A^{k-1}(\y))}.
		\end{equation}
		By simple calculations, we obtain
		\begin{equation}\label{add:1}
			\sum_{i=0}^{K} \norm{c(\A^{i}(\y))} < 2\norm{c(\y)}.
		\end{equation}
		On the other hand, it follows from Lemma \ref{Le_Ax_c} that
		\begin{equation*}
			\norm{\A^{k}(\y) - \A^{k-1}(\y)} \leq \frac{2(\Ma +1)}{\Lsc } \norm{c(\A^{k-1}(\y))},
		\end{equation*}
		which implies that
		\begin{equation}
			\label{Eq_Le_dist_Ainfty_1}
			\begin{aligned}
				&\sum_{i=1}^K \norm{\A^{i}(\y) - \A^{i-1}(\y)} \leq \sum_{i=1}^{K}\frac{2(\Ma +1)}{\Lsc } \norm{c(\A^{i-1}(\y))} \\
				\leq{}& \frac{4(\Ma +1)}{\Lsc } \norm{c(\y)} \leq \frac{4\Mc(\Ma +1)}{\Lsc } \norm{y-x}. 
			\end{aligned}
		\end{equation}
		Together with the fact that $\y \in \BOmegax{x}$, it holds that 
		\begin{equation*}
			\begin{aligned}
				&\norm{\A^K(y) - x} \leq \norm{y-x} + \sum_{i=1}^K \norm{\A^{i}(\y) - \A^{i-1}(\y)} \\
				\leq{}& \frac{4\Mc(\Ma +1) + \Lsc}{\Lsc } \norm{y-x} \leq \varepsilon_x. 
			\end{aligned}
		\end{equation*}
		Namely, the inclusion $\A^j(\y)\in\Omegax{x}$ holds at $j=K$. By mathematical induction, we can conclude that this statement holds for any $j=\mathbb{N}\cup \{0\}$.
		
		Now, it is easy to extend inequalities \eqref{add:1} and 
		\eqref{Eq_Le_dist_Ainfty_1} to the following infinite case, leading to
		\begin{equation*}
			\sum_{i=0}^{+\infty} \norm{c(\A^{i}(\y))} \leq \sum_{i=0}^{+\infty} \frac{1}{2^i} \norm{c(\y)} = 2\norm{c(\y)},
		\end{equation*}
		and
		\begin{equation}
			\sum_{i=1}^{+\infty} \norm{\A^{i}(\y) - \A^{i-1}(\y)} \leq \frac{4(\Ma +1)}{\Lsc } \norm{c(\y)}.  
		\end{equation}
		Then by the dominated convergence theorem, we have that  the sequence $\{ \A^k(x) \}$ is convergent, i.e. $\A^{\infty}(\y)$ exists.  Moreover,
		\begin{equation}
			\norm{\A^{\infty}(\y) - \y} \leq \sum_{k=1}^{\infty} \norm{ \A^{k}(\y) - \A^{k-1}(\y) } \leq  \frac{4(\Ma +1)}{\Lsc } \norm{c(\y)}.
		\end{equation}
		
		Finally, we show that $\A^{\infty}(y) \in \M$. From \eqref{Eq_Le_dist_Ainfty_0}, $\norm{c(\A^k(\y))}  \leq \frac{1}{2} \norm{c(\A^{k-1}(\y))}$ holds for any $k\geq 1$. Then 
		\begin{equation*}
			\norm{c(\A^{\infty}(y))} = \lim\limits_{k \to +\infty} \norm{c(\A^k(\y))} = 0.
		\end{equation*}
		This implies that $\A^{\infty}(y) \in \M$ and the proof is completed.
		\end{proof}

		\subsection{Proofs for Section \ref{Subsection_Theo_properties_CDF}}
		\subsubsection{Proofs for Section \ref{Subsubsection_Relationship_stationary}}
		\label{Section_proofs_321}
		\paragraph{Proof for Proposition \ref{Prop_Firstorder_Equivalence_feasible}}
		\begin{proof}
		For any first-order stationary point, $x \in \M$ of \ref{Prob_Ori}, it follows from the equality
		\eqref{add:2} that
		\begin{equation*}
			\nabla h(x) = \Ja(x) \nabla f(x) + \beta \Jc(x) c(x) = \Ja(x)\Jc(x) \lambda(x) = 0,
		\end{equation*}
		which implies that $x$ is a first-order stationary point of \ref{Prob_Pen}. 
		
		On the other hand, for any $x \in\M$ that is a first-order stationary point of \ref{Prob_Pen}, it holds that
		\begin{equation*}
			\Ja(x) \nabla f(x) = 0,
		\end{equation*}
		which results in the inclusion $\nabla f(x) \in \Nx =  \mathrm{range}(\Jc(x))$ from Lemma \ref{Le_Ja_nullspace}. By Definition \ref{Defin_FOSP}, we conclude that $x$ is a first-order stationary point of \ref{Prob_Ori}. 
		\end{proof}

		\paragraph{Proof for Theorem \ref{The_firstorder_equivalence}}
		\begin{proof}
		For any $\y \in \Omegax{x}$, it follows from the definition of $\Ma$ that
		\begin{equation*}
			\norm{(\Ja(\y) - I_n)\nabla h(\y)} \leq (\Ma + 1)\norm{\nabla h(\y)}. 
		\end{equation*}
		Moreover, we can obtain
		\begin{equation*}
			\begin{aligned}
				&\norm{(\Ja(\y) - I_n)  \nabla g(\y)} =  \norm{(\Ja(\y) - I_n) \Ja(\y)\nabla f(\A(\y)) }
				\leq{} \norm{\Ja(\y)^2 - \Ja(\y)} \norm{\nabla f(\A(\y))} \\
				\leq{}& \La (2\Ma  + 1) \mathrm{dist}(y, \M) \norm{\nabla f(\A(\y))} 
				\leq{} \frac{2\Lf \La (2\Ma  + 1)}{\Lsc } \norm{c(x)}.
			\end{aligned}
		\end{equation*}
		Here, the second inequality results from the Lipschitz continuity of $\Ja(\y)$ and  Lemma \ref{Le_Ja_idempotent}. The last inequality is implied by Lemma \ref{Le_bound_cx}. On the other hand, it holds that 
		\begin{equation*}
			\begin{aligned}
				&\norm{(\Ja(\y) - I_n)\Jc(\y) c(\y)} \geq \norm{\Jc(\y)c(\y)} - \norm{\Ja(\y)\Jc(\y)c(\y)} \\
				\geq{}& \frac{\Lsc }{2}\norm{c(\y)} - \norm{\Ja(\y)\Jc(\y)} \norm{c(\y)} \geq \frac{\Lsc }{2}\norm{c(\y)} - \Lac \mathrm{dist}(y, \M) \norm{c(\y)} \\
				\geq{}& \frac{\Lsc }{2}\norm{c(\y)} - \Lac \varepsilon_x \norm{c(\y)} \geq  \frac{\Lsc }{4}\norm{c(\y)}.
			\end{aligned}
		\end{equation*}
		
		Combining the above two inequalities, we have 
		\begin{equation}\label{add:3}
			\begin{aligned}
				&\norm{\nabla h(\y)} \geq \frac{1}{\Ma + 1}\norm{( \Ja(\y) - I_n )\nabla h(\y)} \\
				\geq{}& \frac{1}{\Ma + 1}  \Big(\beta \norm{(\Ja(\y) - I_n)\Jc(\y) c(\y)}  -  \norm{(\Ja(\y) - I_n) \nabla g(\y)}\Big) \\
				\geq{}& \frac{1}{\Ma + 1}\left( \frac{\beta \Lsc }{4} - \frac{2\Lf \La (2\Ma  + 1)}{\Lsc } \right)\norm{c(\y)}
				\geq{} \frac{\beta \Lsc}{8(\Ma + 1)} \norm{c(\y)}
			\end{aligned}
		\end{equation}
		holds for any $\y \in \BOmegax{x}$.
		Here, the last inequality results from Definition \ref{Cond_beta}. 
		
		Finally, for any first-order stationary point $x^*$ of \ref{Prob_Pen} that satisfies $x^*\in  \Omegax{x}$, it follows from the inequality \eqref{add:3} that
		the feasibility $\norm{c(x^*)} = 0$ is implied
		by the stationarity $\nabla h(x^*) = 0$. Together with Proposition  \ref{Prop_Firstorder_Equivalence_feasible}, we conclude that $x^*$ is a first-order stationary point of \ref{Prob_Ori}.
		\end{proof}

		\medskip

		In the rest of this part, we aim to prove Theorem \ref{The_secondorder_equivalence}. We start with two auxiliary lemmas. 
		\begin{lem}
			\label{Le_esti_hess_A}
			Suppose Assumption \ref{Assumption_4} holds. For any given $x \in \M$, the following equation holds for any $d \in \Tx$,
			\begin{equation*}
				\left(\DJa(x)[\grad f(x)]\right) d = \left(\DJa(x)[\nabla f(x)]\right) d+ \sum_{i=1}^p \lambda_i(x)  \Ja(x)  \nabla^2 c_i(x) \Ja(x)\tp d.  
			\end{equation*}
			Moreover, if $x$ is a first-order stationary point of \ref{Prob_Ori}, then for any $d \in \Tx$, it holds that
			\begin{equation*}
				\left(\DJa(x)[\nabla f(x)]\right) d = - \sum_{i=1}^p \lambda_i(x) \Ja(x) \nabla^2 c_i(x) \Ja(x)\tp d.
			\end{equation*}
		\end{lem}
		\begin{proof}
		Firstly, for any given $x\in\M$, it follows from Assumption \ref{Assumption_2} that the equality $\Ja(x)\Jc(\A(x)) = \Ja(x)\Jc(x) = 0$ holds. Then for any $d \in \Tx$,  we denote
		$\hat{x} \in \mathrm{proj}(x+d, \M)$ and obtain
		\begin{equation*}
			\begin{aligned}
				&\norm{\Ja(x+d) \Jc(\A(x+d)) }=  \norm{\Ja(x+d) \Jc(\A(x+d)) - \Ja(\hat{x})\Jc(\A(\hat{x})) } \\
				\leq {} & \Lac \norm{x+d-\hat{x}} =  \Lac \mathrm{dist}(x+d, \M)
				\overset{(i)}{\leq} \frac{2\Lac}{\Lsc} \norm{c(x+d)} 
				\overset{(ii)}{\leq}{} \frac{\Lac \Lc}{\Lsc} \norm{d}^2.
			\end{aligned}
		\end{equation*}
		Here, the inequality $(i)$ directly follows from Lemma 3.1. Meanwhile, combining the second-order Taylor expansion of  $c(x + d)$ at $x$ with the facts
		that $x\in\M$ and $d \in \Tx$, we can obtain the inequality $(ii)$.
		
		Therefore, it follows from the fact $\nabla (c\circ\A) (x)=\Ja(x)\Jc(\A(x))$ that 
		\begin{equation*}
			\lim\limits_{t \to 0}~ \frac{ \Ja(x+ td) \nabla c_i(\A(x+td)) - \Ja(x) \nabla c_i(\A(x)) }{t} = \lim\limits_{t \to 0}  ~\frac{ \Ja(x+ td) \nabla c_i(\A(x+td))  }{t}= 0
		\end{equation*}
		holds for any $i=1,...,p$ and $d \in \Tx$, which
		further implies $\nabla^2 (c_i\circ\A)(x) d = 0$. 
		Here $\nabla^2 (c_i\circ\A) (x)$ denotes the Hessian of $c_i(\A(x))$ with respect to $x$. 
		
		On the other hand, we notice that
		\begin{equation*}
			\nabla^2 (c_i\circ\A)(x) = \Ja(x) \nabla^2 c_i(\A(x)) \Ja(x)\tp + \DJa(x)[\nabla c_i(\A(x))].
		\end{equation*} 
		Combining Definition \ref{Defin_FOSP} with the equality \eqref{add:4} and $\A(x) = x$, we have
		\begin{equation*}
			\begin{aligned}
				&0 = \sum_{i=1}^p \lambda_i(x) \nabla^2 (c_i\circ\A)(x) d \\
				={}& \left(\DJa(x)[\sum_{i=1}^p \lambda_i(x) \nabla c_i(x)]\right) d + \sum_{i=1}^p \lambda_i(x)  \Ja(x)  \nabla^2 c_i(x)\Ja(x)\tp d\\
				={}& \left(\DJa(x)[\nabla f(x) - \grad f(x)]\right) d + \sum_{i=1}^p \lambda_i(x)  \Ja(x)  \nabla^2 c_i(x) \Ja(x)\tp d.\\
			\end{aligned}
		\end{equation*}
		From here, the first required result follows readily. The second required equality holds because $\grad f(x) = 0$ when $x$ is a stationary point of \ref{Prob_Ori}. 
		\end{proof}

		\medskip
		\begin{lem}
			\label{Le_Hessian_tangent}
			Suppose Assumption \ref{Assumption_4} holds, then for any $x \in \M$ and  $d \in \Tx$, it holds that
			\begin{equation*}
				\nabla^2 h(x)d - \Ja(x) \left( \nabla^2 f(x) - \sum_{i=1}^p \lambda_i(x) \nabla^2 c_i(x) \right) \Ja(x)\tp d =   \left(\DJa(x)[\grad f(x)]\right) d.
			\end{equation*}
			Moreover, if $x$ is a first-order stationary point of \ref{Prob_Ori}, then for any  $d \in \Tx$, it holds that
			\begin{equation*}
				\nabla^2 h(x)d = \Ja(x) \left(\nabla^2 f(x) - \sum_{i=1}^p \lambda_i(x) \nabla^2 c_i(x) \right)\Ja(x)\tp  d.
			\end{equation*}
		\end{lem}
		\begin{proof}
		Firstly, it follows from Proposition \ref{Prop_gradient_h} that $$\nabla^2 h(x) = \Ja(x) \nabla^2 f(\A(x)) \Ja(x)\tp + \DJa(x)[\nabla f(x)] + \beta \Jc(x) \Jc(x)\tp.$$ Together with Lemma \ref{Le_esti_hess_A} and the fact that $\Jc(x)\tp d = 0$, we conclude that 
		\begin{equation*}
			\begin{aligned}
				\nabla^2 h(x)d ={}&  \Ja(x) \nabla^2 f(x) \Ja(x)\tp d + \DJa(x)[\nabla f(x)] d \\
				={}&  \Ja(x) \left( \nabla^2 f(x) - \sum_{i=1}^p \lambda_i(x) \nabla^2 c_i(x) \right) \Ja(x)\tp d + \left(\DJa(x)[\grad f(x)]\right) d,
			\end{aligned}
		\end{equation*}
		and complete the proof. 
		\end{proof}

		Now, we are ready to prove the main theorem.
		\paragraph{Proof for Theorem \ref{The_secondorder_equivalence}}
		\begin{proof}
		
		Firstly, if $x$ is a second-order stationary point of \ref{Prob_Pen}, then Theorem \ref{The_firstorder_equivalence} implies that $x \in \M$. Therefore, we conclude that $\grad f(x) = 0$ and $\lambda_{\min}(\nabla^2 h(x)) \geq 0$. It follows from Lemma \ref{Le_Ja_identical} that  $\Ja(x) \tp d_1 = d_1$ holds for any $d_1 \in \Tx$. Then together with Lemma \ref{Le_Hessian_tangent}, we obtain that
		\begin{equation*}
			d_1\tp \nabla^2 h(x) d_1  = d_1\tp \left( \nabla^2 f(x) - \sum_{i=1}^p \lambda_i(x) \nabla^2 c_i(x) \right) d_1.
		\end{equation*}
		Together with  Lemma \ref{Le_Hessian_tangent} and Proposition \ref{Prop_FOSP_Rie}, we arrive at
		\begin{equation*}
			\lambda_{\min}(\hess f(x)) \geq \min_{d \in \Tx, \norm{d} = 1} ~ d\tp \nabla^2 h(x) d  \geq 0,
		\end{equation*}
		which implies that $x$ is a second-order stationary point of \ref{Prob_Ori}.

		On the other hand, suppose $x$ is a second-order stationary point of \ref{Prob_Ori}, which implies that $\grad f(x) = 0$ and $\hess f(x) \succeq 0$. As a result, we have $\ca{H}(x) \succeq 0$, where $\ca{H}(x)$ is defined in \eqref{Eq_defin_HX}. Then it follows from Lemma \ref{Le_Hessian_tangent} that for any $d_1 \in \Tx$, we have
		\begin{equation*}
			\begin{aligned}
				&d_1\tp \nabla^2 h(x) d_1 = d_1\tp \Ja(x) \left( \nabla^2 f(x) - \sum_{i=1}^p \lambda_i(x) \nabla^2 c_i(x) \right)\Ja(x)\tp d_1\\
				={}& d_1\tp \left( \nabla^2 f(x) - \sum_{i=1}^p \lambda_i(x) \nabla^2 c_i(x) \right) d_1
			\end{aligned}
		\end{equation*}
		Moreover, by \eqref{Eq_Prop_gradient_h_0} in Proposition \ref{Prop_gradient_h}, Lemma \ref{Le_Ja_identical} and Lemma \ref{Le_esti_hess_A}, we have that for $d_1 \in \Tx$, 
		\begin{equation}
			\begin{aligned}
				&d_1\tp \nabla^2 g(x) d_1 = d_1\tp \Ja(x) \left( \nabla^2 f(x) - \sum_{i=1}^p \lambda_i(x) \nabla^2 c_i(x) \right) \Ja(x)\tp d_1\\
				={}& d_1\tp  \left( \nabla^2 f(x) - \sum_{i=1}^p \lambda_i(x) \nabla^2 c_i(x) \right)  d_1.
			\end{aligned}
		\end{equation}
		Thus, 
		\begin{equation}
			\begin{aligned}
				&\lambda_{\max}( \ca{H}(x) ) = \sup_{d_1 \in \Tx, \norm{d_1} = 1}~ d_1\tp  \left( \nabla^2 f(x) - \sum_{i=1}^p \lambda_i(x) \nabla^2 c_i(x) \right)  d_1\\
				={}& \sup_{d_1 \in \Tx, \norm{d_1} = 1}~ d_1\tp \nabla^2 g(x) d_1\leq \Lg. 
			\end{aligned}
		\end{equation}
		As a result, for any  $\beta \geq \beta_{x}$, we obtain that for any $d_1 \in \Nx$, 
		\begin{equation*}
			\begin{aligned}
				&d_2\tp \nabla^2 h(x) d_2 = \beta d_2\tp \Jc(x) \Jc(x)\tp d_2 + d_2\tp \nabla^2 g(x) d_2\\
				\geq{}& \beta \Lsc^2 \norm{d_2}^2 - \Lg \norm{d_2}^2 
				\geq \Lg \Ma^2 \norm{d_2}^2 \\
				\geq{}& \lambda_{\max}(\ca{H}(x)) \norm{\Ja(x)}^2 \norm{d_2}^2 \\
				\geq{}& d_2\tp \Ja(x) \left( \nabla^2 f(x) - \sum_{i=1}^p \lambda_i(x) \nabla^2 c_i(x) \right) \Ja(x)\tp d_2.
			\end{aligned} 
		\end{equation*}
		Note that the last inequality holds because $\Ja(x)\tp d_2 \in \Tx$. 
		
		Moreover, Lemma \ref{Le_Ja_identical} implies that $\Ja(x)\tp(d_1+d_2) \in \Tx$ holds for any $d_1 \in \Tx$ and $d_2 \in \Nx$. Together with Definition \ref{Defin_SOSP} and the fact that $\hess f(x) \succeq 0$, we arrive at
		\begin{equation*}
			(d_1+d_2)\tp \Ja(x) \left( \nabla^2 f(x) - \sum_{i=1}^p \lambda_i(x) \nabla^2 c_i(x) \right) \Ja(x)\tp(d_1+d_2) \geq 0.
		\end{equation*}
		Additionally, it follows from Lemma \ref{Le_Hessian_tangent} that 
		\begin{align*}
			d_2\tp \nabla^2 h(x) d_1 =& d_2\tp \Ja(x) \left( \nabla^2 f(x) - \sum_{i=1}^p \lambda_i(x) \nabla^2 c_i(x) \right) \Ja(x)\tp d_1.
		\end{align*}
		Finally, we can conclude that
		\begin{equation*}
			\begin{aligned}
				0 \leq{}&(d_1+d_2)\tp \Ja(x) \left( \nabla^2 f(x) - \sum_{i=1}^p \lambda_i(x) \nabla^2 c_i(x) \right) \Ja(x)\tp(d_1+d_2) \\
				={}& d_1\tp \nabla^2 h(x) d_1 + 2 d_1\tp \nabla^2 h(x) d_2 +  d_2\tp \Ja(x) \left( \nabla^2 f(x) - \sum_{i=1}^p \lambda_i(x) \nabla^2 c_i(x) \right) \Ja(x)\tp d_2\\
				\leq{}& d_1\tp \nabla^2 h(x) d_1 + 2 d_1\tp \nabla^2 h(x) d_2 + d_2\tp \nabla^2 h(x) d_2\\
				={}& (d_1 + d_2)\tp \nabla^2 h(x) (d_1 + d_2).
			\end{aligned}
		\end{equation*}
		From the arbitrariness of $d_1 \in \Tx$ and $d_2 \in \Nx$, we get $\nabla^2 h(x) \succeq 0$. This complete the proof. 
		\end{proof}

		\subsubsection{Proofs for Section \ref{Subsubsection_Estimation_stationarity}}
		\label{Section_proofs_322}
		\paragraph{Proof for Proposition \ref{Prop_postprocessing}}
		\begin{proof}
		
		It follows from Lemma \ref{Le_Ax_c} and  Lemma \ref{Le_A_secondorder_descrease}  that
		\begin{equation*}
			\begin{aligned}
				&|f(\A^{2}(\y)) - f(\A(\y))| 
				\leq \Lf  \norm{\A^{2}(\y) - \A(\y)} \leq \frac{2\Lf (\Ma +1)}{\Lsc }\norm{c(\A(\y))}\\
				\leq{}&  \frac{8\Lf (\Ma +1)\Lac }{\Lsc ^3} \norm{c(\y)}^2.
			\end{aligned}
		\end{equation*}
		By Lemma \ref{Le_A_secondorder_descrease}, we know that the inequality $\norm{c(\A(y))} \leq \frac{1}{2} \norm{c(y)}$ holds for any $y \in \Omegax{x}$. As a result, 
		\begin{equation}
			\label{Eq_Prop_postprocessing_1}
			\begin{aligned}
				&h(\A(\y)) - h(\y) \leq \left|f(\A^{2}(\y)) - f(\A(\y)) \right| + \frac{\beta}{2} \left( \norm{c(\A(\y))}^2 - \norm{c(\y)}^2 \right)\\
				\leq{}& - \left(\frac{\beta}{4} -\frac{8\Lf (\Ma +1)\Lac }{\Lsc ^3}  \right)\norm{c(\y)}^2 \leq -\frac{\beta}{8} \norm{c(\y)}^2.
			\end{aligned}
		\end{equation}
		Here the last inequality uses the fact that $\beta \geq \beta_{x}$. 
		
		Together with Lemma \ref{Le_dist_Ainfty}, \eqref{Eq_Prop_postprocessing_1} further implies that 
		\begin{equation*}
			h(\A^{\infty}(\y)) - h(\y) = \sum_{i=0}^{+\infty} h(\A^{i+1}(\y)) - h(\A^{i}(\y))\leq -\frac{\beta}{8} \sum_{i=0}^{+\infty} \norm{c(\A^i(\y))}^2
			\leq -\frac{\beta}{4} \norm{c(\y)}^2,
		\end{equation*}
		and we complete the proof. 
		\end{proof}

		\paragraph{Proof for Proposition \ref{Prop_upperbound_projection_gradient}}
		\begin{proof}
		Firstly, it follows from Lemma \ref{Le_dist_Ainfty} that $\A^{\infty}(\y)$ exists and $\A^{\infty}(\y) \in \Omegax{x} \cap \M$. 
		Let $U$ be a matrix whose columns form an orthonormal basis of $\ca{T}_{\A^{\infty}(\y)}$, from the definition of $\ca{T}_{\A^{\infty}(\y)} = \mathrm{Null}( \Jc(\A^{\infty}(\y))\tp )$, it holds that $U\tp \Jc(\A^{\infty}(\y)) = 0$. Then we can conclude that 
		\begin{equation}
			\label{Eq_Prop_upperbound_projection_gradient_1}
			\begin{aligned}
				&\norm{U\tp \Jc(\y)c(\y)} \leq \norm{U\tp \Jc(\y)} \norm{c(\y)} \leq \norm{U\tp  \Big(\Jc(\y) -  \Jc(\A^{\infty}(\y)) \Big)   } \norm{c(\y)} \\
				\leq{}& \norm{\Jc(\y) - \Jc(\A^{\infty}(\y))}\norm{c(\y)} \leq \frac{4\Lc(\Ma + 1)}{\Lsc} \norm{c(\y)}^2. 
			\end{aligned}
		\end{equation}
		Here the last inequality results from Lemma \ref{Le_dist_Ainfty}. 
		
		Because $U$ be a matrix whose columns form an orthonormal basis of the tangent space at $\A^{\infty}(\y)$, Lemma \ref{Le_Ja_identical} and Proposition \ref{Prop_FOSP_Rie} imply that $U U\tp \nabla g(\A^{\infty}(\y)) = \grad f(\A^{\infty}(y))$.
		Therefore, we can obtain
		\begin{equation*}
			\begin{aligned}
				&\norm{\nabla h(\y)} \geq \norm{U\tp \nabla h(\y)} = \norm{U\tp \left(\nabla g(\y) + \beta \Jc(\y) c(\y) \right)} \\
				\geq{}& \norm{U\tp \nabla g(\y)} - \beta \norm{U\tp \Jc(\y)c(\y)} \\
				\geq{}& \norm{U\tp \nabla g(\A^{\infty}(\y))}  - \Lg \norm{\A^{\infty}(\y) - \y} - \beta \norm{U\tp \Jc(\y)c(\y)} \\
				\overset{(i)}{\geq} {}& \norm{\grad f(\A^{\infty}(\y))} - \frac{4\Lg(\Ma + 1)}{\Lsc} \norm{c(\y)} - \frac{4\beta \Lc(\Ma + 1)}{\Lsc}\norm{c(\y)}^2.\\
			\end{aligned}
		\end{equation*}
		Here the inequality $(i)$ results from \eqref{Eq_Prop_upperbound_projection_gradient_1} and Lemma \ref{Le_dist_Ainfty}. 
		As a result, we further have
		\begin{equation*}
			\begin{aligned}
				&\norm{\nabla h(\y)} \geq \frac{1}{2}  \norm{\nabla h(\y)} + \frac{1}{2}  \norm{\nabla h(\y)}\\
				\overset{(ii)}{\geq}{}& \frac{1}{2}  \norm{\nabla h(\y)} + \frac{\beta \Lsc}{16(\Ma +1)} \norm{c(\y)}\\
				\geq & \frac{1}{2}\norm{\grad f(\A^{\infty}(\y))} 
				+ \frac{\beta \Lsc}{16(\Ma +1)} \norm{c(\y)}- \frac{2\Lg(\Ma + 1)}{\Lsc} \norm{c(\y)} \\&- \frac{2\beta\Lc (\Ma + 1)}{\Lsc}\norm{c(\y)}^2\\
				\overset{(iii)}{\geq}{}{}& \frac{1}{2}\norm{\grad f(\A^{\infty}(\y))} + \frac{\beta \Lsc}{16(\Ma +1)} \norm{c(\y)}- \frac{2\Lg(\Ma + 1)}{\Lsc} \norm{c(\y)} \\
				&- \frac{2\beta\Lc  (\Ma + 1)}{\Lsc} \cdot \frac{ \Mc \Lsc \varepsilon_x}{4\Mc(\Ma +1) + \Lsc} \norm{c(\y)}\\
				\geq{}& \frac{1}{2} \norm{\grad f(\A^{\infty}(\y))}.
			\end{aligned}
		\end{equation*}
		Here the inequality $(ii)$ is implied by Theorem \ref{The_firstorder_equivalence}, and the inequality $(iii)$ follows from the definition of $\BOmegax{x}$ and Lemma \ref{Le_bound_cx}. 
		\end{proof}

		\paragraph{Proof for Proposition \ref{Prop_sepctral_interlace}} 
		\begin{proof}
		It follows from Lemma \ref{Le_Ja_identical} that $\Ja(x) \tp d = d$ holds for any $d \in \Tx$, which implies the equality
		\begin{equation*}
			d\tp \nabla^2 h(x) d = d\tp \nabla^2 g(x) d = d\tp \left( \nabla^2 f(x) - \sum_{i=1}^p \lambda_i(x) \nabla^2 c_i(x) \right) d + d\tp \left(\DJa(x)[\grad f(x)]\right) d.
		\end{equation*}
		Together with Proposition \ref{Prop_FOSP_Rie}, we obtain the follow inequalities
		\begin{align*}
			\lambda_{\min}(\hess f(x)) \geq \min_{d \in \Tx, \norm{d} = 1} ~ d\tp \nabla^2 h(x) d - \La \norm{\grad f(x)} \geq \lambda_{\min}(\nabla^2 g(x)) - \La \norm{\grad f(x)}, \\
			\lambda_{\max}(\hess f(x)) \leq \max_{d \in \Tx, \norm{d} = 1} ~ d\tp \nabla^2 h(x) d + \La \norm{\grad f(x)} \leq \lambda_{\max}(\nabla^2 g(x)) + \La \norm{\grad f(x)},
		\end{align*} 
		which complete the proof. 
		\end{proof}

		\paragraph{Proof for Proposition \ref{Prop_esti_lambda_min}}
		\begin{proof}
		Firstly, 
		it follows from  Lemma \ref{Le_dist_Ainfty} and Theorem \ref{The_firstorder_equivalence} that
		\begin{equation}
			\label{Eq_Prop_esti_lambda_min_01}
			\norm{y - \A^{\infty}(y)} \leq \frac{4 (\Ma + 1)}{\Lsc} \norm{c(y)} \leq \frac{32 (\Ma + 1)^2}{\beta \Lsc^2} \norm{\nabla h(y)}.
		\end{equation}
		
		Let $U_y$ be a matrix whose columns form an orthonormal basis of $\ca{T}_{\A^{\infty}(y)}$. Lemma \ref{Le_Hessian_tangent} implies that
		\begin{equation*}
			\begin{aligned}
				& \lambda_{\min}(\hess f(\A^{\infty}(y))) \geq \lambda_{\min}( U_y\tp \nabla^2 g(\A^{\infty}(y)) U_y )  - \La \norm{\grad f(\A^{\infty}(y))}\\
				\geq{}&  \lambda_{\min}( U_y\tp \nabla^2 g(\A^{\infty}(y)) U_y )  - 2\La \norm{\nabla h(\A^{\infty}(y))}\\
				\geq{}& \lambda_{\min}( U_y\tp \nabla^2 g(\A^{\infty}(y)) U_y )  - 2\La \norm{\nabla h(y)} - 2\La\Lg \norm{y - \A^{\infty}(y)}\\
				\geq{}& \lambda_{\min}( U_y\tp \nabla^2 g(\A^{\infty}(y)) U_y ) - 2\La\left( 1 + \frac{32\Lg (\Ma + 1)^2}{\beta \Lsc^2}\right)\norm{\nabla h(y)}\\
				\geq{}& \lambda_{\min}( U_y\tp \nabla^2 g(\A^{\infty}(y)) U_y )  - \frac{5}{2} \La \norm{\nabla h(y)}.
			\end{aligned}
		\end{equation*}
		Here the last inequality follows from the fact that $\beta \geq \beta_x$. 
		Then, we can obtain
		\begin{equation*}
			\begin{aligned}
				& \lambda_{\min}( U_y\tp \nabla^2 g(\A^{\infty}(y)) U_y ) \geq \lambda_{\min}( U_y\tp \nabla^2 g(y) U_y ) - \norm{U_y\tp \left(\nabla^2 g(\A^{\infty}(y)) - \nabla^2 g(y)\right) U_y}\\
				\geq{}& \lambda_{\min}( U_y\tp \nabla^2 g(y) U_y ) - \norm{\nabla^2 g(\A^{\infty}(y)) - \nabla^2 g(y)}.
			\end{aligned}
		\end{equation*}
		Together with the equality that $U_y\tp \Jc(\A^{\infty}(y)) = 0$ and \eqref{Eq_Prop_esti_lambda_min_01}, we arrive at
		\begin{equation*}
			\begin{aligned}
				&\norm{U_y\tp \left( \Jc(y)\Jc(y)\tp + \DJc(y)[c(y)] \right) U_y} \leq 3 \Lc\Mc \norm{y - \A^{\infty}(y)} \\
				\leq{}&  \frac{96 \Lc\Mc(\Ma + 1)^2}{\beta \Lsc^2} \norm{\nabla h(y)},
			\end{aligned}
		\end{equation*}
		which further implies the inequality
		\begin{equation*}
			\begin{aligned}
				&\lambda_{\min}( \nabla^2 h(y)) \leq \lambda_{\min}( U_y\tp \nabla^2 h(y) U_y ) \\
				\leq{}& \lambda_{\min}( U_y\tp \nabla^2 g(y) U_y ) + \frac{96 \Lc\Mc(\Ma + 1)^2}{ \Lsc^2} \norm{\nabla h(y)}.
			\end{aligned}
		\end{equation*}
		Finally, we conclude that
		\begin{equation*}
			\begin{aligned}
				& \lambda_{\min}(\hess f(\A^{\infty}(y))) \geq \lambda_{\min}( U_y\tp \nabla^2 g(\A^{\infty}(y)) U_y )  - \La \norm{\grad f(\A^{\infty}(y))} \\
				\geq{}& \lambda_{\min}( U_y\tp \nabla^2 g(y) U_y ) - \norm{\nabla^2 g(\A^{\infty}(y)) - \nabla^2 g(y)} - \La \norm{\grad f(\A^{\infty}(y))}\\
				\geq{}& \lambda_{\min}( \nabla^2 h(y)) - \norm{\nabla^2 g(\A^{\infty}(y)) - \nabla^2 g(y)} - \left(\frac{5}{2}\La + \frac{96 \Lc\Mc(\Ma + 1)^2}{ \Lsc^2} \right) \norm{\nabla h(y)},
			\end{aligned}
		\end{equation*}
		and complete the proof. 
		\end{proof}

		\subsubsection{Proofs for Section \ref{Subsubsection_Relationships_oters}}
		\label{Section_proofs_323}
		\paragraph{Proof for Proposition \ref{Prop_KL}}
		\begin{proof}
		Since $f(x)$ satisfies the Riemannian {\L}ojasiewicz gradient inequality at  $x \in \M$ with exponent $\theta$, 
		there exists a neighborhood $\ca{U} \subset \BOmegax{x}$ and a constant $C > 0$ such that for any $y \in  \M \cap \ca{U}$, $\A^{\infty}(y) \in \ca{U}$ and $\norm{\grad f(y)} \geq C \left| f(y) - f(x)  \right|^{1-\theta}$. 
		
		For any $z \in \ca{U} \subset \tilde{\Omega}_{r}$, we denote $w:=\A^{\infty}(z) \in \M \cap \ca{U}$, 
		it then follows from Theorem \ref{The_firstorder_equivalence} and Proposition \ref{Prop_upperbound_projection_gradient} that
		\begin{equation}
			\label{Eq_Prop_KL_1}
			\begin{aligned}
				\norm{\nabla h(z)} \geq{}& \frac{1}{2} \norm{\grad f(w)} + \frac{\beta \Lsc}{16(\Ma  + 1)}  \norm{c(z)} \\
				\geq{}& \frac{C}{2}\left| f(w) - f(x)  \right|^{1-\theta} + \frac{\beta \Lsc}{16(\Ma  + 1)} \norm{c(z)},
			\end{aligned}
		\end{equation}
		which further implies
		\begin{equation*}
			\begin{aligned}
				&\left| h(z) - h(w) \right|^{1-\theta} = \left|f(\A(z)) - f(\A^{\infty}(z)) + \frac{\beta}{2} \norm{c(z)}^2 \right|^{1-\theta}\\
				\leq{}& \left| f(\A(z)) - f(\A^{\infty}(z)) \right|^{1-\theta} + \beta^{1-\theta} \norm{c(z)}^{2-2\theta}\\
				\overset{(i)}{\leq}{}& \left( \Lf \norm{\A(z) - \A^{\infty}(z)} \right)^{1-\theta} + \beta^{1-\theta} \norm{c(z)}^{2-2\theta}\\
				\overset{(ii)}{\leq}{}& \left(\left(\frac{16\Lf (\Ma +1)\Lac }{\Lsc ^3} \right)^{1-\theta} + \beta^{1-\theta}\right)  \norm{c(z)}^{2-2\theta}
				\leq{} \frac{5\beta^{1-\theta}}{4}\norm{c(z)}^{2-2\theta}.
			\end{aligned}
		\end{equation*}
		Here, the inequality $(i)$ results from the Lipschitz continuity of $f$. Meanwhile, the inequaltiy $(ii)$ is concluded from Lemma \ref{Le_A_secondorder_descrease} and Lemma \ref{Le_dist_Ainfty}. 
		As a result, it follows from the monotonicity of $t^{1-\theta}$ that
		\begin{equation}
			\label{Eq_Prop_KL_2}
			\left| h(z) - h(x) \right|^{1-\theta} \leq \left| h(w) - h(x)  \right|^{1-\theta} + \left| h(z) - h(w)  \right|^{1-\theta} \leq  \left| f(w) - f(x)  \right|^{1-\theta}+ \frac{5\beta^{1-\theta}}{4}\norm{c(z)}^{2-2\theta}.
		\end{equation}
		Substituting \eqref{Eq_Prop_KL_2} into \eqref{Eq_Prop_KL_1}, we obtain
		\begin{equation*}
			\begin{aligned}
				&\norm{\nabla h(z)} \geq \frac{C}{2}\left| f(w) - f(x)  \right|^{1-\theta} + \frac{\beta \Lsc}{16(\Ma  + 1)}\norm{c(z)} \\
				\geq{}& \frac{C}{2}\left| h(z) - h(x) \right|^{1-\theta}  - \frac{5C \beta^{1-\theta}}{8}\norm{c(z)}^{2-2\theta} + \frac{\beta \Lsc}{16(\Ma  + 1)} \norm{c(z)}.
			\end{aligned}
		\end{equation*}
		Denote $\HOmegax{x} := \left\{ \tilde{x}\in \BOmegax{x}: \norm{c(\tilde{x})}^{1-2\theta} \leq \left(\frac{ \beta^\theta \Lsc }{10(\Ma + 1) C }\right) \right\}$, then $\HOmegax{x}$ is a neighborhood of $x$. Resulting from the fact that $\theta \in (0, \frac{1}{2}]$, the following inequality holds for any $z \in  \HOmegax{x} $,
		\begin{equation*}
			\frac{\beta \Lsc}{16(\Ma  + 1)} \norm{c(z)} \geq \frac{5C \beta^{1-\theta}}{8}\norm{c(z)}^{2-2\theta},
		\end{equation*}
		which further implies the inequality
		\begin{equation*}
			\norm{\nabla h(z)} \geq\frac{C}{2} \left| h(z) - h(x) \right|^{1-\theta}.
		\end{equation*}
		Therefore, we complete the proof.
		\end{proof}

		\paragraph{Proof for Theorem \ref{The_local_minimizer}}
		\begin{proof}
		Suppose $\tilde{x}$ is a local minimizer of \ref{Prob_Ori}, then there exists a neighborhood $\ca{U}_1$ of $\tilde{x}$ such that $f(y) \geq f(\tilde{x})$ holds for any $y \in \M \cap \ca{U}_1$. We denote $\tilde{\ca{U}}_1:= \left\{ x \in \BOmegax{\tilde{x}}: \A^{\infty}(x) \in \ca{U}_1 \right\}$ 
		which is a neighborhood of $\tilde{x}$ from
		Lemma \ref{Le_dist_Ainfty}. Then 
		it follows from Proposition \ref{Prop_postprocessing} that the inequality
		\begin{equation*}
			h(y) \geq h(\A^{\infty}(y)) + \frac{\beta}{16}\norm{c(y)}^2 \geq h(\tilde{x}),
		\end{equation*}
		holds for any $y \in \BOmegax{\tilde{x}} \cap \tilde{\ca{U}}_1$, which further implies
		that $\tilde{x}$ is a local minimizer of \ref{Prob_Pen}. 
		
		On the other hand, suppose there exists $x \in \M$ such that $\tilde{x} \in \BOmegax{x} $ is a local minimizer of \ref{Prob_Pen}, i.e. there exists a neighborhood $\ca{U}_2 \subset \BOmegax{x}$ of $\tilde{x}$ such that $h(y) \geq h(\tilde{x})$ holds for any $y \in \ca{U}_2$. Firstly, we have $\tilde{x} \in \M$ from Theorem \ref{The_firstorder_equivalence}. Recalling the fact that $h(y) = f(y)$ holds for any $y \in \M$, we immediately arrive at $f(y) \geq f(\tilde{x}), ~ \forall y \in \ca{U}_2 \cap \M$.  
		Namely, $\tilde{x}$ is a local minimizer of \ref{Prob_Ori}. 
		\end{proof}

		\section{Supplementary Examples}
		\label{Section_Supplementary_Examples}
		Developing Riemannian solvers by the frameworks introduced in \cite{Absil2009optimization} requires several basic geometrical materials of the manifold. 
		Currently, the generalized Stiefel manifold, the hyperbolic manifold, and the symplectic Stiefel manifold are not supported by the existing Riemannian optimization packages programmed in Python \cite{townsend2016pymanopt,meghwanshi2018mctorch,geoopt2020kochurov}. On the other hand, even when the geometrical materials of the above-mentioned manifolds are available, it is still unclear whether
		various existing efficient unconstrained solvers can be easily implemented and added into the
		existing Python-based Riemannian optimization packages. 
		
		In this section, we present several supplementary examples to  
		illustrate that \ref{Prob_Pen} can be directly embedded in various existing unconstrained solvers to solve \ref{Prob_Ori}. All the numerical experiments in this section are run in serial on a platform with Intel(R) Xeon(R) Gold 6242R CPU @ 3.10GHz under Ubuntu 20.04 running Python 3.7.0 and Numpy 1.20.0 \cite{numpy2020array}.

		We choose the solvers from SciPy package \cite{scipy2020SciPyNMeth}, which provides various highly efficient solvers for unconstrained optimization. The detailed descriptions of the selected solvers are presented in Table \ref{Table_solvers}. We select the symplectic manifold as an example since there is no existing symplectic manifold solver
		available in any Python package. However, we can easily solve the problem by the SciPy package with our \ref{Prob_Pen} approach.

		\begin{table}[htbp]
			\centering
			\scriptsize
			\caption{Detailed descriptions for selected solvers from the SciPy package.}
			\label{Table_solvers}
			\begin{tabular}{|lp{0.6\textwidth}p{0.18\textwidth}|}
				\hline
				Name & Descriptions & Riemannian version in Python? \\ \hline
				CG          &  The nonlinear conjugate gradient method by Polak and Ribiere, which is a variant of the Fletcher-Reeves method  \cite{nocedal2006numerical}. This solver is written in Python.     &   Y    \\ 
				BFGS        & The quasi-Newton method proposed by  Broyden, Fletcher, Goldfarb, and Shanno (BFGS) \cite{nocedal2006numerical} and programmed in Python. & N\\
				L-BFGS-B    & The limit-memory BFGS method \cite{byrd1995limited}. SciPy package provides a python wrapper for the original FORTRAN solver developed by \cite{zhu1997algorithm,morales2011remark}. &  N \\
				TNC			& The truncated Newton method  \cite{nocedal2006numerical}. SciPy package provides a python wrapper for its C implementation by \cite{nash1984newton}. & N\\
				Newton-CG   & The Newton-CG method \cite{nocedal2006numerical} that uses conjugate gradient method method to the compute the search direction. This solver is programmed in Python. & N \\
				Trust-krylov & The Newton GLTR trust-region method \cite{gould1999solving}. The trust-region subproblems are solved by trilib \cite{lenders2018trlib}, which is a C programmed Krylov solver. & N\\
				Trust-ncg & The Newton conjugate gradient trust-region method \cite{nocedal2006numerical}, programmed in Python. & Y \\
				Trust-exact & The trust-region method for unconstrained minimization, where the trust-region subproblems are solved exactly by factorizing the Hessian matrix \cite{conn2000trust}. Therefore, it requires the explicit expression of the Hessian matrix of the objective function. & N\\  
				\hline
			\end{tabular}
			
		\end{table}

		\subsection{Preliminary numerical experiments on unconstrained optimization approaches}
		\label{Section_Numerical_symp}

		\begin{table}
			\centering
			\tiny
			\caption{Numerical results of the  nearest  symplectic matrix problem \eqref{Example_NSM} with fixed $2s = 10$.  }
			\label{Table_nearest_symplectic_matrix_fix_m}
			\begin{tabular}{c|c|cccccc}
				\hline
				Test instance             & Solver & Fval & Iter & Obj\_eval& Grad & Feas & CPU time (s) \\ \hline
				\multirow{8}{*}{$2m = 100$} 
				& CG            & 1.59e+00 & 45 & 76 & 5.77e-06 & 1.19e-06 & 0.03 \\
				& BFGS          & 1.59e+00 & 71 & 90 & 6.24e-06 & 9.47e-07 & 1.55 \\
				& L-BFGS-B      & 1.59e+00 & 49 & 56 & 3.81e-06 & 6.68e-07 & 0.03 \\
				& TNC           & 1.59e+00 & 41 & 489 & 5.38e-06 & 5.34e-07 & 0.13 \\
				& Newton-CG     & 1.59e+00 & 16 & 24 & 1.10e-06 & 1.93e-07 & 0.09 \\
				& Trust-krylov  & 1.59e+00 & 15 & 16 & 7.57e-06 & 1.32e-06 & 0.12 \\
				& Trust-ncg     & 1.59e+00 & 19 & 20 & 5.50e-07 & 9.40e-08 & 0.12 \\
				& Trust-exact   & 1.59e+00 & 37 & 38 & 6.45e-07 & 2.05e-07 & 29.10 \\ \hline
				\multirow{8}{*}{$2m = 500$} 
				& CG            & 1.23e+00 & 35 & 64 & 9.41e-06 & 1.92e-06 & 0.03 \\
				& BFGS          & 1.23e+00 & 51 & 75 & 9.21e-06 & 1.73e-06 & 49.18 \\
				& L-BFGS-B      & 1.23e+00 & 39 & 45 & 6.48e-06 & 1.38e-06 & 0.04 \\
				& TNC           & 1.23e+00 & 45 & 620 & 7.36e-07 & 8.79e-08 & 0.29 \\
				& Newton-CG     & 1.23e+00 & 15 & 19 & 9.18e-07 & 1.63e-07 & 0.09 \\
				& Trust-krylov  & 1.23e+00 & 13 & 14 & 9.68e-06 & 1.71e-06 & 0.13 \\
				& Trust-ncg     & 1.23e+00 & 14 & 15 & 2.67e-07 & 4.58e-08 & 0.10 \\
				& Trust-exact   & 1.23e+00 & 33 & 34 & 9.28e-07 & 2.52e-07 & 482.73 \\ \hline
				\multirow{8}{*}{$2m = 1000$} 
				& CG            & 1.27e+00 & 27 & 52 & 9.08e-06 & 3.32e-06 & 0.05 \\
				& BFGS          & 1.27e+00 & 52 & 72 & 9.27e-06 & 2.33e-06 & 258.38 \\
				& L-BFGS-B      & 1.27e+00 & 35 & 40 & 9.80e-06 & 3.63e-06 & 0.06 \\
				& TNC           & 1.27e+00 & 35 & 507 & 6.12e-07 & 6.49e-08 & 0.37 \\
				& Newton-CG     & 1.27e+00 & 14 & 18 & 1.77e-06 & 3.11e-07 & 0.09 \\
				& Trust-krylov  & 1.27e+00 & 12 & 13 & 5.45e-06 & 9.57e-07 & 0.14 \\
				& Trust-ncg     & 1.27e+00 & 14 & 15 & 2.27e-07 & 3.93e-08 & 0.11 \\
				& Trust-exact   & - & - & - & - & - & $> 1200$\\ \hline
				\multirow{8}{*}{$2m = 2000$} 
				& CG            & 1.18e+00 & 33 & 61 & 4.51e-06 & 6.13e-07 & 0.12 \\
				& BFGS          & - & - & - & - & - & $> 1200$\\
				& L-BFGS-B      & 1.18e+00 & 39 & 47 & 2.28e-06 & 3.27e-07 & 0.17 \\
				& TNC           & 1.18e+00 & 33 & 553 & 3.70e-07 & 4.54e-08 & 0.72 \\
				& Newton-CG     & 1.18e+00 & 14 & 17 & 2.40e-07 & 4.31e-08 & 0.15 \\
				& Trust-krylov  & 1.18e+00 & 12 & 13 & 1.42e-06 & 2.64e-07 & 0.20 \\
				& Trust-ncg     & 1.18e+00 & 18 & 19 & 3.50e-08 & 6.05e-09 & 0.26 \\
				& Trust-exact   & - & - & - & - & - & $> 1200$\\ \hline
				\multirow{8}{*}{$2m = 10000$}
				& CG            & 1.10e+00 & 30 & 60 & 8.85e-06 & 1.55e-06 & 0.42 \\
				& BFGS          & - & - & - & - & - & $> 1200$\\
				& L-BFGS-B 		& 1.10e+00 & 33 & 38 & 5.52e-06 & 1.00e-06 & 0.66 \\
				& TNC           & 1.10e+00 & 20 & 156 & 2.84e-07 & 3.33e-08 & 1.38 \\
				& Newton-CG     & 1.10e+00 & 13 & 16 & 5.83e-06 & 6.78e-07 & 0.31 \\
				& Trust-krylov  & 1.10e+00 & 11 & 12 & 9.26e-06 & 1.85e-06 & 0.47 \\
				& Trust-ncg     & 1.10e+00 & 16 & 17 & 1.62e-07 & 2.79e-08 & 0.56 \\
				& Trust-exact   & - & - & - & - & - & $> 1200$\\ \hline
			\end{tabular}

		\end{table}
		
		\begin{table}
			\centering
			\tiny
			\caption{Numerical results of the  nearest  symplectic matrix problem \eqref{Example_NSM} with fixed $2m = 1000$.  }
			\label{Table_nearest_symplectic_matrix_fix_s}
			\begin{tabular}{c|c|cccccc}
				\hline
				Test instance             & Solver & Fval & Iter & Obj\_eval& Grad & Feas & CPU time (s) \\ \hline
				\multirow{8}{*}{$2s = 2$} 
				& CG            & 2.15e-01 & 16 & 33 & 6.32e-06 & 1.65e-06 & 0.02 \\
				& BFGS          & 2.15e-01 & 13 & 16 & 8.20e-06 & 8.28e-07 & 1.18 \\
				& L-BFGS-B      & 2.15e-01 & 14 & 16 & 1.34e-06 & 2.09e-07 & 0.02 \\
				& TNC           & 2.15e-01 & 35 & 333 & 2.90e-07 & 5.52e-08 & 0.12 \\
				& Newton-CG     & 2.15e-01 & 10 & 13 & 5.92e-07 & 1.49e-07 & 0.04 \\
				& Trust-krylov  & 2.15e-01 & 10 & 11 & 2.25e-06 & 5.67e-07 & 0.05 \\
				& Trust-ncg     & 2.15e-01 & 14 & 15 & 5.56e-07 & 1.40e-07 & 0.06 \\
				& Trust-exact   & 2.15e-01 & 27 & 28 & 4.42e-06 & 1.11e-06 & 124.31 \\ \hline
				\multirow{8}{*}{$2s = 10$} 
				& CG            & 1.17e+00 & 33 & 59 & 2.71e-06 & 3.13e-07 & 0.04 \\
				& BFGS          & 1.17e+00 & 48 & 72 & 9.34e-06 & 9.89e-07 & 234.10 \\
				& L-BFGS-B      & 1.17e+00 & 43 & 49 & 1.64e-06 & 2.27e-07 & 0.13 \\
				& TNC           & 1.17e+00 & 32 & 538 & 6.86e-07 & 8.59e-08 & 0.39 \\
				& Newton-CG     & 1.17e+00 & 15 & 18 & 1.50e-06 & 2.56e-07 & 0.08 \\
				& Trust-krylov  & 1.17e+00 & 12 & 13 & 3.03e-06 & 5.38e-07 & 0.14 \\
				& Trust-ncg     & 1.17e+00 & 14 & 15 & 2.85e-07 & 4.87e-08 & 0.11 \\
				& Trust-exact   & - & - & - & - & - & $> 1200$\\ \hline
				\multirow{8}{*}{$2s = 50$} 
				& CG            & 7.19e+00 & 47 & 80 & 1.34e-06 & 1.51e-07 & 0.41 \\
				& BFGS          & - & - & - & - & - & $>1200$ \\
				& L-BFGS-B      & 7.19e+00 & 53 & 57 & 6.83e-07 & 1.02e-07 & 0.62 \\
				& TNC           & 7.19e+00 & 41 & 748 & 2.64e-07 & 2.93e-08 & 4.41 \\
				& Newton-CG     & 7.19e+00 & 17 & 20 & 2.05e-07 & 3.59e-08 & 0.63 \\
				& Trust-krylov  & 7.19e+00 & 15 & 16 & 4.11e-06 & 7.06e-07 & 0.93 \\
				& Trust-ncg     & 7.19e+00 & 17 & 18 & 1.91e-07 & 3.26e-08 & 0.86 \\
				& Trust-exact   & - & - & - & - & - & $>1200$ \\\hline
				\multirow{8}{*}{$2s = 100$} 
				& CG            & 1.65e+01 & 58 & 100 & 6.64e-07 & 1.51e-07 & 0.84 \\
				& BFGS          & - & - & - & - & - & $>1200$ \\
				& L-BFGS-B      & 1.65e+01 & 67 & 73 & 2.39e-07 & 3.81e-08 & 1.21 \\
				& TNC           & 1.65e+01 & 43 & 710 & 4.09e-07 & 4.27e-08 & 8.10 \\
				& Newton-CG     & 1.65e+01 & 18 & 21 & 4.71e-06 & 8.52e-07 & 1.06 \\
				& Trust-krylov  & 1.65e+01 & 19 & 20 & 2.54e-06 & 4.33e-07 & 1.85 \\
				& Trust-ncg     & 1.65e+01 & 23 & 24 & 8.83e-08 & 1.53e-08 & 1.99 \\
				& Trust-exact   & - & - & - & - & - & $>1200$ \\ \hline
				\multirow{8}{*}{$2s = 500$} 
				& CG            & 1.21e+02 & 135 & 224 & 2.65e-07 & 5.28e-08 & 10.38 \\
				& BFGS          & - & - & - & - & - & $>1200$ \\
				& L-BFGS-B      & 1.21e+02 & 133 & 200 & 3.52e-07 & 5.78e-08 & 21.90 \\
				& TNC           & 1.21e+02 & 67 & 977 & 5.81e-07 & 9.86e-08 & 64.28 \\
				& Newton-CG     & 1.21e+02 & 31 & 36 & 1.26e-06 & 2.20e-07 & 12.96 \\
				& Trust-krylov  & 1.21e+02 & 28 & 29 & 4.40e-06 & 7.55e-07 & 21.64 \\
				& Trust-ncg     & 1.21e+02 & 47 & 37 & 2.09e-07 & 3.50e-08 & 54.90 \\
				& Trust-exact   & - & - & - & - & - & $>1200$ \\ \hline
			\end{tabular}
			
		\end{table}

			In this subsection, we test the performance of various unconstrained optimization approaches on solving \ref{Prob_Ori} through \ref{Prob_Pen}. 
			
			We first consider the following optimization problem over the symplectic Stiefel manifold (i.e., $\M = \left\{ X \in \bb{R}^{2m\times 2s}: X \tp Q_m X = Q_s \right\}$ with $n = 4ms$, as described  in Table \ref{Table_implementation}),
			\begin{equation}
				\label{Example_NSM}
				\begin{aligned}
					\min_{x \in \M} \quad \frac{1}{2} \norm{x - \tilde{w}}^2.
				\end{aligned}
			\end{equation}
			Problem \eqref{Example_NSM} is usually referred as the  nearest  symplectic matrix problem \cite{gao2021riemannian,wu2010critical}, which aims to calculate the nearest symplectic matrix on the symplectic Stiefel manifold  to a target point $\tilde{w} \in \bb{R}^n$ with respect to the $\ell_2$-norm. In our numerical examples, we follow the settings in \cite{gao2021riemannian} to randomly generate $\tilde{w}$ in $\bb{R}^n$ and scale it by $\tilde{w} = \tilde{w} / \mathrm{norm}(\tilde{w})$.  We set the penalty parameter $\beta = 2$ in \ref{Prob_Pen}, and initiate all the selected solvers in Table \ref{Table_solvers} at the the same initial point, which is randomly generated over the symplectic Stiefel manifold. 
			Moreover, we adopt the auto-differentiation packages to automatically generate the gradient and Hessian from the expression of $h(x)$.  Specifically, we generate the gradient $\nabla h(x)$ and explicit expression of $\nabla^2 h(x)$ by the autograd package \cite{maclaurin2015autograd}. Furthermore, for the solvers ``Trust-krylov'' and ``Trust-ncg'', the Hessian-vector product for $\nabla^2 h(x)$ is automatically generated by the JAX package \cite{jax2018github} from the expression of $\nabla h(x)$. We terminate the solvers when $\norm{\nabla h(\xk)} \leq 10^{-5}$, or the number of iterations exceeds $10000$, while keeping all the other parameters as the default values.

			Additionally, we also consider the following optimization problem over the generalized Stiefel manifold, (i.e., $\M = \left\{ X \in \bb{R}^{m\times s}: X \tp B X = I_s \right\}$ with $B \in \bb{R}^{m\times m}$ as a prefixed positive-definite symmetric matrix and $n = ms$),
			\begin{equation}
				\label{Example_GenEig}
				\min_{X \in \bb{R}^{m\times s}} \quad f(X) = -\frac{1}{2} \tr\left( X\tp AX \right) \qquad \text{s.t.} ~  X\tp BX = I_s. 
			\end{equation}
			Problem \eqref{Example_GenEig} is usually referred as the  generalized eigenvalue problem, which aims to compute the largest $s$ eigenvalues and their corresponding eigenvectors for the pair $(A, B)$. In our numerical examples, we follows the settings in \cite{scott1981solving} to randomly generate the sparse matrix $A$ and $B$ by the $\mathrm{scipy.sparse.random}$ function from SciPy, where the density parameter  of the generated matrix is fixed as $0.01$. Then we scale $A$ and $B$ by $A = A/ \norm{A}$ and $B = B/\norm{B}$.  Moreover, we set $B = 1.1 *  I_m + B$ to ensure that $B \succ 0$. 
			
			Similar to the settings for the nearest symplectic matrix problem \eqref{Example_NSM}, we choose the constraint dissolving mapping $\A$ as suggested in Table \ref{Table_implementation}, set the penalty parameter $\beta = 2$ in \ref{Prob_Pen}, and initiate all the selected solvers in Table \ref{Table_solvers} at the the same initial point, which is randomly generated over the generalized Stiefel manifold. We terminate the solvers once $\norm{\nabla h(\xk)} \leq 10^{-5}$, or the number of iterations exceeds $10000$, or the running time exceeds $1200$ seconds, while keeping all the other parameters as the default values.

			Table \ref{Table_nearest_symplectic_matrix_fix_m}-\ref{Table_generalized_eigenvalue_fix_s} illustrate the performance of all the solvers from Table \ref{Table_solvers} in solving problem \eqref{Example_NSM} and \eqref{Example_GenEig}, 
			under different combinations of problem parameters $m$ and $s$. The terms ``Fval'', ``Iter'', ``Obj\_eval''
			``Grad", ``Feas", and ``CPU time" 
			stand for the function value,  the number of iterations, the number of function value evaluations,  $\norm{\nabla h(x^*)}$, $\norm{c(x^*)}$, and the wall-clock running time, respectively. Here $x^*$ is the final solution returned by these solvers. We can learn from these tables
			that all the solvers can directly minimize \ref{Prob_Pen} and yield solutions with similar function values and high accuracy in feasibility. 
			This shows that solving \ref{Prob_Ori} via our \ref{Prob_Pen} formulation is not sensitive to the choice of the unconstrained optimization solver. Of course, the running times for various solvers may differ. But for our example, CG, L-BFGS-B, Newton-CG, Trust-krylov, Trust-ncg, and to a lesser extent TNC, can all solve the \ref{Prob_Pen} problem highly efficiently.

		\begin{table}
			\centering
			\tiny
			\caption{Numerical results of the generalized eigenvalue problem \eqref{Example_GenEig} with fixed $s = 20$.  }
			\label{Table_generalized_eigenvalue_fix_m}
			\begin{tabular}{c|c|cccccc}
				\hline
				Test instance             & Solver & Fval & Iter & Obj\_eval& Grad & Feas & CPU time (s) \\ \hline
				\multirow{8}{*}{$m = 100$} 
				& CG            & -7.11e+00 & 375 & 567 & 9.17e-06 & 1.37e-07 & 0.15 \\
				& BFGS          & -7.11e+00 & 250 & 272 & 8.02e-06 & 1.37e-07 & 30.23 \\
				& L-BFGS-B      & -7.11e+00 & 264 & 274 & 1.16e-06 & 8.69e-08 & 0.11 \\
				& TNC           & -7.11e+00 & 523 & 1119 & 9.89e-06 & 3.43e-07 & 0.68 \\
				& Newton-CG     & -7.11e+00 & 50 & 65 & 1.08e-06 & 5.68e-09 & 0.16 \\
				& Trust-krylov  & -7.11e+00 & 44 & 45 & 1.57e-06 & 3.41e-07 & 0.39 \\
				& Trust-ncg     & -7.11e+00 & 59 & 60 & 2.03e-06 & 2.57e-07 & 0.11 \\
				& Trust-exact   & -7.80e+00 & 66 & 68 & 1.34e-07 & 8.50e-08 & 130.74 \\ \hline
				\multirow{8}{*}{$m = 500$} 
				& CG            & -6.32e+00 & 621 & 1115 & 1.74e-06 & 2.18e-07 & 0.81 \\
				& BFGS   & - & - & - & - & - & $>1200$ \\
				& L-BFGS-B      & -6.32e+00 & 696 & 725 & 1.84e-06 & 2.61e-07 & 1.02 \\
				& TNC           & -6.32e+00 & 73 & 880 & 8.21e-06 & 1.15e-06 & 0.81 \\
				& Newton-CG     & -6.32e+00 & 107 & 170 & 3.32e-06 & 5.82e-07 & 1.96 \\
				& Trust-krylov  & -6.32e+00 & 94 & 95 & 3.27e-08 & 6.46e-09 & 1.70 \\
				& Trust-ncg     & -6.32e+00 & 128 & 129 & 1.45e-06 & 1.70e-07 & 0.91 \\
				& Trust-exact   & - & - & - & - & - & $>1200$ \\ \hline
				\multirow{8}{*}{$m = 1000$} 
				& CG            & -4.63e+00 & 361 & 661 & 1.62e-06 & 1.63e-07 & 1.85 \\
				& BFGS   & - & - & - & - & - & $>1200$ \\
				& L-BFGS-B      & -4.63e+00 & 387 & 413 & 2.66e-06 & 1.27e-07 & 2.58 \\
				& TNC           & -4.63e+00 & 64 & 767 & 8.05e-06 & 1.36e-06 & 1.83 \\
				& Newton-CG     & -4.63e+00 & 85 & 117 & 7.56e-06 & 1.11e-06 & 4.36 \\
				& Trust-krylov  & -4.63e+00 & 68 & 69 & 7.26e-07 & 1.57e-07 & 3.30 \\
				& Trust-ncg     & -4.63e+00 & 93 & 94 & 1.95e-06 & 3.74e-07 & 2.11 \\
				& Trust-exact   & - & - & - & - & - & $>1200$ \\ \hline
				\multirow{8}{*}{$m = 5000$} 
				& CG            & -2.08e+00 & 1549 & 2910 & 6.79e-06 & 4.86e-07 & 158.54 \\
				& BFGS   & - & - & - & - & - & $>1200$ \\
				& L-BFGS-B      & -2.08e+00 & 1787 & 1906 & 6.56e-06 & 6.99e-07 & 121.96 \\
				& TNC           & -2.08e+00 & 249 & 3269 & 7.43e-06 & 1.35e-06 & 184.79 \\
				& Newton-CG     & -2.08e+00 & 226 & 373 & 3.72e-05 & 5.89e-07 & 500.96 \\
				& Trust-krylov  & -2.08e+00 & 212 & 213 & 7.49e-07 & 1.68e-07 & 208.13 \\
				& Trust-ncg     & -2.08e+00 & 282 & 283 & 9.09e-06 & 1.26e-06 & 140.87 \\
				& Trust-exact   & - & - & - & - & - & $>1200$ \\ \hline
			\end{tabular}

		\end{table}

		\begin{table}
			\centering
			\tiny
			\caption{Numerical results of the generalized eigenvalue problem \eqref{Example_GenEig} with fixed $m = 1000$.  }
			\label{Table_generalized_eigenvalue_fix_s}
			\begin{tabular}{c|c|cccccc}
				\hline
				Test instance             & Solver & Fval & Iter & Obj\_eval& Grad & Feas & CPU time (s) \\ \hline
				\multirow{8}{*}{$s = 10$} 
				& CG            & -1.67e+00 & 1307 & 2342 & 2.38e-06 & 1.34e-07 & 12.34 \\
				& BFGS   & - & - & - & - & - & $>1200$ \\
				& L-BFGS-B      & -1.67e+00 & 935 & 989 & 2.92e-06 & 2.30e-07 & 6.11 \\
				& TNC           & -1.67e+00 & 1337 & 9024 & 7.78e-06 & 1.11e-06 & 30.86 \\
				& Newton-CG     & -1.67e+00 & 198 & 319 & 8.55e-06 & 3.04e-07 & 23.80 \\
				& Trust-krylov  & -1.67e+00 & 141 & 142 & 7.82e-07 & 1.71e-07 & 8.91 \\
				& Trust-ncg     & -1.67e+00 & 194 & 195 & 7.71e-06 & 9.66e-07 & 11.01 \\
				& Trust-exact   & - & - & - & - & - & $>1200$ \\ \hline
				\multirow{8}{*}{$s = 20$} 
				& CG            & -3.29e+00 & 709 & 1278 & 3.17e-06 & 2.77e-07 & 10.81 \\
				& BFGS   & - & - & - & - & - & $>1200$ \\
				& L-BFGS-B      & -3.29e+00 & 671 & 715 & 3.80e-06 & 2.21e-07 & 8.09 \\
				& TNC           & -3.29e+00 & 100 & 1241 & 8.25e-06 & 1.50e-06 & 7.94 \\
				& Newton-CG     & -3.29e+00 & 106 & 155 & 3.68e-08 & 4.29e-09 & 16.24 \\
				& Trust-krylov  & -3.29e+00 & 105 & 106 & 8.19e-08 & 1.75e-08 & 15.32 \\
				& Trust-ncg     & -3.29e+00 & 141 & 142 & 7.72e-06 & 1.36e-06 & 9.43 \\
				& Trust-exact   & - & - & - & - & - & $>1200$ \\ \hline
				\multirow{8}{*}{$s = 50$} 
				& CG            & -7.88e+00 & 888 & 1600 & 3.61e-06 & 2.50e-07 & 39.64 \\
				& BFGS   & - & - & - & - & - & $>1200$ \\
				& L-BFGS-B      & -7.88e+00 & 1034 & 1088 & 4.58e-06 & 5.04e-07 & 34.57 \\
				& TNC           & -7.88e+00 & 140 & 1995 & 8.85e-06 & 1.59e-06 & 54.80 \\
				& Newton-CG     & -7.88e+00 & 194 & 298 & 3.50e-05 & 1.19e-06 & 183.26 \\
				& Trust-krylov  & -7.88e+00 & 138 & 139 & 4.22e-06 & 9.28e-07 & 91.45 \\
				& Trust-ncg     & -7.88e+00 & 177 & 178 & 1.66e-06 & 1.41e-07 & 53.52 \\
				& Trust-exact   & - & - & - & - & - & $>1200$ \\ \hline
				\multirow{8}{*}{$m = 100$} 
				& CG            & -1.49e+01 & 702 & 1259 & 5.60e-06 & 5.32e-07 & 65.21 \\
				& BFGS   & - & - & - & - & - & $>1200$ \\
				& L-BFGS-B      & -1.49e+01 & 730 & 771 & 8.74e-06 & 6.26e-07 & 52.89 \\
				& TNC           & -1.49e+01 & 106 & 1430 & 5.50e-06 & 1.01e-06 & 83.95 \\
				& Newton-CG     & -1.49e+01 & 120 & 186 & 1.35e-05 & 1.71e-06 & 136.75 \\
				& Trust-krylov  & -1.49e+01 & 116 & 117 & 6.22e-06 & 1.38e-06 & 131.52 \\
				& Trust-ncg     & -1.49e+01 & 143 & 144 & 1.19e-07 & 2.52e-08 & 91.91 \\
				& Trust-exact   & - & - & - & - & - & $>1200$ \\ \hline
			\end{tabular}

		\end{table}

		Finally, we remark that while fixing the penalty parameter $\beta=2$
		in the above numerical experiments is sufficient for 
		the corresponding CDF to be an exact penalty function, for other application examples, we may need to 
		dynamically increase the penalty parameter in order to 
		make the corresponding CDF $h(\cdot)$ an exact penalty function, or choose the penalty parameter as suggested in Remark \ref{Rmk_esti_beta}. We leave the strategy to adjust $\beta$ 
		for future investigation.

		\subsection{Comparison with Riemannian optimization approaches}
			
			In this section, we test the numerical performance of our proposed constraint dissolving approaches and compare them with the state-of-the-art Riemannian optimization solvers from the PyManopt package  (version 2.0.0) \cite{townsend2016pymanopt}, which is the python version of the well-recognized optimization package Manopt \cite{boumal2014manopt}. 
			Our test example is the problem of finding the nearest low-rank correlation matrix (NCM) to a given matrix $G \in \bb{R}^{m\times m}$, which can be reformulated as the following optimization problem over the oblique manifold,
			\begin{equation}
				\label{Example_NCM}
				\begin{aligned}
					\min_{X \in \bb{R}^{m\times s}} \quad & f(X) = \frac{1}{2} \norm{ H\circ(XX\tp - G)}_F^2\\
					\text{s.t.} \quad & \mathrm{Diag}(XX\tp) = I_m.
				\end{aligned}
			\end{equation}
			Here $H \in \bb{R}^{m\times m}$ is a  weight matrix with nonnegative entries. For all the numerical experiments in this subsection, we generate the matrix $\hat{G}$ from the gene expression data provided in \cite{li2010inexact}. Then the matrix $G$ is generated from perturbing $\hat{G}$  by $G = (1-\theta)\hat{G} + \theta E$, where $\theta \geq 0$ is a prefixed parameter, and $E \in \bb{R}^{m\times m}$ is a randomly generated matrix with all of its diagonal entries equal to $1$. The weight matrix $H$ in \eqref{Example_NCM} is chosen as a symmetric matrix whose entries are uniformly distributed in $[0, 1]$.  
			
			Based on the numerical experiments in Section \ref{Section_Numerical_symp}, we choose the L-BFGS-B, CG and Trust-ncg solvers from the SciPy package. Moreover, we choose the Riemannian conjugate gradient method (RCG) \cite{boumal2015low}  and Riemannian trust-region method (RTR) \cite{absil2007trust,Absil2009optimization} from the PyManopt package. 
			We stop all the compared solvers once the norm of its Riemannian gradient is smaller than $10^{-5}$, or the maximum number of iterations exceeds $10000$. All the other parameters are fixed as their default values. All the solvers start from the same initial point, which is randomly generated on the oblique manifold in each test instance. Moreover, in all the test instances, the gradients and Hessians of the objective function $f$ are automatically computed through the automatic differentiation algorithm from PyTorch.
			Furthermore, in our proposed constraint dissolving approaches, we choose the constraint dissolving mapping $\A$ as suggested in Table \ref{Table_implementation}, while the penalty parameter $\beta$ is chosen as suggested in Remark \ref{Rmk_esti_beta}, where we fix $N_{\beta} = 20$, $\theta_{\beta} = 2.5$, $\delta_{\beta} = 1$, $\varepsilon_{\beta} = 10^{-10}$ and choose $\tilde{x}$ as the initial point of the algorithm in each test instance.
			
			Table \ref{Table_NCM} exhibits the numerical results for solving the nearest correlation matrix problem by  our proposed constraint dissolving approaches and the Riemannian optimization solvers from PyManopt.  From Table \ref{Table_NCM}, we can conclude that the CG solver from the SciPy package achieves comparable performance with the RCG solver provided by the PyManopt package. Moreover, when the column size $s$ of our test problems is relatively large, the Trust-ncg solver shows superior performance over the RTR solver. Furthermore, benefited from the highly efficient L-BFGS-B solver that is programmed in FORTRAN and wrapped by the SciPy package, our proposed constraint dissolving approach gains significant advantages against the compared Riemannian solvers in almost all the test instances. Therefore, we can conclude that our proposed constraint dissolving approaches  can achieve comparable efficiency as the state-of-the-art Riemannian optimization solvers. More importantly, solving \ref{Prob_Ori} through our proposed constraint dissolving approaches can benefit from the advanced features of existing unconstrained optimization solvers (e.g., the wrapper for FORTRAN/C solvers), and achieve higher efficiency than existing Riemannian optimization solvers. 
			
		\begin{table}[htbp]
			\centering
			\tiny
			\caption{Numerical results of the NCM problem \eqref{Example_NCM} on Arabidopsis and Leukemia dataset. }
			\label{Table_NCM}
			\begin{tabular}{l|l|ccccc|ccccc}
				\hline
				\multicolumn{2}{c|}{\multirow{2}{*}{}} & \multicolumn{5}{c|}{Arabidopsis}          & \multicolumn{5}{c}{Leukemia}          \\
				\multicolumn{2}{c|}{}                  & Fval & Iter & Grad & Feas & CPU time (s) & Fval & Iter & Grad & Feas & CPU time (s) \\ \hline
				\multirow{5}{*}{$s = 5$}   
				& CG & 7.67e+03  & 259   & 7.31e-06     & 5.29e-15     & 1.72 & 1.86e+04  & 274   & 3.83e-05     & 6.50e-15     & 2.96 \\
				& L-BFGS-B & 7.67e+03  & 322   & 8.23e-06     & 5.16e-15     & 1.48 & 1.86e+04  & 157   & 7.57e-06     & 6.06e-15     & 1.24 \\
				& Trust-ncg & 7.67e+03  & 57  & 6.27e-06     & 4.85e-15     & 3.51 & 1.86e+04  & 71  & 3.93e-05     & 6.22e-15     & 20.96 \\ \cline{2-12}
				& RCG & 7.67e+03  & 389    & 6.70e-06     & 5.43e-15     & 1.74 & 1.86e+04  & 261    & 9.05e-06     & 6.55e-15     & 2.08\\ 
				& RTR & 7.67e+03  & 55    & 3.90e-06     & 5.10e-15     & 3.29 & 1.86e+04  & 75    & 9.30e-06     & 6.38e-15     & 21.21 \\ \hline
				\multirow{5}{*}{$s = 10$}   
				& CG & 2.17e+03  & 216   & 9.97e-06     & 1.63e-14     & 1.35 & 6.38e+03  & 306   & 4.93e-06     & 6.27e-15     & 44.36 \\
				& L-BFGS-B & 2.17e+03  & 239   & 9.51e-06     & 5.02e-15     & 1.09 & 6.38e+03  & 216   & 9.94e-06     & 5.62e-15     & 22.12 \\
				& Trust-ncg & 2.16e+03  & 65  & 2.56e-07     & 5.12e-15     & 5.52 & 6.38e+03  & 86  & 1.41e-07     & 6.07e-15     & 68.99 \\\cline{2-12}
				& RCG & 2.16e+03  & 272    & 8.98e-06     & 5.82e-15     & 1.17 & 6.38e+03  & 382    & 9.87e-06     & 7.02e-15     & 38.99\\
				& RTR & 2.16e+03  & 71    & 9.01e-06     & 5.73e-15     & 8.60 & 6.38e+03  & 58    & 9.23e-06     & 6.78e-15     & 92.62 \\ \hline
				\multirow{5}{*}{$s = 15$}   
				& CG & 1.11e+03  & 375   & 4.33e-06     & 4.76e-15     & 53.37  & 3.44e+03  & 302   & 9.09e-06     & 5.67e-15     & 44.60 \\
				& L-BFGS-B & 1.11e+03  & 245   & 8.70e-06     & 4.49e-15     & 25.12 & 3.44e+03  & 234   & 9.29e-06     & 6.15e-15     & 24.56 \\
				& Trust-ncg & 1.11e+03  & 65  & 3.19e-06     & 4.43e-15     & 69.94  & 3.44e+03  & 80  & 1.14e-06     & 5.78e-15     & 66.31 \\\cline{2-12}
				& RCG & 1.11e+03  & 647    & 7.71e-06     & 6.30e-15     & 63.58  & 3.44e+03  & 442    & 8.65e-06     & 7.17e-15     & 45.27\\
				& RTR & 1.11e+03  & 52    & 9.23e-07     & 6.03e-15     & 99.15 & 3.44e+03  & 55    & 9.34e-07     & 7.34e-15     & 120.35 \\ \hline
				\multirow{5}{*}{$s = 20$}   
				& CG & 7.26e+02  & 498   & 9.75e-06     & 4.02e-15     & 69.60  & 2.23e+03  & 515   & 5.72e-06     & 1.15e-14     & 65.19 \\
				& L-BFGS-B & 7.26e+02  & 437   & 7.80e-06     & 5.52e-15     & 44.96 & 2.23e+03  & 262   & 9.21e-06     & 6.60e-15     & 27.55 \\
				& Trust-ncg & 7.26e+02  & 70  & 4.94e-06     & 4.46e-15     & 56.95  & 2.23e+03  & 90  & 1.43e-07     & 6.53e-15     & 97.37 \\\cline{2-12}
				& RCG & 7.26e+02  & 854    & 5.67e-06     & 6.06e-15     & 84.50  & 2.23e+03  & 663    & 8.26e-06     & 7.14e-15     & 68.19\\
				& RTR & 7.26e+02  & 98    & 6.88e-06     & 6.10e-15     & 103.80 & 2.23e+03  & 47    & 9.71e-06     & 7.61e-15     & 131.94 \\ \hline
				\multirow{5}{*}{$s = 25$}   
				& CG & 5.44e+02  & 620   & 9.40e-06     & 5.57e-15     & 88.79  & 1.61e+03  & 431   & 9.30e-06     & 1.01e-13     & 61.72 \\
				& L-BFGS-B & 5.44e+02  & 380   & 9.51e-06     & 5.47e-15     & 38.84 & 1.61e+03  & 240   & 9.03e-06     & 4.73e-15     & 25.75 \\
				& Trust-ncg & 5.44e+02  & 79  & 9.13e-07     & 4.08e-15     & 100.85  & 1.61e+03  & 67  & 8.96e-06     & 5.63e-15     & 67.89 \\\cline{2-12}
				& RCG & 5.44e+02  & 1307    & 9.90e-06     & 6.02e-15     & 129.46  & 1.61e+03  & 560    & 9.53e-06     & 7.34e-15     & 57.77\\
				& RTR & 5.44e+02  & 57    & 4.31e-07     & 5.99e-15     & 176.39 & 1.61e+03  & 47    & 4.29e-06     & 7.25e-15     & 151.39 \\ \hline
			\end{tabular}
			
		\end{table}

		\bibliographystyle{plain}
		\bibliography{ref}

\begin{thebibliography}{10}

\bibitem{abrudan2009conjugate}
Traian Abrudan, Jan Eriksson, and Visa Koivunen.
\newblock Conjugate gradient algorithm for optimization under unitary matrix
  constraint.
\newblock {\em Signal Processing}, 89(9):1704--1714, 2009.

\bibitem{abrudan2008steepest}
Traian~E Abrudan, Jan Eriksson, and Visa Koivunen.
\newblock Steepest descent algorithms for optimization under unitary matrix
  constraint.
\newblock {\em IEEE Transactions on Signal Processing}, 56(3):1134--1147, 2008.

\bibitem{absil2007trust}
P-A Absil, Christopher~G Baker, and Kyle~A Gallivan.
\newblock Trust-region methods on {R}iemannian manifolds.
\newblock {\em Foundations of Computational Mathematics}, 7(3):303--330, 2007.

\bibitem{Absil2009optimization}
P-A Absil, Robert Mahony, and Rodolphe Sepulchre.
\newblock {\em Optimization algorithms on matrix manifolds}.
\newblock Princeton University Press, 2009.

\bibitem{bai2014minimization}
Zhaojun Bai and Ren-Cang Li.
\newblock Minimization principles and computation for the generalized linear
  response eigenvalue problem.
\newblock {\em BIT Numerical Mathematics}, 54(1):31--54, 2014.

\bibitem{becigneul2018riemannian}
Gary B{\'e}cigneul and Octavian-Eugen Ganea.
\newblock Riemannian adaptive optimization methods.
\newblock {\em arXiv preprint arXiv:1810.00760}, 2018.

\bibitem{bolte2014proximal}
J{\'e}r{\^o}me Bolte, Shoham Sabach, and Marc Teboulle.
\newblock Proximal alternating linearized minimization for nonconvex and
  nonsmooth problems.
\newblock {\em Mathematical Programming}, 146(1):459--494, 2014.

\bibitem{boumal2020introduction}
Nicolas Boumal.
\newblock An introduction to optimization on smooth manifolds.
\newblock {\em Available at http://sma.epfl.ch/~nboumal/book/index.html}, 2020.

\bibitem{boumal2015low}
Nicolas Boumal and P-A Absil.
\newblock Low-rank matrix completion via preconditioned optimization on the
  grassmann manifold.
\newblock {\em Linear Algebra and its Applications}, 475:200--239, 2015.

\bibitem{boumal2014manopt}
Nicolas Boumal, Bamdev Mishra, P-A Absil, and Rodolphe Sepulchre.
\newblock Manopt, a matlab toolbox for optimization on manifolds.
\newblock {\em The Journal of Machine Learning Research}, 15(1):1455--1459,
  2014.

\bibitem{jax2018github}
James Bradbury, Roy Frostig, Peter Hawkins, Matthew~James Johnson, Chris Leary,
  Dougal Maclaurin, George Necula, Adam Paszke, Jake Vander{P}las, Skye
  Wanderman-{M}ilne, and Qiao Zhang.
\newblock {JAX}: composable transformations of {P}ython+{N}um{P}y programs,
  2018.

\bibitem{byrd1995limited}
Richard~H Byrd, Peihuang Lu, Jorge Nocedal, and Ciyou Zhu.
\newblock A limited memory algorithm for bound constrained optimization.
\newblock {\em SIAM Journal on scientific computing}, 16(5):1190--1208, 1995.

\bibitem{cartis2012adaptive}
Coralia Cartis, Nicholas~IM Gould, and Ph~L Toint.
\newblock An adaptive cubic regularization algorithm for nonconvex optimization
  with convex constraints and its function-evaluation complexity.
\newblock {\em IMA Journal of Numerical Analysis}, 32(4):1662--1695, 2012.

\bibitem{cartis2012complexity}
Coralia Cartis, Nicholas~IM Gould, and Ph~L Toint.
\newblock Complexity bounds for second-order optimality in unconstrained
  optimization.
\newblock {\em Journal of Complexity}, 28(1):93--108, 2012.

\bibitem{cartis2018worst}
Coralia Cartis, Nicholas~IM Gould, and Philippe~L Toint.
\newblock Worst-case evaluation complexity and optimality of second-order
  methods for nonconvex smooth optimization.
\newblock In {\em Proceedings of the International Congress of Mathematicians:
  Rio de Janeiro 2018}, pages 3711--3750. World Scientific, 2018.

\bibitem{conn2000trust}
Andrew~R Conn, Nicholas~IM Gould, and Philippe~L Toint.
\newblock {\em Trust region methods}.
\newblock SIAM, 2000.

\bibitem{criscitiello2019efficiently}
Chris Criscitiello and Nicolas Boumal.
\newblock Efficiently escaping saddle points on manifolds.
\newblock {\em arXiv preprint arXiv:1906.04321}, 2019.

\bibitem{criscitiello2020accelerated}
Chris Criscitiello and Nicolas Boumal.
\newblock An accelerated first-order method for non-convex optimization on
  manifolds.
\newblock {\em arXiv preprint arXiv:2008.02252}, 2020.

\bibitem{di1986exact}
Gianni Di~Pillo and Luigi Grippo.
\newblock An exact penalty function method with global convergence properties
  for nonlinear programming problems.
\newblock {\em Mathematical Programming}, 36(1):1--18, 1986.

\bibitem{edelman1998geometry}
Alan Edelman, Tom{\'a}s~A Arias, and Steven~T Smith.
\newblock The geometry of algorithms with orthogonality constraints.
\newblock {\em SIAM journal on Matrix Analysis and Applications},
  20(2):303--353, 1998.

\bibitem{estrin2020implementing}
Ron Estrin, Michael~P Friedlander, Dominique Orban, and Michael~A Saunders.
\newblock Implementing a smooth exact penalty function for equality-constrained
  nonlinear optimization.
\newblock {\em SIAM Journal on Scientific Computing}, 42(3):A1809--A1835, 2020.

\bibitem{fletcher1970class}
Roger Fletcher.
\newblock A class of methods for nonlinear programming with termination and
  convergence properties.
\newblock {\em Integer and nonlinear programming}, pages 157--173, 1970.

\bibitem{fletcher2002nonlinear}
Roger {Fletcher} and Sven {Leyffer}.
\newblock Nonlinear programming without a penalty function.
\newblock {\em Mathematical Programming}, 91(2):239--269, 2002.

\bibitem{gao2019parallelizable}
Bin Gao, Xin Liu, and Ya-xiang Yuan.
\newblock Parallelizable algorithms for optimization problems with
  orthogonality constraints.
\newblock {\em SIAM Journal on Scientific Computing}, 41(3):A1949--A1983, 2019.

\bibitem{gao2021riemannian}
Bin Gao, Nguyen~Thanh Son, P-A Absil, and Tatjana Stykel.
\newblock Riemannian optimization on the symplectic {S}tiefel manifold.
\newblock {\em SIAM Journal on Optimization}, 31(2):1546--1575, 2021.

\bibitem{ge2015escaping}
Rong Ge, Furong Huang, Chi Jin, and Yang Yuan.
\newblock Escaping from saddle points--online stochastic gradient for tensor
  decomposition.
\newblock In {\em Conference on learning theory}, pages 797--842. PMLR, 2015.

\bibitem{GolubMatrix}
Gene~H Golub and Charles~F Van~Loan.
\newblock {\em Matrix computations}.
\newblock JHU press, 2013.

\bibitem{gould1999solving}
Nicholas~IM Gould, Stefano Lucidi, Massimo Roma, and Philippe~L Toint.
\newblock Solving the trust-region subproblem using the {L}anczos method.
\newblock {\em SIAM Journal on Optimization}, 9(2):504--525, 1999.

\bibitem{grippo1986nonmonotone}
Luigi Grippo, Francesco Lampariello, and Stephano Lucidi.
\newblock A nonmonotone line search technique for {N}ewton’s method.
\newblock {\em SIAM Journal on Numerical Analysis}, 23(4):707--716, 1986.

\bibitem{numpy2020array}
Charles~R Harris, K~Jarrod Millman, St{\'e}fan~J van~der Walt, Ralf Gommers,
  Pauli Virtanen, David Cournapeau, Eric Wieser, Julian Taylor, Sebastian Berg,
  Nathaniel~J Smith, et~al.
\newblock Array programming with numpy.
\newblock {\em Nature}, 585(7825):357--362, 2020.

\bibitem{hestenes1952methods}
Magnus~R Hestenes and Eduard Stiefel.
\newblock Methods of conjugate gradients for solving linear systems.
\newblock {\em Journal of research of the National Bureau of Standards},
  49(6):409, 1952.

\bibitem{hosseini2015convergence}
S~Hosseini.
\newblock Convergence of nonsmooth descent methods via {K}urdyka--{L}ojasiewicz
  inequality on {R}iemannian manifolds.
\newblock {\em Hausdorff Center for Mathematics and Institute for Numerical
  Simulation, University of Bonn (2015,(INS Preprint No. 1523))}, 2015.

\bibitem{hu2020brief}
Jiang Hu, Xin Liu, Zai-Wen Wen, and Ya-Xiang Yuan.
\newblock A brief introduction to manifold optimization.
\newblock {\em Journal of the Operations Research Society of China},
  8(2):199--248, 2020.

\bibitem{hu2020anefficiency}
Xiaoyin {Hu} and Xin {Liu}.
\newblock An efficient orthonormalization-free approach for sparse dictionary
  learning and dual principal component pursuit.
\newblock {\em Sensors}, 20(3041), 2020.

\bibitem{jin2017escape}
Chi Jin, Rong Ge, Praneeth Netrapalli, Sham~M Kakade, and Michael~I Jordan.
\newblock How to escape saddle points efficiently.
\newblock In {\em International Conference on Machine Learning}, pages
  1724--1732. PMLR, 2017.

\bibitem{jin2018accelerated}
Chi Jin, Praneeth Netrapalli, and Michael~I Jordan.
\newblock Accelerated gradient descent escapes saddle points faster than
  gradient descent.
\newblock In {\em Conference On Learning Theory}, pages 1042--1085. PMLR, 2018.

\bibitem{geoopt2020kochurov}
Max Kochurov, Rasul Karimov, and Serge Kozlukov.
\newblock Geoopt: {R}iemannian optimization in pytorch.
\newblock Technical report, arXiv preprint arXiv:2005.02819, 2020.

\bibitem{lee2019first}
Jason~D Lee, Ioannis Panageas, Georgios Piliouras, Max Simchowitz, Michael~I
  Jordan, and Benjamin Recht.
\newblock First-order methods almost always avoid strict saddle points.
\newblock {\em Mathematical programming}, 176(1):311--337, 2019.

\bibitem{lei2019stochastic}
Yunwen Lei, Ting Hu, Guiying Li, and Ke~Tang.
\newblock Stochastic gradient descent for nonconvex learning without bounded
  gradient assumptions.
\newblock {\em IEEE transactions on neural networks and learning systems},
  31(10):4394--4400, 2019.

\bibitem{lenders2018trlib}
Felix Lenders, Christian Kirches, and Andreas Potschka.
\newblock trlib: A vector-free implementation of the gltr method for iterative
  solution of the trust region problem.
\newblock {\em Optimization Methods and Software}, 33(3):420--449, 2018.

\bibitem{li2015accelerated}
Huan Li and Zhouchen Lin.
\newblock Accelerated proximal gradient methods for nonconvex programming.
\newblock In {\em Advances in neural information processing systems}, pages
  379--387, 2015.

\bibitem{li2010inexact}
Lu~Li and Kim-Chuan Toh.
\newblock An inexact interior point method for l 1-regularized sparse
  covariance selection.
\newblock {\em Mathematical Programming Computation}, 2(3):291--315, 2010.

\bibitem{lojasiewicz1961probleme}
Stanis{\l}aw {\L}ojasiewicz.
\newblock Sur le probleme de la division.
\newblock 1961.

\bibitem{lojasiewicz1963propriete}
Stanislaw Lojasiewicz.
\newblock Une propri{\'e}t{\'e} topologique des sous-ensembles analytiques
  r{\'e}els.
\newblock {\em Les {\'e}quations aux d{\'e}riv{\'e}es partielles}, 117:87--89,
  1963.

\bibitem{maclaurin2015autograd}
Dougal Maclaurin, David Duvenaud, and Ryan~P Adams.
\newblock Autograd: Effortless gradients in numpy.
\newblock In {\em ICML 2015 AutoML workshop}, volume 238, page~5, 2015.

\bibitem{meghwanshi2018mctorch}
Mayank Meghwanshi, Pratik Jawanpuria, Anoop Kunchukuttan, Hiroyuki Kasai, and
  Bamdev Mishra.
\newblock Mctorch, a manifold optimization library for deep learning.
\newblock Technical report, arXiv preprint arXiv:1810.01811, 2018.

\bibitem{morales2011remark}
Jos{\'e}~Luis Morales and Jorge Nocedal.
\newblock Remark on “algorithm 778: {L}-{BFGS}-{B}: Fortran subroutines for
  large-scale bound constrained optimization”.
\newblock {\em ACM Transactions on Mathematical Software (TOMS)}, 38(1):1--4,
  2011.

\bibitem{nash1984newton}
Stephen~G Nash.
\newblock Newton-type minimization via the {L}anczos method.
\newblock {\em SIAM Journal on Numerical Analysis}, 21(4):770--788, 1984.

\bibitem{nesterov2006cubic}
Yurii Nesterov and Boris~T Polyak.
\newblock Cubic regularization of {N}ewton method and its global performance.
\newblock {\em Mathematical Programming}, 108(1):177--205, 2006.

\bibitem{nickel2018learning}
Maximillian Nickel and Douwe Kiela.
\newblock Learning continuous hierarchies in the {L}orentz model of hyperbolic
  geometry.
\newblock In {\em International Conference on Machine Learning}, pages
  3779--3788. PMLR, 2018.

\bibitem{nocedal2006numerical}
Jorge Nocedal and Stephen Wright.
\newblock {\em Numerical optimization}.
\newblock Springer Science \& Business Media, 2006.

\bibitem{ochs2014ipiano}
Peter Ochs, Yunjin Chen, Thomas Brox, and Thomas Pock.
\newblock ipiano: Inertial proximal algorithm for nonconvex optimization.
\newblock {\em SIAM Journal on Imaging Sciences}, 7(2):1388--1419, 2014.

\bibitem{powell1969method}
Michael~JD Powell.
\newblock A method for nonlinear constraints in minimization problems.
\newblock {\em Optimization}, pages 283--298, 1969.

\bibitem{qi2010riemannian}
Chunhong Qi, Kyle~A Gallivan, and P-A Absil.
\newblock {R}iemannian {BFGS} algorithm with applications.
\newblock In {\em Recent advances in optimization and its applications in
  engineering}, pages 183--192. Springer, 2010.

\bibitem{sato2016dai}
Hiroyuki Sato.
\newblock A {D}ai--{Y}uan-type {R}iemannian conjugate gradient method with the
  weak {W}olfe conditions.
\newblock {\em Computational optimization and Applications}, 64(1):101--118,
  2016.

\bibitem{scott1981solving}
David~S Scott.
\newblock Solving sparse symmetric generalized eigenvalue problems without
  factorization.
\newblock {\em SIAM Journal on Numerical Analysis}, 18(1):102--110, 1981.

\bibitem{siegel2019accelerated}
Jonathan~W Siegel.
\newblock Accelerated optimization with orthogonality constraints.
\newblock {\em arXiv preprint arXiv:1903.05204}, 2019.

\bibitem{son2021symplectic}
Nguyen~Thanh Son, P-A Absil, Bin Gao, and Tatjana Stykel.
\newblock Symplectic eigenvalue problem via trace minimization and {R}iemannian
  optimization.
\newblock {\em arXiv preprint arXiv:2101.02618}, 2021.

\bibitem{steihaug1983conjugate}
Trond Steihaug.
\newblock The conjugate gradient method and trust regions in large scale
  optimization.
\newblock {\em SIAM Journal on Numerical Analysis}, 20(3):626--637, 1983.

\bibitem{toint1981towards}
Philippe Toint.
\newblock Towards an efficient sparsity exploiting {N}ewton method for
  minimization.
\newblock In {\em Sparse matrices and their uses}, pages 57--88. Academic
  Press, 1981.

\bibitem{townsend2016pymanopt}
James Townsend, Niklas Koep, and Sebastian Weichwald.
\newblock Pymanopt: A python toolbox for optimization on manifolds using
  automatic differentiation.
\newblock {\em arXiv preprint arXiv:1603.03236}, 2016.

\bibitem{scipy2020SciPyNMeth}
Pauli Virtanen, Ralf Gommers, Travis~E Oliphant, Matt Haberland, Tyler Reddy,
  David Cournapeau, Evgeni Burovski, Pearu Peterson, Warren Weckesser, Jonathan
  Bright, et~al.
\newblock Scipy 1.0: fundamental algorithms for scientific computing in python.
\newblock {\em Nature methods}, 17(3):261--272, 2020.

\bibitem{wang2020multipliers}
Lei Wang, Bin Gao, and Xin Liu.
\newblock Multipliers correction methods for optimization problems over the
  {S}tiefel manifold.
\newblock {\em arXiv preprint arXiv:2011.14781}, 2020.

\bibitem{wang2020cubic}
Zhe Wang, Yi~Zhou, Yingbin Liang, and Guanghui Lan.
\newblock Cubic regularization with momentum for nonconvex optimization.
\newblock In {\em Uncertainty in Artificial Intelligence}, pages 313--322.
  PMLR, 2020.

\bibitem{wen2013feasible}
Zaiwen Wen and Wotao Yin.
\newblock A feasible method for optimization with orthogonality constraints.
\newblock {\em Mathematical Programming}, 142(1-2):397--434, 2013.

\bibitem{wu2010critical}
R-B Wu, Raj Chakrabarti, and Herschel Rabitz.
\newblock Critical landscape topology for optimization on the symplectic group.
\newblock {\em Journal of optimization theory and applications},
  145(2):387--406, 2010.

\bibitem{xiao2021solving}
Nachuan Xiao and Xin Liu.
\newblock Solving optimization problems over the {S}tiefel manifold by smooth
  exact penalty function.
\newblock {\em arXiv preprint arXiv:2110.08986}, 2021.

\bibitem{xiao2020class}
Nachuan Xiao, Xin Liu, and Ya-xiang Yuan.
\newblock A class of smooth exact penalty function methods for optimization
  problems with orthogonality constraints.
\newblock {\em Optimization Methods and Software}, pages 1--37, 2020.

\bibitem{xiao2020l21}
Nachuan Xiao, Xin Liu, and Ya-xiang Yuan.
\newblock Exact penalty function for $\ell_{2,1}$ norm minimization over the
  {S}tiefel manifold.
\newblock {\em SIAM Journal on Optimization}, 31(4):3097--3126, 2021.

\bibitem{xiao2021penalty}
Nachuan Xiao, Xin Liu, and Ya-xiang Yuan.
\newblock A penalty-free infeasible approach for a class of nonsmooth
  opimtization problems over the {S}tiefel manifold.
\newblock {\em arXiv preprint arXiv:2103.03514}, 2021.

\bibitem{yuan2015recent}
Ya-xiang Yuan.
\newblock Recent advances in trust region algorithms.
\newblock {\em Mathematical Programming}, 151(1):249--281, 2015.

\bibitem{zavala2014scalable}
Victor~M Zavala and Mihai Anitescu.
\newblock Scalable nonlinear programming via exact differentiable penalty
  functions and trust-region {N}ewton methods.
\newblock {\em SIAM Journal on Optimization}, 24(1):528--558, 2014.

\bibitem{zhang2020riemannian}
Erchuan Zhang and Lyle Noakes.
\newblock Riemannian cubics in quadratic matrix {L}ie groups.
\newblock {\em Applied Mathematics and Computation}, 375:125082, 2020.

\bibitem{zhang2018towards}
Hongyi Zhang and Suvrit Sra.
\newblock Towards {R}iemannian accelerated gradient methods.
\newblock {\em arXiv preprint arXiv:1806.02812}, 2018.

\bibitem{zhang2018r}
Jingzhao Zhang, Hongyi Zhang, and Suvrit Sra.
\newblock R-spider: A fast {R}iemannian stochastic optimization algorithm with
  curvature independent rate.
\newblock {\em arXiv preprint arXiv:1811.04194}, 2018.

\bibitem{zhu1997algorithm}
Ciyou Zhu, Richard~H Byrd, Peihuang Lu, and Jorge Nocedal.
\newblock Algorithm 778: {L}-{BFGS}-{B}: Fortran subroutines for large-scale
  bound-constrained optimization.
\newblock {\em ACM Transactions on mathematical software (TOMS)},
  23(4):550--560, 1997.

\end{thebibliography}
	
\end{document}